\documentclass{amsart}
\usepackage{amsfonts}
\usepackage{amsmath}
\usepackage{amssymb}
\usepackage{graphicx}%
\providecommand{\U}[1]{\protect \rule{.1in}{.1in}}
\theoremstyle{plain}
\newtheorem{theorem}{Theorem}[section]

\newtheorem{corollary}[theorem]{Corollary}

\newtheorem{definition}[theorem]{Definition}
\newtheorem{example}[theorem]{Example}

\newtheorem{lemma}[theorem]{Lemma}

\newtheorem{proposition}[theorem]{Proposition}
\newtheorem{remark}[theorem]{Remark}

 \numberwithin{equation}{section}
\begin{document}
\title[Symmetric Norms]{LEBESGUE AND HARDY SPACES FOR SYMMETRIC NORMS I}
\author{Yanni Chen}
\address{Department of Mathematics, University of New Hampshire, Durham, NH 03824, U.S.A.}
\email{yet2@wildcats.unh.edu}


\begin{abstract}
In this paper, we{ define and study a class $\mathcal{R}_{c}$ of norms on
}$L^{\infty}\left(  \mathbb{T}\right)  $, called \emph{continuous rotationally
symmetric norms}, which properly contains the class $\left \{  \left \Vert
\cdot \right \Vert _{p}:1\leq p<\infty \right \}  .$ For $\alpha \in \mathcal{R}%
_{c}$ we define $L^{\alpha}\left(  \mathbb{T}\right)  $ and the Hardy space
$H^{\alpha}\left(  \mathbb{T}\right)  $, and we extend many of the classical
results, including the dominated convergence theorem, convolution theorems,
dual spaces, Beurling-type invariant spaces, inner-outer factorizations,
characterizing the multipliers and the closed densely-defined operators
commuting with multiplication by $z$. We also prove a duality theorem for a
version of $L^{\alpha}$ in the setting of von Neumann algebras.

\end{abstract}
\maketitle

\bigskip

\section{Introduction}

Suppose $\mathbb{T}$ is the unit circle in the complex plane $\mathbb{C}$ and
$m$ is Haar measure (i.e., normalized arc length) on $\mathbb{T}$. We let
$\mathbb{R},$ $\mathbb{Z}$, $\mathbb{N}$, respectively, denote the sets of
real numbers, integers, and positive integers. In this paper we focus on norms
on $L^{\infty}\left(  \mathbb{T}\right)  .$ In particular we focus on norms
that are rotationally symmetric (defined below) and we define Lebesgue and
Hardy spaces with respect to these norms and extend many classical results
concerning $L^{p}$ and $H^{p}$ to this setting. Since the family $L^{p}\left(
\mathbb{T}\right)  $ (and $H^{p}\left(  \mathbb{T}\right)  $) spaces are
linearly ordered by inclusion, many classical proofs are broken into the cases
$1\leq p\leq2$ and $2<p<\infty$. In our setting this dichotomy does not exist,
requiring new techniques. In subsequent papers we prove a version of
Beurling's theorem in a vector-valued analogue of $H^{p}\left(  \mathbb{T}%
\right)  $ (see \cite{RTS}), we extend results concerning noncommutative
$H^{p}$-spaces (see \cite{Sai}) , and we extend our results to general finite
measure spaces with norms symmetric with respect to a group of
measure-preserving transformations, we extend our results to analogues of
$H^{p}$-spaces on nice multiply connected domains, and in a final paper we
extend our results to $\sigma$-finite measure spaces and corresponding
commutative and noncommutative analogues of $H^{p}$-spaces.

\bigskip

We say that a seminorm $\alpha$ on $L^{\infty}\left(  \mathbb{T}\right)  $ is
a \emph{gauge norm} if

\begin{enumerate}
\item $\alpha(1)=1,$

\item $\alpha(|f|)=\alpha(f)$ for every $f\in L^{\infty}\left(  \mathbb{T}%
\right)  $.
\end{enumerate}

\bigskip

We say that a seminorm $\alpha$ on $L^{\infty}\left(  \mathbb{T}\right)  $ is
\emph{rotationally symmetric} if it is a gauge norm and

\ (3) $\alpha \left(  f_{w}\right)  =\alpha \left(  f\right)  $ for every
$w\in \mathbb{T}$ , $f\in L^{\infty}\left(  \mathbb{T}\right)  ,$

where
\[
f_{w}\left(  z\right)  =f\left(  \overline{w}z\right)  ,
\]
whenever $f:\mathbb{T}\rightarrow \mathbb{C}$ and $w\in \mathbb{T}$.

For any measurable $f:\mathbb{T}\rightarrow \mathbb{C},$ we define
$\alpha \left(  f\right)  $ by%
\[
\alpha(f)=\sup \{ \alpha(s):s\text{ is a simple function},|s|\leq|f|\}.
\]
And we define the Banach space $\mathcal{L}^{\alpha}\left(  \mathbb{T}\right)
$ to be the set of all measurable functions $f:\mathbb{T}\rightarrow
\mathbb{C}$ such that $\alpha \left(  f\right)  <\infty$, and define the
(sometimes proper) closed linear subspace $L^{\alpha}\left(  \mathbb{T}%
\right)  $ to be the $\alpha$-closure of $L^{\infty}\left(  \mathbb{T}\right)
$. For later use we define $H^{\alpha}\left(  \mathbb{T}\right)  $ to be the
$\alpha$-closed linear span of $\left \{  1,z,z^{2},\ldots \right \}  $, which is
a closed subspace of $L^{\alpha}\left(  \mathbb{T}\right)  $.

We define a seminorm $\alpha$ to be \emph{continuous} if,
\[
m\left(  E_{n}\right)  \rightarrow0\Rightarrow \alpha \left(  \chi_{E_{n}%
}\right)  \rightarrow0,
\]
and if, in addition, $L^{\alpha}\left(  \mathbb{T}\right)  =\mathcal{L}%
^{\alpha}\left(  \mathbb{T}\right)  $, we say that $\alpha$ is \emph{strongly
continuous}.

Note that, for $1\leq p\leq \infty,$ $\left \Vert \cdot \right \Vert _{p}$ is a
rotationally symmetric norm that is strongly continuous when $p<\infty$. Thus
the spaces $L^{\alpha}\left(  \mathbb{T}\right)  $ and $H^{\alpha}\left(
\mathbb{T}\right)  $ are generalizations of the classical Lebesgue spaces
$L^{p}\left(  \mathbb{T}\right)  $ and the Hardy spaces $H^{p}\left(
\mathbb{T}\right)  $. In this paper we extend many of the classical results,
often using new techniques, to these more general spaces, including a
dominated convergence theorem, convolution theorems, invariant subspace
theorems, dual spaces.

More precisely, in Section 2 we describe some of the basic properties and
constructions of these seminorms, and show that in many cases (e.g., when they
are continuous), $\left \Vert \cdot \right \Vert _{1}\leq \alpha \leq \left \Vert
\cdot \right \Vert _{\infty}$, and that $C\left(  \mathbb{T}\right)  $ is dense
in $L^{\alpha}\left(  \mathbb{T}\right)  $.

In Section 3 we define the dual seminorm $\alpha^{\prime}$ for each continuous
rotationally symmetric seminorm $\alpha$. We show that dual space of
$L^{\alpha}\left(  \mathbb{T}\right)  $ is $\mathcal{L}^{\alpha^{\prime}%
}\left(  \mathbb{T}\right)  $. We also show that $\alpha^{\prime \prime}%
=\alpha$ holds and we show that $L^{\alpha}\left(  \mathbb{T}\right)  $ (and
therefore $H^{\alpha}\left(  \mathbb{T}\right)  $) is reflexive whenever
$\alpha$ and $\alpha^{\prime}$ are both strongly continuous. Furthermore, we
show that a continuous $\alpha$ is strongly continuous if and only if
$L^{\alpha}\left(  \mathbb{T}\right)  $ is weakly sequentially complete.

In Section 4 we prove a general continuity theorem and a dominated convergence
theorem for $L^{\alpha}\left(  \mathbb{T}\right)  $, and we show that each of
these theorems characterize $L^{\alpha}\left(  \mathbb{T}\right)  $ in
$\mathcal{L}^{\alpha}\left(  \mathbb{T}\right)  $. We also give an example of
a continuous norm $\alpha$ with $\mathcal{L}^{\alpha}\left(  \mathbb{T}%
\right)  \neq L^{\alpha}\left(  \mathbb{T}\right)  .$

In Section 5 we define $L^{\alpha}\left(  \mathbb{T},X\right)  $ for a
separable Banach space $X$, and we discuss properties of the convolution
$f\ast g$ with $f\in L^{\alpha}\left(  \mathbb{T},X\right)  $ and $g\in
L^{1}\left(  \mathbb{T}\right)  $ using the fact that the set of continuous
functions from $\mathbb{T}$ to $X$ is dense in $L^{\alpha}\left(
\mathbb{T},X\right)  $ when $\alpha$ is continuous. We use convolutions with
the Fejer kernel to show that when $\alpha$ is continuous and $f\in L^{\alpha
}\left(  \mathbb{T},X\right)  ,$ then $f$ is the $\alpha$-limit of the Cesaro
sums of the sequence of partial sums of its Fourier series, and we use the
Poisson kernel to extend $f$ to a function from $\mathbb{D}=\left \{
z\in \mathbb{C}:\left \vert z\right \vert \leq1\right \}  $ to $X$.

In Section 6 we study $H^{\alpha}\left(  \mathbb{T}\right)  $ when $\alpha$ is
continuous and we show that, as in the classical setting, $H^{\alpha}\left(
\mathbb{T}\right)  $ is isometrically isomorphic to $H^{\alpha}\left(
\mathbb{D}\right)  $. We characterize $H^{\alpha}\left(  \mathbb{T}\right)  $
as the functions in $L^{\alpha}\left(  \mathbb{T}\right)  $ whose negative
Fourier coefficients vanish and we show $H^{\alpha}\left(  \mathbb{T}\right)
=H^{1}\left(  \mathbb{T}\right)  \cap L^{\alpha}\left(  \mathbb{T}\right)  $.
Similarly, we define $\mathcal{H}^{\alpha}\left(  \mathbb{T}\right)
=H^{1}\left(  \mathbb{T}\right)  \cap \mathcal{L}^{\alpha}\left(
\mathbb{T}\right)  ~$to be the functions in $\mathcal{L}^{\alpha}\left(
\mathbb{T}\right)  $ whose negative Fourier coefficients vanish. We show that,
for a separable Banach space $X$
\[
L^{\alpha}\left(  \mathbb{T},X\right)  =sp^{-\alpha}\left(  \left \{
z^{n}\cdot x:n\in \mathbb{Z}\text{ and }x\in X\right \}  \right)
\]
and
\[
H^{\alpha}\left(  \mathbb{T},X\right)  =sp^{-\alpha}\left(  \left \{
z^{n}\cdot x:n\geq0\text{ and }x\in X\right \}  \right)  .
\]
\newline Only when $\alpha$ is strongly continuous do we get that $f\in
H^{\alpha}\left(  \mathbb{D}\right)  $ if and only if $f$ is analytic and
$\sup \left \{  \alpha \left(  f\left(  re^{it}\right)  \right)  :0<r<1\right \}
<\infty$.

In Section 7 we prove analogues of Beurling's invariant subspace theorems for
$H^{\alpha}\left(  \mathbb{T}\right)  $ when $\alpha$ is continuous. In
Section 8, we show that the Riesz-Smirnov inner-outer factorization works in
$H^{\alpha}\left(  \mathbb{T}\right)  $. In Section 9 we characterize the
multipliers of $L^{\alpha}\left(  \mathbb{T}\right)  $ and $H^{\alpha}\left(
\mathbb{T}\right)  $ and show that they are multiplier pairs in the sense of
\cite{HN1}$;~$we also prove that $\mathcal{L}^{\alpha}\left(  \mathbb{T}%
\right)  =L^{\alpha}\left(  \mathbb{T}\right)  $ if and only if $\mathcal{H}%
^{\alpha}\left(  \mathbb{T}\right)  =H^{\alpha}\left(  \mathbb{T}\right)  .$
In Section 10 we characterize the closed densely defined operators on
$H^{\alpha}\left(  \mathbb{T}\right)  $ that commute with multiplication by
$z$. In the final section we state and prove a corrected version of a theorem
in \cite{FHNS} concerning the dual space of a von Neumann algebra version of
an $L^{\alpha}$-space.

\section{Completing Norms on $L^{\infty}\left(  \mathbb{T}\right)  $}

We begin by defining two classes of seminorms on $L^{\infty}\left(
\mathbb{T}\right)  $. Suppose $f:\mathbb{T}\rightarrow \mathbb{C}$ and
$w\in \mathbb{T}$. We define $f_{w}:\mathbb{T}\rightarrow \mathbb{C}$ by%
\[
f_{w}\left(  z\right)  =f\left(  \overline{w}z\right)  .
\]

We say that a seminorm $\alpha$ on $L^{\infty}(\mathbb{T})$ is
\emph{rotationally\ symmetric} if

\begin{enumerate}
\item $\alpha(1)=1,$

\item $\alpha(|f|)=\alpha(f)$ for every $f\in L^{\infty}\left(  \mathbb{T}%
\right)  $,

\item $\alpha \left(  f_{w}\right)  =\alpha \left(  f\right)  $ for every
$w\in \mathbb{T}$ , $f\in L^{\infty}\left(  \mathbb{T}\right)  .$
\end{enumerate}

\bigskip

To define the smaller class of symmetric gauge norms, we need to define
$MP\left(  \mathbb{T}\right)  $ to be the set of invertible measure-preserving
maps $\phi:\mathbb{T}\rightarrow \mathbb{T}$. A seminorm $\beta$ on $L^{\infty
}(\mathbb{T})$ is a$\ $\emph{symmetric\ gauge seminorm} if

\begin{enumerate}
\item $\beta(1)=1,$

\item $\beta(|f|)=\beta(f)$ for every $f\in L^{\infty}(\mathbb{T}),$

\item $\beta(f\circ \phi)=\beta(f)$ for every $f\in L^{\infty}(\mathbb{T})$ and
$\phi \in MP\left(  \mathbb{T}\right)  .$
\end{enumerate}

We often identify $\mathbb{T}$ with $(0,1]$ (identifying $e^{2\pi it}$ with
$t$) and $m$ with Lebesgue measure on $(0,1]$. We will make it clear whenever
we do this. One place where this is convenient is when we want to talk about
increasing or decreasing functions. One particular fact that makes our work
more easily understood is the following lemma from \cite[Theorem
3.4.1]{Lorentz}.

\begin{lemma}
Suppose $f:(0,1]\rightarrow \lbrack0,\infty)$ is measurable. Then there is a
unique non-increasing right-continuous function $f^{\bigstar}$ on $(0,1]$ and
an invertible measure-preserving map $\phi:(0,1]\rightarrow(0,1]$ such that
$f^{\bigstar}=f\circ \phi$ a.e. $\left(  m\right)  $.\bigskip
\end{lemma}

It follows from the definition that if $\alpha$ is a symmetric gauge norm,
then $\alpha \left(  f\right)  =\alpha \left(  \left \vert f\right \vert \right)
=\alpha \left(  \left \vert f\right \vert ^{\bigstar}\right)  .$ Note that if
$f=\chi_{E}$ for $E\subset(0,1]$ with $m\left(  E\right)  >0$, then
$f^{\bigstar}=\chi_{\left(  0,m\left(  E\right)  \right)  }$, which in
$\mathbb{T}$ we would represent as $f^{\bigstar}=\chi_{I_{m\left(  E\right)
}}$, where%
\[
I_{m\left(  E\right)  }=\left \{  e^{it}:0\leq t\leq2\pi m\left(  E\right)
\right \}  .
\]
Therefore if $\beta$ is a symmetric gauge norm, then $\beta \left(  \chi
_{E}\right)  =\beta \left(  \chi_{F}\right)  $ whenever $m\left(  E\right)
=m\left(  F\right)  $. More generally, $m\left(  E\right)  \leq m\left(
F\right)  $ implies $\chi_{I_{m\left(  E\right)  }}\leq \chi_{I_{m\left(
F\right)  }},$ which implies $\beta \left(  \chi_{E}\right)  \leq \beta \left(
\chi_{F}\right)  ,$ i.e., $\beta \left(  \chi_{E}\right)  $ is increasing with
respect to $m\left(  E\right)  $. Hence, we conclude that $\lim
_{m(E)\rightarrow0^{+}}\beta(\chi_{E})$ exists.\bigskip

We say that a rotationally symmetric seminorm $\alpha$ on $L^{\infty
}(\mathbb{T})$ is $continuous$ if
\[
\lim_{m(E)\rightarrow0^{+}}\alpha(\chi_{E})=0.
\]
In general, for a rotationally symmetric norm $\alpha,$ we do not know
$\lim_{m(E)\rightarrow0^{+}}\alpha(\chi_{E})$ exists.

Let $\mathcal{R}$ denote the set of all rotationally symmetric seminorms on
$L^{\infty}\left(  \mathbb{T}\right)  $, and let $\mathcal{S}$ denote the set
of all symmetric gauge seminorms on $L^{\infty}\left(  \mathbb{T}\right)  .$
Clearly, $\mathcal{S}\subset \mathcal{R}$. We let $\mathcal{R}_{c}$ and
$\mathcal{S}_{c}$ denote, respectively, the continuous seminorms in
$\mathcal{R}$ and $\mathcal{S}$. In most of this paper we will be considering
elements of $\mathcal{R}_{c}$.

Although $\alpha \in \mathcal{R}$ is defined only on $L^{\infty}(\mathbb{T}),$
we can define $\alpha(f)$ for all measurable functions $f$ on $\mathbb{T}$ by
\[
\alpha(f)=\sup \{ \alpha(s):s\text{ is a simple function},|s|\leq|f|\}.
\]
It is clear that $\alpha \left(  f\right)  =\alpha \left(  \left \vert
f\right \vert \right)  $ still holds.

We define $L^{\alpha}(\mathbb{T})$ to be the completion of $L^{\infty
}(\mathbb{T})$ with respect to $\alpha,$ and
\[
{\mathcal{L}}^{\alpha}(\mathbb{T})=\{f:\alpha(f)<\infty \}.
\]

If $\alpha \in \mathcal{R}$, we say that $\alpha$ is \emph{strongly continuous}
if and only if $\alpha \in \mathcal{R}_{c}$ and $L^{\alpha}(\mathbb{T}%
)={\mathcal{L}}^{\alpha}(\mathbb{T})$.

\begin{proposition}
\label{prop1} Suppose $\alpha \in \mathcal{R}$ and $f,g:\mathbb{T}%
\rightarrow \mathbb{C}$ are measurable. The following are true.

\begin{enumerate}
\item For all measurable functions $f,g$ on $\mathbb{T}$,

\begin{enumerate}
\item $|f|\leq|g|\Longrightarrow \alpha(f)\leq \alpha(g),$

\item $\alpha(fg)\leq \alpha(f)\Vert g\Vert_{\infty},$ and

\item $\alpha(f)\leq \Vert f\Vert_{\infty}$;
\end{enumerate}

\item $\alpha \in{\mathcal{R}}_{c}\  \Longleftrightarrow \limsup_{m\left(
E\right)  \rightarrow0^{+}}\alpha(\chi_{E})=0$;

\item $\liminf_{m\left(  E\right)  \rightarrow0^{+}}\alpha \left(  \chi
_{E}\right)  =\inf \left \{  \alpha \left(  \chi_{E}\right)  :m\left(  E\right)
>0\right \}  $;

\item If $t>0$, then
\[
\liminf_{m\left(  E\right)  \rightarrow0^{+}}\alpha(\chi_{E})\geq
t\  \Longleftrightarrow \ t\Vert f\Vert_{\infty}\leq \alpha(f)\leq \left \Vert
f\right \Vert _{\infty};
\]

\item If $\alpha \in{\mathcal{S}}$, then $\alpha \notin{\mathcal{S}}_{c}$
$\Longleftrightarrow \alpha$ is equivalent to $\Vert \Vert_{\infty}$;

\item \label{part1} If $\alpha \in \mathcal{S}\cup \mathcal{R}_{c}$, $0\leq
f_{1}\leq f_{2}\leq \cdots$ and $f_{n}\rightarrow f$ a.e. $\left(  m\right)  $,
then $\alpha \left(  f_{n}\right)  \rightarrow \alpha \left(  f\right)  ;$

\item If $\alpha \in \mathcal{R}_{c}$, then $C\left(  \mathbb{T}\right)
^{-\alpha}=L^{\alpha}\left(  \mathbb{T}\right)  $;

\item If $\alpha \in \mathcal{S}\cup \mathcal{R}_{c}$, then $\Vert f\Vert_{1}%
\leq \alpha(f);$

\item If $\alpha \in{\mathcal{S}}\cup{\mathcal{R}}_{c}$ and $\lambda
\in \mathbb{C}$, then

\begin{enumerate}
\item $\alpha \left(  f+g\right)  \leq \alpha \left(  f\right)  +\alpha \left(
g\right)  ,$

\item $\alpha \left(  \lambda f\right)  =\left \vert \lambda \right \vert
\alpha \left(  f\right)  ,$ and

\item $\alpha$ is a norm on $\mathcal{L}^{\alpha}\left(  \mathbb{T}\right)  $;
\end{enumerate}

\item If $\alpha \in \mathcal{R}_{c}$, then $\left(  \mathcal{L}^{\alpha}\left(
\mathbb{T}\right)  ,\alpha \right)  $ is a Banach space and
\[
L^{\infty}\left(  \mathbb{T}\right)  \subset L^{\alpha}\left(  \mathbb{T}%
\right)  \subset \mathcal{L}^{\alpha}\left(  \mathbb{T}\right)  \subset
L^{1}\left(  \mathbb{T}\right)  \text{.}%
\]

\end{enumerate}
\end{proposition}

\begin{proof}
$(1)$ (a) If $\left \vert f\right \vert \leq \left \vert g\right \vert $, then
there are two measurable functions $u,v$ with $\left \vert u\right \vert
=\left \vert v\right \vert =1$ and $f=g\left(  u+v\right)  /2$, which implies
\[
\alpha \left(  f\right)  \leq \left[  \alpha \left(  \left \vert ug\right \vert
\right)  +\alpha \left(  \left \vert vg\right \vert \right)  \right]
/2=\alpha \left(  \left \vert g\right \vert \right)  =\alpha \left(  g\right)  .
\]

(b) Since $|fg|\leq \Vert f\Vert_{\infty}|g|$ a.e. $\left(  m\right)  ,$ it
follows from part (a) that
\[
\alpha(fg)=\alpha(|fg|)\leq \alpha(\Vert f\Vert_{\infty}|g|)=\Vert
f\Vert_{\infty}\alpha(g).
\]

(c) Since $\alpha(1)=1,$ we know from part (b) that
\[
\alpha(f)=\alpha(f\cdot1)\leq \alpha(1)\Vert f\Vert_{\infty}=\Vert
f\Vert_{\infty}.
\]

$(2) $ If $\alpha \in{\mathcal{R}}_{c},$ then $\lim_{m(E)\rightarrow0^{+}%
}\alpha(\chi_{E})=0,$ which implies that
\[
\limsup_{m(E)\rightarrow0 ^{+}} \alpha(\chi_{E})=\lim_{m(E)\rightarrow0^{+}%
}\alpha(\chi_{E})=0.
\]
On the other hand, if $\limsup_{m(E)\rightarrow0^{+}} \alpha(\chi_{E})=0,$ it
follows from
\[
0\leq \liminf_{m(E)\rightarrow0^{+}} \alpha(\chi_{E})\leq \limsup
_{m(E)\rightarrow0^{+}} \alpha(\chi_{E})=0
\]
that $\lim_{m(E)\rightarrow0^{+}}\alpha(\chi_{E})=0,$ which means $\alpha
\in{\mathcal{R}}_{c}.$

$(3)$ It is clear that
\[
\liminf_{m\left(  E\right)  \rightarrow0^{+}}\alpha \left(  \chi_{E}\right)
=\sup_{r>0}\inf \left \{  \alpha \left(  \chi_{E}\right)  :0<m\left(  E\right)
<r\right \}  ,
\]
and $\inf \left \{  \alpha \left(  \chi_{E}\right)  :0<m\left(  E\right)
<r\right \}  $ is decreasing. Then there is a decreasing sequence $\{E_{n}\}$
in $\mathbb{T}$ such that when $0<m(E_{n})<r$ for all $n\geq1,$ we have
\[
\alpha(\chi_{E_{n}})\rightarrow \inf \left \{  \alpha \left(  \chi_{E}\right)
:0<m\left(  E\right)  <r\right \}  .
\]
Suppose $0<m(E_{n})<r$ and $r_{1}<r.$ Then there is an $F_{n}\subset E_{n}$
such that $0<m(F_{n})<r_{1},$ which yields $\alpha(\chi_{F_{n}})\leq
\alpha(\chi_{E_{n}}).$ Hence
\begin{align*}
\inf \left \{  \alpha \left(  \chi_{E}\right)  :0<m\left(  E\right)
<r_{1}\right \}   &  \leq \inf \left \{  \alpha \left(  \chi_{F_{n}}\right)
:0<m\left(  F_{n}\right)  <r_{1}\right \} \\
&  \leq \inf \left \{  \alpha \left(  \chi_{E_{n}}\right)  :0<m\left(
E_{n}\right)  <r\right \} \\
&  =\lim_{n\rightarrow \infty}\alpha(\chi_{E_{n}})\\
&  =\inf \left \{  \alpha \left(  \chi_{E}\right)  :0<m\left(  E\right)
<r\right \}  ,
\end{align*}
since $\inf \left \{  \alpha \left(  \chi_{E}\right)  :0<m\left(  E\right)
<r\right \}  $ is decreasing, which implies that
\[
\liminf_{m\left(  E\right)  \rightarrow0^{+}}\alpha \left(  \chi_{E}\right)
=\inf \left \{  \alpha \left(  \chi_{E}\right)  :m\left(  E\right)  >0\right \}
.
\]

$(4)$ Suppose $\lim \inf_{m\left(  E\right)  \rightarrow0^{+}}\alpha(\chi
_{E})\geq t.$ It follows from part (3) that \newline$\inf \left \{
\alpha \left(  \chi_{E}\right)  :m\left(  E\right)  >0\right \}  \geq t.$ Then
for any measurable subset $E\subset \mathbb{T}$ with $m(E)>0,$ $\alpha(\chi
_{E})\geq t.$ If $F=\{x\in{\mathbb{T}}:|f(x)|\geq \Vert f\Vert_{\infty}\},$
then $m(F)>0$ and $\alpha(\chi_{F})\geq t.$ Hence by part (1) we have
\[
\Vert f\Vert_{\infty}\geq \alpha(f)\geq \alpha(f\chi_{F})\geq \alpha(\Vert
f\Vert_{\infty}\chi_{F})=\Vert f\Vert_{\infty}\alpha(\chi_{F})\geq t\Vert
f\Vert_{\infty}.
\]
Conversely, if $\ t\Vert f\Vert_{\infty}\leq \alpha(f)\leq \left \Vert
f\right \Vert _{\infty},$ then for every measurable set $E\subset \mathbb{T},$
$\alpha(\chi_{E})\geq t\Vert \chi_{E}\Vert_{\infty}=t,$ thus $\lim
\inf_{m\left(  E\right)  \rightarrow0^{+}}\alpha(\chi_{E})\geq t.$

$(5)$ It was shown in \cite{HN2}.

$(6)$ The case $\alpha \in \mathcal{S}$ was proved in \cite{FHNS}. We can assume
$\alpha \in \mathcal{R}_{c}$. Suppose $0\leq s\leq f$ and $0\leq t<1.$ Write
$s=\sum_{1\leq k\leq m}a_{k}\chi_{E_{k}}$ with $0<a_{k}$ for $1\leq k\leq m$
and $\left \{  E_{1},\ldots,E_{m}\right \}  $ disjoint. If we let $E_{k,n}%
=\left \{  \omega \in E_{k}:ta_{k}<f_{n}\left(  \omega \right)  \right \}  ,$ we
see that%
\[
E_{k,1}\subset E_{k,2}\subset \cdots \text{ and}\cup_{1\leq n<\infty}%
E_{k,n}=E_{k}.
\]
Since $\alpha$ is continuous,
\[
\alpha \left(  \chi_{E_{k}}-\chi_{E_{k,n}}\right)  =\alpha \left(  \chi
_{E_{k}\backslash E_{k,n}}\right)  \rightarrow0.
\]
Hence%
\[
t\alpha \left(  s\right)  =\lim_{n\rightarrow \infty}\alpha \left(  \sum
_{k=1}^{m}ta_{k}\chi_{E_{k,n}}\right)  \leq \lim_{n\rightarrow \infty}%
\alpha \left(  f_{n}\right)  .
\]
Since $t$ was arbitrary, for every simple function $s$ with $0\leq s\leq f,$
we have%
\[
\alpha \left(  s\right)  \leq \lim_{n\rightarrow \infty}\alpha \left(
f_{n}\right)  .
\]
By the definition of $\alpha \left(  f\right)  $, we see that $\alpha \left(
f\right)  \leq \lim_{n\rightarrow \infty}\alpha \left(  f_{n}\right)  $, and
$\alpha \left(  f_{n}\right)  \leq \alpha \left(  f\right)  $ for each $n\geq1$
follows from $\left(  1\right)  $.

$(7)$ Suppose $f\in L^{\infty}({\mathbb{T}}).$ By Lusin's theorem, there is a
sequence $\{g_{n}\}$ in $C({\mathbb{T}})$ such that if $E_{n}=\left \{
z\in \mathbb{T}:g_{n}\left(  z\right)  -f\left(  z\right)  \neq0\right \}  $,
then $m\left(  E_{n}\right)  \rightarrow0$, and such that $\Vert g_{n}%
\Vert_{\infty}\leq \Vert f\Vert_{\infty}<\infty$ for all $n\geq1.$ Since
$\alpha$ is continuous, $\alpha \left(  \chi_{E_{n}}\right)  \rightarrow0.$ But%
\[
\alpha(g_{n}-f)=\alpha(\left(  g_{n}-f\right)  \chi_{E_{n}})\leq2\left \Vert
f\right \Vert _{\infty}\alpha \left(  \chi_{E_{n}}\right)  \rightarrow0.
\]
Thus $L^{\infty}({\mathbb{T}})\subset C\left(  \mathbb{T}\right)  ^{-\alpha}$;
hence, $L^{\alpha}({\mathbb{T}})=L^{\infty}\left(  \mathbb{T}\right)
^{-\alpha}\subset C\left(  \mathbb{T}\right)  ^{-\alpha}\subset L^{\alpha
}({\mathbb{T}}).$ Therefore $L^{\alpha}({\mathbb{T}})=C\left(  \mathbb{T}%
\right)  ^{-\alpha}.$

$(8)$ If $\alpha \in \mathcal{S},$ then it is true \cite{HN2}. We can assume
$\alpha \in \mathcal{R}_{c}$.

It is well-known that if $f\in C\left(  \mathbb{T}\right)  $ and
$\omega=e^{2\pi i\theta}$ with $\theta$ irrational, then%
\[
\lim_{N\rightarrow \infty}\left \Vert \frac{1}{N}\sum_{k=1}^{N}\left \vert
f_{w^{k}}\right \vert -\left \Vert f\right \Vert _{1}\cdot1\right \Vert _{\infty
}=0,
\]
which, by part $\left(  1\right)  $, implies%
\[
\left \Vert f\right \Vert _{1}=\alpha \left(  \left \Vert f\right \Vert _{1}%
\cdot1\right)  =\lim_{N\rightarrow \infty}\alpha \left(  \frac{1}{N}\sum
_{k=1}^{N}\left \vert f_{w^{k}}\right \vert \right)  \leq \alpha \left(  f\right)
,
\]
since $\alpha$ is rotationally invariant.

Next suppose $f\in L^{\infty}\left(  \mathbb{T}\right)  $. If we choose the
$g_{n}$'s as in the proof of part $\left(  7\right)  $ (replacing $\alpha$
with $\left \Vert \cdot \right \Vert _{1}$), we get $\left \Vert g_{n}%
-f\right \Vert _{1}\rightarrow0$. Hence%
\[
\left \Vert f\right \Vert _{1}=\lim_{n\rightarrow \infty}\left \Vert
g_{n}\right \Vert _{1}\leq \lim_{n\rightarrow \infty}\alpha \left(  g_{n}\right)
=\alpha \left(  f\right)  .
\]

Finally, suppose $f:\mathbb{T}\rightarrow \mathbb{C}$ is measurable. We can
choose a sequence $\left \{  s_{n}\right \}  $ of simple functions such that
$0\leq s_{1}\leq s_{2}\leq \cdots$ and $s_{n}\left(  \omega \right)
\rightarrow \left \vert f\left(  \omega \right)  \right \vert $ for every
$\omega \in \mathbb{T}$. It follows from part $\left(  6\right)  $ and the
monotone convergence theorem that
\[
\alpha \left(  f\right)  =\lim_{n\rightarrow \infty}\alpha \left(  s_{n}\right)
\geq \lim_{n\rightarrow \infty}\left \Vert s_{n}\right \Vert _{1}=\left \Vert
f\right \Vert _{1}.
\]

$(9)$ (a) Since $\alpha(h)=\alpha(|h|)$ for every measurable function $h$, we
may, without loss of generality, assume $f\ $and $g$ are both nonnegative
measurable functions on $\mathbb{T}.$ Choose the $s_{n}$'s as in the proof of
part (8) to get $\alpha(f)=\lim_{n\rightarrow \infty}\alpha(s_{n}).$ Similarly,
choose another sequence $\{r_{n}\}$ of simple functions such that
$\alpha(g)=\lim_{n\rightarrow \infty}\alpha(r_{n}).$ It is easy to see
$\{s_{n}+r_{n}\}$ is an increasing sequence of simple functions with
$s_{n}(w)+r_{n}(w)\rightarrow f(w)+g(w)$ for every $w\in \mathbb{T}.$ It
follows from part (6) that%
\[
\alpha(f+g)=\lim_{n\rightarrow \infty}\alpha(s_{n}+r_{n})\leq \lim
_{n\rightarrow \infty}\alpha(s_{n})+\lim_{n\rightarrow \infty}\alpha
(r_{n})=\alpha(f)+\alpha(g).
\]

(b) It follows from the proof in (9a) that
\[
\alpha(\lambda f)=\alpha(|\lambda||f|)=\lim_{n\rightarrow \infty}%
\alpha(|\lambda|s_{n})=|\lambda|\lim_{n\rightarrow \infty}\alpha(s_{n}%
)=|\lambda|\alpha(f).
\]

(c) Suppose $\alpha(f)=0$ for some measurable function $f$ on $\mathbb{T}.$
Then, by part (8), $0\leq \Vert f\Vert_{1}\leq \alpha(f)=0,$ which implies that
$f=0.$ Hence in this case, the seminorm $\alpha$ is actually a norm.

$(10)$ It follows from the definition of ${\mathcal{L}}^{\alpha}(\mathbb{T})$
that $({\mathcal{L}}^{\alpha}(\mathbb{T}),\alpha)$ is a normed space. To prove
the completeness, suppose $\{f_{n}\}$ is a sequence in ${\mathcal{L}}^{\alpha
}(\mathbb{T})$ with $\sum_{n=1}^{\infty}\alpha(f_{n})<\infty.$ Then
\[
\infty>\sum_{n=1}^{\infty}\alpha(f_{n})\geq \sum_{n=1}^{\infty}\Vert f_{n}%
\Vert_{1}=\sum_{n=1}^{\infty}\int_{\mathbb{T}}|f_{n}|dm=\int_{\mathbb{T}}%
\sum_{n=1}^{\infty}|f_{n}|dm.
\]
If we let $f=\sum_{n=1}^{\infty}|f_{n}|,$ then $f\in L^{1}(\mathbb{T}).$ Since
$\{g_{N}=\sum_{n=1}^{N}|f_{n}|:N\geq1\}$ is an increasing sequence with
$g_{N}(w)\rightarrow f(w)$ a.e. $(m),$ it follows from part (6) that
\[
\alpha(f)=\lim_{N\rightarrow \infty}\alpha(g_{N})=\lim_{N\rightarrow \infty
}\alpha(\sum_{n=1}^{N}|f_{n}|)=\sum_{n=1}^{\infty}\alpha(|f_{n}|)=\sum
_{n=1}^{\infty}\alpha(f_{n})<\infty,
\]
which implies that $f\in{\mathcal{L}}^{\alpha}(\mathbb{T}).$ Applying part (6)
again,
\[
\alpha(f-\sum_{n=1}^{N}|f_{n}|)=\alpha(\sum_{n=N}^{\infty}|f_{n}|)=\sum
_{n=N}^{\infty}\alpha(|f_{n}|)=\sum_{n=N}^{\infty}\alpha(f_{n})\rightarrow0.
\]
Hence $\sum_{n=1}^{\infty}f_{n}\leq \sum_{n=1}^{\infty}|f_{n}|<\infty,$ which
tells us that $({\mathcal{L}}^{\alpha}(\mathbb{T}),\alpha)$ is a Banach space.
Furthermore, it follows from the definition of $L^{\alpha}(\mathbb{T}),$
${\mathcal{L}}^{\alpha}(\mathbb{T})$ and part (8) that
\[
L^{\infty}(\mathbb{T})\subset L^{\alpha}(\mathbb{T})\subset{\mathcal{L}%
}^{\alpha}(\mathbb{T})\subset L^{1}(\mathbb{T}).
\]

\end{proof}

The following lemma gives ways of constructing examples of rotationally
symmetric norms of different types. See Remark \ref{examples} for more
details. We first construct a sigma-algebra on $\mathcal{R}$. For each $f\in
L^{\infty}\left(  \mathbb{T}\right)  $, we have a mapping $\pi_{f}%
:\mathcal{R}\rightarrow \lbrack0,\infty)$ by%
\[
\pi_{f}\left(  \alpha \right)  =\alpha \left(  f\right)  .
\]
Note that a net $\left \{  \alpha_{\lambda}\right \}  $ in $\mathcal{R}$
converges pointwise to $\alpha$ if and only if, for every $f\in L^{\infty
}\left(  \mathbb{T}\right)  $, $\pi_{f}\left(  \alpha_{\lambda}\right)
\rightarrow \pi_{f}\left(  \alpha \right)  $. It follows that the weak topology
$\mathcal{T}\left(  \mathcal{R}\right)  $ on $\mathcal{R}$ induced by the
family $\left \{  \pi_{f}:f\in L^{\infty}\left(  \mathbb{T}\right)  \right \}  $
is the topology of pointwise convergence, and the weak topology on
$\mathcal{S}$ induced by $\left \{  \pi_{f}|_{\mathcal{S}}:f\in L^{\infty
}\left(  \mathbb{T}\right)  \right \}  $ is $\mathcal{T}\left(  \mathcal{S}%
\right)  =\left \{  U\cap \mathcal{S}:U\in \mathcal{T}\left(  \mathcal{R}\right)
\right \}  $ and is the topology of pointwise convergence on $\mathcal{S}$. We
let $\mathcal{M}\left(  \mathcal{R}\right)  $ denote the smallest $\sigma
$-algebra on $\mathcal{R}$ for which each $\pi_{f}$ $\left(  f\in L^{\infty
}\left(  \mathbb{T}\right)  \right)  $ is measurable. Similarly,
$\mathcal{M}\left(  \mathcal{S}\right)  =\left \{  E\cap \mathcal{S}%
:E\in \mathcal{M}\left(  \mathcal{R}\right)  \right \}  $ is the smallest
$\sigma$-algebra on $\mathcal{S}$ for which each $\pi_{f}$ is measurable.

If $\left(  \Omega,\mathcal{N},\lambda \right)  $ is a probability space and
$\rho:\Omega \rightarrow \mathcal{R}$, then $\rho$ is $\mathcal{M}\left(
\mathcal{R}\right)  $-$\mathcal{N}$ measurable if and only if, for each $f\in
L^{\infty}\left(  \mathbb{T}\right)  $, $\pi_{f}\circ \rho:\Omega
\rightarrow \lbrack0,\infty)$ is $\mathcal{N}$-measurable. We can uniquely
define $\int_{\Omega}\rho d\lambda$ as a function on $L^{\infty}\left(
\mathbb{T}\right)  $ by%
\[
\left(  \int_{\Omega}\rho d\lambda \right)  \left(  f\right)  =\int_{\Omega
}\left(  \pi_{f}\circ \rho \right)  d\lambda=\int_{\Omega}\left(  \rho \left(
\omega \right)  \right)  \left(  f\right)  d\lambda \left(  \omega \right)  .
\]

\begin{lemma}
\label{2}The following are true.

\begin{enumerate}
\item ${\mathcal{R}},\ {\mathcal{S}},\ {\mathcal{R}}_{c},\ {\mathcal{S}}_{c}$
are convex;

\item ${\mathcal{R}}$ and ${\mathcal{S}}$ are compact in the topology of
pointwise convergence, $\mathcal{S}$ is metrizable,and $\mathcal{M}\left(
\mathcal{S}\right)  $ is the collection $Bor\left(  \mathcal{S}\right)  $ of
Borel subsets of $\mathcal{S}$;

\item If $\alpha_{1},\  \alpha_{2},\  \cdots \in R_{c},$ and $t_{1}%
,\ t_{2},\cdots>0$ with $\sum_{t=1}^{\infty}t_{n}=1,$ then $\sum_{t=1}%
^{\infty}t_{n}\alpha_{n}\in R_{c}$;

\item Suppose $\mathcal{E}$ is a nonempty set of seminorms on $L^{\infty
}(\mathbb{T})$ such that

\begin{enumerate}
\item $\gamma(|f|)=\gamma(f)$ for every $f\in L^{\infty}(\mathbb{T})$ and for
every $\gamma \in{\mathcal{E}}$ and

\item $1=\sup \{ \gamma(1):\gamma \in{\mathcal{E}}\},$

and define $\alpha$ by
\[
\alpha(f)=\sup_{\gamma \in{\mathcal{E}}}\gamma(f).
\]
Then

\item If $\gamma \left(  f\right)  =\gamma \left(  f_{\omega}\right)  $ for
every $f\in L^{\infty}\left(  \mathbb{T}\right)  ,$ for every $\gamma
\in \mathcal{E}$ and every $\omega \in \mathbb{T}$, then $\alpha \in$
$\mathcal{R};$

\item If $\gamma \left(  f\right)  =\gamma \left(  f\circ \phi \right)  $ for
every $f\in L^{\infty}\left(  \mathbb{T}\right)  ,$ for every $\gamma
\in \mathcal{E}$ and every $\phi \in MP\left(  \mathbb{T}\right)  $, then
$\alpha \in \mathcal{S}$.
\end{enumerate}

\item If $h:{\mathbb{T}}\rightarrow{\mathbb{C}}$ and $h\geq0$ and
$0<\int_{\mathbb{T}}hdm\leq1,$ then the maps $\alpha_{h},\beta_{h}:L^{\infty
}(\mathbb{T})\rightarrow \lbrack0,\infty)$ defined by
\[
\alpha_{h}(f)=\sup_{w\in \mathbb{T}}\int_{\mathbb{T}}|f_{w}|hdm\text{ }%
\]
and
\[
\beta_{h}\left(  f\right)  =\sup_{\phi \in MP\left(  \mathbb{T}\right)  }%
\int_{\mathbb{T}}|f\circ \phi|hdm=\int_{\mathbb{T}}|f|^{\bigstar}h^{\bigstar
}dm
\]
satisfy, for every $f\in L^{\infty}\left(  \mathbb{T}\right)  $, $\alpha
_{h}(f)=\alpha_{h}(|f|)=\alpha_{h}(f_{w})$ for every $w\in \mathbb{T}$ and
$\beta_{h}(f)=\beta_{h}(|f|)=\beta_{h}(f\circ \phi)$ for every $\phi \in
MP\left(  \mathbb{T}\right)  $. If $\int_{\mathbb{T}}hdm=1,$ then $\alpha
_{h}\in{\mathcal{R}_{c}}$ and $\beta_{h}\in \mathcal{S}_{c}$.

\item If $\alpha \in \mathcal{S}$ and $t=\lim_{m\left(  E\right)  \rightarrow
0^{+}}\alpha \left(  \chi_{E}\right)  $, then there is a unique $\beta
\in \mathcal{S}_{c}$ such that
\[
\alpha=\left(  1-t\right)  \beta+t\Vert \cdot \Vert_{\infty}\text{.}%
\]

\end{enumerate}
\end{lemma}

\begin{proof}
$(1)$ This is obvious.

$(2)$ Suppose $\{ \alpha_{\lambda}\}$ is an ultranet in $\mathcal{R}$
(respectively, $\mathcal{S}$). Then $\alpha_{\lambda}(f)\leq \Vert
f\Vert_{\infty}<\infty$ for every $f\in L^{\infty}(\mathbb{T}),$ which implies
that $\left \{  \alpha_{\lambda}(f)\right \}  $ is an ultranet in the compact
set $\left \{  z\in \mathbb{C}:\left \vert z\right \vert \leq \left \Vert
f\right \Vert _{\infty}\right \}  ,$ so
\[
\alpha(f)=\lim_{\lambda}\alpha_{\lambda}\left(  f\right)
\]
exists for every $f\in L^{\infty}\left(  \mathbb{T}\right)  $. It is clear
that $\alpha \in \mathcal{R}$ (respectively, $\mathcal{S}$). Since every net has
a subnet that is an ultranet, we see that $\mathcal{R}$ and $\mathcal{S}$ are
compact. For the proof that $\mathcal{S}$ is metrizable we identify $(0,1]$
with $\mathbb{T}$ (identifying $t$ with $e^{2\pi it}$). Let $\mathcal{W}$ be
the set of all simple functions on $(0,1]$ of the form $s=\sum_{k=1}^{n}%
r_{k}\chi_{\left(  a_{k-1},a_{k}\right)  }$ with $0=a_{0}<a_{1}<\cdots
<a_{n}\leq1,$ and $0\leq r_{1},\ldots,r_{n}$ and $a_{0},r_{1},a_{1}%
,\cdots,r_{n},a_{n}$ rational numbers. Clearly $\mathcal{W=}\left \{
f_{1},f_{2},\ldots \right \}  $ is countable. We claim that $\mathcal{T}\left(
\mathcal{S}\right)  $ is the weak topology induced by $\left \{  \pi_{f_{n}%
}:n\in \mathbb{N}\right \}  $. Suppose $\left \{  \alpha_{\lambda}\right \}  $ is
a net in $\mathcal{S}$ and $\alpha \in \mathcal{S}$ and $\lim_{\lambda}%
\pi_{f_{n}}\left(  \alpha_{\lambda}\right)  =\pi_{f_{n}}\left(  \alpha \right)
$ for every $n\in \mathbb{N}$. From the compactness of $\mathcal{S}$, we can
choose a subnet $\left \{  \alpha_{\lambda_{k}}\right \}  $ and a $\beta
\in \mathcal{S}$ such that, for every $f\in L^{\infty}\left(  (0,1]\right)  $,
$\lim_{k}\pi_{f}\left(  \alpha_{\lambda_{k}}\right)  =\pi_{f}\left(
\beta \right)  .$ Since, by definition, $\pi_{g}\left(  \gamma \right)
=\gamma \left(  g\right)  ,$ we see that $\beta \left(  f\right)  =\alpha \left(
f\right)  $ for every $f\in \mathcal{W}$. Suppose $s\in L^{\infty}\left(
(0,1]\right)  $ is a simple function. Then $\left \vert s\right \vert $ is a
nonnegative simple function, so there is a $\phi \in MP\left(  (0,1]\right)  $
such that
\[
\left \vert s\right \vert \circ \phi=\sum_{k=1}^{n}s_{k}\chi_{\left(
b_{k-1},b_{k}\right)  }%
\]
with $0=b_{0}<b_{1}<\cdots<b_{n}\leq1$ and $s_{1},\ldots,s_{n}\geq0.$ Clearly,
there is a sequence $\left \{  w_{n}\right \}  $ in $\mathcal{W}$ such that
$0\leq w_{1}\leq w_{2}\leq \cdots$ and $\lim_{j\rightarrow \infty}w_{j}\left(
t\right)  =\left(  \left \vert s\right \vert \circ \phi \right)  \left(  t\right)
$ for every $t\in(0,1]$, that is, if, for each $k$, we choose rational numbers
$0\leq s_{k,1}\leq s_{k,2}\leq \cdots$ converging to $s_{k}$ and rational
numbers $b_{k-1}<\cdots<c_{k,2}<c_{k,1}<d_{k,1}<d_{k,2}<\cdots<b_{k}$ with
$c_{k,j}\rightarrow b_{k-1}$ and $d_{k,j}\rightarrow b_{k}$, and we let
$w_{j}=\sum_{k=1}^{n}s_{k,j}\chi_{\left(  c_{k,j},d_{k,j}\right)  }$. It
follows from part (\ref{part1}) of Lemma \ref{prop1} that
\[
\alpha \left(  s\right)  =\alpha \left(  \left \vert s\right \vert \circ
\phi \right)  =\lim_{j\rightarrow \infty}\alpha \left(  w_{j}\right)
=\lim_{j\rightarrow \infty}\beta \left(  w_{j}\right)  =\beta \left(  s\right)
.
\]
Since the simple functions are $\left \Vert \cdot \right \Vert _{\infty}$-dense
in $L^{\infty}\left(  \mathbb{T}\right)  $, we see that $\alpha=\beta$. Hence
$\alpha_{\lambda}\rightarrow \alpha$ with respect to $\mathcal{T}\left(
\mathcal{S}\right)  $. Thus the weak topology on $\mathcal{S}$ induced by
$\left \{  \pi_{f}:f\in \mathcal{W}\right \}  $ is $\mathcal{T}\left(
\mathcal{S}\right)  $. Since $\mathcal{W}$ is countable, $\mathcal{S}$ is
metrizable. Thus $\left(  \mathcal{S},\mathcal{T}\left(  \mathcal{S}\right)
\right)  $ is a compact metric space, and $Bor\left(  \mathcal{S}\right)  $ is
the smallest $\sigma$-algebra for which each $\pi_{f}$ $\left(  f\in
\mathcal{W}\right)  $ measurable, and hence it is the smallest $\sigma
$-algebra that makes each $\pi_{f}$ $\left(  f\in L^{\infty}\left(
\mathbb{T}\right)  \right)  $ measurable.

$(3)$ Let $\alpha=\sum_{n=1}^{\infty}t_{n}\alpha_{n}.$ It easily follows that
$\alpha$ is a rotationally symmetric semi-norm. Suppose $\sum_{t=1}^{\infty
}t_{n}=1$ and $\epsilon>0.$ Then there is an $N\in \mathbb{N}$ such that
$\sum_{n=N+1}^{\infty}t_{n}<\frac{\epsilon}{2}$ for all $n\geq N.$ Thus for
any measurable set $E\subset \mathbb{T},$
\[
\sum_{n=N+1}^{\infty}t_{n}\alpha(\chi_{E})\leq \sum_{n=N+1}^{\infty}t_{n}%
\cdot1<\frac{\epsilon}{2}.
\]
Since $\alpha_{1},\alpha_{2},\cdot \cdot \cdot,\alpha_{N}$ are continuous, there
is a $\delta>0$ such that when $E\subset \mathbb{T}$ and $0<m(E)<\delta,$ we
have $\alpha_{k}(\chi_{E})<\frac{\epsilon}{2N}$ for all $1\leq k\leq N.$ Thus
$n\geq N$ and $E\subset \mathbb{T}$ and $0<m(E)<\delta$ implies%
\begin{align*}
\alpha(\chi_{E})  &  =\sum_{n=1}^{N}t_{n}\alpha_{n}(\chi_{E})+\sum
_{n=N+1}^{\infty}t_{n}\alpha_{n}(\chi_{E})\\
&  <\sum_{n=1}^{N}t_{n}\frac{\epsilon}{2N}+\frac{\epsilon}{2}\\
&  <\frac{\epsilon}{2}+\frac{\epsilon}{2}=\epsilon.
\end{align*}
Hence $\alpha$ is a continuous rotationally symmetric semi-norm.

$(4$) Suppose the hypotheses of (4a), (4b), (4c) hold. Then $\alpha
(1)=\sup_{\gamma \in{\mathcal{E}}}\gamma(1)=1,$ $\alpha(|f|)=\sup_{\gamma
\in{\mathcal{E}}}(\gamma(|f|))=\alpha(f),$ and $\alpha(f_{w})=\sup_{\gamma
\in \mathcal{E}}\gamma(f_{w})=\sup_{\gamma \in \mathcal{E}}\gamma(f)=\alpha(f).$
Since $\mathcal{E}$ is a set of seminorms on $L^{\infty}(\mathbb{T}),$
$\sup_{\gamma \in \mathcal{E}}\gamma$ is a seminorm, which implies that $\alpha$
is a seminorm on $L^{\infty}(\mathbb{T}).$ Hence $\alpha \in \mathcal{R}$. The
proof of (4d) is similar.

$(5)$ It is clear that $\alpha_{h}(1)=\sup_{w\in \mathbb{T}}\int_{\mathbb{T}%
}|1|hdm=\int_{\mathbb{T}}hdm=1.$ Since $h\geq0$ and $0<\int_{\mathbb{T}%
}hdm\leq1,$ it follows that $\alpha_{h}(f)\geq0,$ $\alpha(\lambda
f)=|\lambda|\alpha(f)$ and $\alpha(f+g)\leq \alpha(f)+\alpha(g)$ for every
$f,g\in L^{\infty}(\mathbb{T})$ and every $\lambda \in \mathbb{C}.$ Furthermore,
suppose $\epsilon>0.$ Since $\int_{\mathbb{T}}hdm\leq1,$ there is a $\delta>0$
such that when $E\subset \mathbb{T}$ and $m(E)<\delta,$ we have $|\int
_{E}hdm|<\epsilon,$ thus
\[
|\alpha_{h}(\chi_{E})|=|\int_{\mathbb{T}}h(\chi_{E})_{w}dm|=|\int_{\mathbb{T}%
}h\chi_{\bar{w}E}dm|=|\int_{\bar{w}E}hdm|<\epsilon
\]
as $m(E)<\delta.$ Therefore $\lim_{m(E)\rightarrow0^{+}}\alpha_{h}(\chi
_{E})=0,$ which means $\alpha_{h}\in{\mathcal{R}}_{c}.$ The proof for
$\beta_{h}$ is similar.

$(6)$ If $t=0$, then $\alpha \in \mathcal{S}_{c}$ and $\beta=\alpha$. Suppose
$\lim_{m(E)\rightarrow0}\alpha(\chi_{E})=t>0.$ Then
\[
\Vert f\Vert_{\infty}\geq \alpha(f)\geq \alpha(f\chi_{E})\geq t\Vert
f\Vert_{\infty},
\]
which means $0<t\leq1$ and $\beta=\left(  \alpha-t\Vert \cdot \Vert_{\infty
}\right)  /\left(  1-t\right)  $ defines an element of $\mathcal{S}_{c}.$
\end{proof}

\section{Dual Norms}

Suppose $\alpha$ is a rotationally invariant norm on $L^{\infty}\left(
\mathbb{T}\right)  $. We define the \emph{dual norm} $\alpha^{\prime}$ on
$L^{\infty}\left(  \mathbb{T}\right)  $ by%
\begin{align*}
\alpha^{\prime}\left(  f\right)   &  =\sup \left \{  \left \vert \int
_{\mathbb{T}}fhdm\right \vert :h\in L^{\infty}\left(  \mathbb{T}\right)
,\alpha \left(  h\right)  \leq1\right \} \\
&  =\sup \left \{  \int_{\mathbb{T}}\left \vert fh\right \vert dm:h\in L^{\infty
}\left(  \mathbb{T}\right)  ,\alpha \left(  h\right)  \leq1\right \}  .
\end{align*}

\begin{lemma}
Suppose $\alpha \in \mathcal{R}$. The following statements are true.

\begin{enumerate}
\item $\alpha^{\prime}\in \mathcal{R};$

\item $\alpha \in{\mathcal{S}}\Longrightarrow \alpha^{\prime}\in \mathcal{S}.$
\end{enumerate}
\end{lemma}

\begin{proof}
$(1)$ It is clear that $\alpha^{\prime}$ is a seminorm. Suppose $f\in
L^{\infty}(\mathbb{T})$ and $w\in \mathbb{T}.$ It follows from the definition
of $\alpha^{\prime}$ that $\alpha^{\prime}(1)=1,$ $\alpha^{\prime}%
(|f|)=\alpha^{\prime}(f)$ and $\alpha^{\prime}(f_{w})=\alpha^{\prime}(f).$
Hence $\alpha^{\prime}\in \mathcal{R}.$

$(2)$ This assertion was proved in \cite{FHNS}.
\end{proof}

The following result is probably not new \cite[Theorem 5.11]{Shaefer}. We will
visit these ideas again in Section 10.

\begin{proposition}
\label{reflexive space} \label{dualspace} Suppose $\alpha \in{\mathcal{R}}.$ Then

\begin{enumerate}
\item If $\alpha$ is continuous, then $L^{\alpha}(\mathbb{T})^{\sharp
}={\mathcal{L}}^{\alpha^{\prime}}\left(  \mathbb{T}\right)  $, i.e., for every
$\phi \in L^{\alpha}(\mathbb{T})^{\sharp}$, there is a $h\in \mathcal{L}%
^{\alpha^{\prime}}(\mathbb{T})$ such that
\[
\phi(f)=\int_{\mathbb{T}}fhdm\
\]
for all $f\in L^{\alpha}\left(  \mathbb{T}\right)  $ and with $\left \Vert
\phi \right \Vert =\alpha^{\prime}\left(  h\right)  ;$

\item If $\alpha$ is continuous, then $\alpha^{\prime \prime}=\alpha$;

\item $L^{\alpha}(\mathbb{T})^{\# \#}=L^{\alpha}\left(  \mathbb{T}\right)  $
and $H^{\alpha}(\mathbb{T})^{\# \#}=H^{\alpha}\left(  \mathbb{T}\right)  $ if
$\alpha$ and $\alpha^{\prime}$ are both strongly continuous;

\item If $\alpha \in \mathcal{S}$, then $L^{\alpha}\left(  \mathbb{T}\right)
^{\# \#}=L^{\alpha}\left(  \mathbb{T}\right)  $ if and only if $\alpha$ and
$\alpha^{\prime}$ are both strongly continuous.
\end{enumerate}
\end{proposition}

\begin{proof}
$(1)$ Suppose $\alpha \in \mathcal{R}_{c}$. For any measurable set
$E\subset \mathbb{T},$ define
\[
\lambda(E)=\phi(\chi_{E}).
\]
Then $\lambda(\emptyset)=\phi(\chi_{\emptyset})=\phi(0)=0.$ Also, Since $\phi$
is linear, and since $\chi_{A\cup B}=\chi_{A}+\chi_{B}$ if $A$ and $B$ are
disjoint, we see $\lambda$ is additive. To prove countable additivity, suppose
$E$ is the union of countably many disjoint measurable sets $E_{i},$ put
$A_{k}=E_{1}\cup E_{2}\cup \cdots \cup E_{k},$ and note that
\[
m(E-A_{k})\rightarrow0\  \  \  \ (k\rightarrow \infty).
\]
The continuity of $\alpha$ implies $\alpha(\chi_{E-A_{k}})=\alpha(\chi
_{E}-\chi_{A_{k}})\rightarrow0,$ and $\phi(\chi_{E}-\chi_{A_{k}}%
)\rightarrow0.$ Therefore $\lambda(A_{k})\rightarrow \lambda(E),$ which is
$\lambda(E)=\lambda(\cup_{k=1}^{\infty}A_{k})=\sum_{k=1}^{\infty}%
\lambda({A_{k}}).$ So $\lambda$ is a complex measure. It is clear that
$\lambda(E)=0$ if $m(E)=0,$ since then $\Vert \chi_{E}\Vert_{\infty}=0.$ Thus
$\lambda<<m,$ and the Radon-Nikodym theorem ensures the existence of a
function $h\in L^{1}(\mathbb{T})$ such that, for every measurable $E\subset
X,$%
\[
\phi(\chi_{E})=\lambda(E)=\int_{\mathbb{T}}\chi_{E}hdm
\]
and such that%
\[
\phi(f)=\int_{\mathbb{T}}fhdm
\]
for every $f\in L^{\infty}\left(  \mathbb{T}\right)  $ (First consider $f$ a
simple function and use $\alpha \leq \left \Vert \cdot \right \Vert _{\infty}.$).

The uniqueness of $h$ is clear, for if $h$ and $h^{\prime}$ satisfy (1), then
the integral of $h-h^{\prime}$ over any measurable set $E$ of finite measure
is $0$ (as we see by taking $\chi_{E}$ for $f$), and the $\sigma$-finiteness
of $m$ implies that $h-h^{\prime}=0$ a.e. $(m).$ Furthermore, since
$L^{\infty}\left(  \mathbb{T}\right)  $ is dense in $L^{\alpha}\left(
\mathbb{T}\right)  $, we see
\begin{align*}
\left \Vert \phi \right \Vert  &  =\sup \left \{  \left \vert \phi \left(  f\right)
\right \vert :f\in L^{\infty}\left(  \mathbb{T}\right)  \text{, }\alpha \left(
f\right)  \leq1\right \} \\
&  =\sup \left \{  \left \vert \int_{\mathbb{T}}fhdm\right \vert :f\in L^{\infty
}\left(  \mathbb{T}\right)  \text{, }\alpha \left(  f\right)  \leq1\right \} \\
&  =\alpha^{\prime}\left(  h\right)  .
\end{align*}
Thus $h\in \mathcal{L}^{\alpha^{\prime}}\left(  \mathbb{T}\right)  .$

$\left(  2\right)  $ Suppose $f\in L^{\infty}(\mathbb{T})$ with $\alpha(f)=1.$
It follows from
\[
\alpha^{\prime}(h)=\sup \left \{  \int_{\mathbb{T}}\left \vert fh\right \vert
dm:h\in L^{\infty}\left(  \mathbb{T}\right)  ,\alpha \left(  h\right)
\leq1\right \}
\]
that
\[
\alpha^{\prime \prime}(f)=\sup_{h\in L^{\infty}(\mathbb{T}),\alpha^{\prime
}(h)\leq1}\int_{\mathbb{T}}|fh|dm\leq \sup_{h\in L^{\infty}(\mathbb{T}%
),\alpha^{\prime}(h)\leq1}\alpha^{\prime}(h)=1.
\]
By the Hahn-Banach theorem, there is a continuous linear functional $\phi \in
L^{\alpha}(\mathbb{T})^{\sharp}$ such that $\phi(f)=\alpha(f)=1$ and
$\Vert \phi \Vert=1.$ Since $\phi \in L^{\alpha}(\mathbb{T})^{\sharp},$ there is
an element $h\in{\mathcal{L}}^{\alpha^{\prime}}(\mathbb{T})$ such that
$\phi(|f|)=\int_{\mathbb{T}}|f||h|dm=1$ and $\alpha^{\prime}(h)=\Vert \phi
\Vert=1.$ Thus
\[
1=\int_{\mathbb{T}}|f||h|dm\leq \sup_{h\in{\mathcal{L}}^{\alpha}(\mathbb{T}%
),\alpha^{\prime}(h)\leq1}\int_{\mathbb{T}}|f||h|dm=\alpha^{\prime \prime}(f),
\]
and so $\alpha^{\prime \prime}({f})=1=\alpha(f).$ Next suppose $f\neq0.$ Then
$\alpha(\frac{f}{\alpha(f)})=1,$ and it follows that $\alpha^{\prime \prime
}(\frac{f}{\alpha(f)})=1,$ which means $\alpha^{\prime \prime}(f)=\alpha(f).$

$(3)$ This is clear from $\left(  1\right)  $ and $\left(  2\right)  $.

$\left(  4\right)  $ If $\alpha \in \mathcal{S}$ and $\alpha$ or $\alpha
^{\prime}$ is not continuous, then one of $\alpha,\alpha^{\prime}$ is
equivalent to $\left \Vert \cdot \right \Vert _{\infty}$ and the other is
equivalent to $\left \Vert \cdot \right \Vert _{1}$, so $L^{\alpha}\left(
\mathbb{T}\right)  ^{\# \#}\neq L^{\alpha}\left(  \mathbb{T}\right)  $. Also
if $L^{\alpha^{\prime}}\left(  \mathbb{T}\right)  \neq \mathcal{L}%
^{\alpha^{\prime}}\left(  \mathbb{T}\right)  ,$ then there is a $0\neq \phi
\in \mathcal{L}^{\alpha^{\prime}}\left(  \mathbb{T}\right)  ^{\#}$ such that
$\phi=0$ on $L^{\alpha^{\prime}}\left(  \mathbb{T}\right)  $ and such a $\phi$
cannot be written in the form $\phi \left(  f\right)  =\int_{\mathbb{T}}fhdm$,
e.g., let $f=\bar{h}/\left \vert h\right \vert \in L^{\infty}\left(
\mathbb{T}\right)  $, which implies $L^{\alpha}\left(  \mathbb{T}\right)  ^{\#
\#}\neq L^{\alpha}\left(  \mathbb{T}\right)  $.
\end{proof}

\begin{theorem}
\label{multalpha1}Suppose $\alpha$ is a rotationally symmetric norm and
$T:\mathcal{L}^{\alpha}\left(  \mathbb{T}\right)  \rightarrow L^{1}\left(
\mathbb{T}\right)  $ is a bounded linear operator such that, for every $h\in
L^{\infty}\left(  \mathbb{T}\right)  $ and every $g\in \mathcal{L}^{\alpha
}\left(  \mathbb{T}\right)  ,$%
\[
T\left(  hg\right)  =hT\left(  g\right)  .
\]
Then there is an $f\in \mathcal{L}^{\alpha^{\prime}}\left(  \mathbb{T}\right)
$ such that, for every $g\in \mathcal{L}^{\alpha}\left(  \mathbb{T}\right)  ,$%
\[
Tg=fg.
\]
Moreover, $\left \Vert T\right \Vert =\alpha^{\prime}\left(  f\right)  .$

The same conclusion holds if $\mathcal{L}^{\alpha}\left(  \mathbb{T}\right)  $
is replaced with $L^{\alpha}\left(  \mathbb{T}\right)  $.
\end{theorem}

\begin{proof}
Let $f=T\left(  1\right)  .$ Then $Tg=fg$ for every $g\in L^{\infty}\left(
\mathbb{T}\right)  $. Suppose $g\in \mathcal{L}^{\alpha}\left(  \mathbb{T}%
\right)  .$ Define
\[
u\left(  z\right)  =\left \{
\begin{array}
[c]{cc}%
g\left(  z\right)  & \text{if }\left \vert g\left(  z\right)  \right \vert
\leq1\\
1 & \text{if }\left \vert g\left(  z\right)  \right \vert >1
\end{array}
\right.
\]
and%
\[
v\left(  z\right)  =\left \{
\begin{array}
[c]{cc}%
1 & \text{if }\left \vert g\left(  z\right)  \right \vert \leq1\\
1/g\left(  z\right)  & \text{if }\left \vert g\left(  z\right)  \right \vert >1
\end{array}
\right.  .
\]
Then $u,v\in L^{\infty}\left(  \mathbb{T}\right)  $, $v\left(  z\right)  $ is
never $0$, and $g=u/v$. Then%
\[
vT\left(  g\right)  =T\left(  u\right)  =uT\left(  1\right)  =fu
\]
implies $Tg=fg$. Also%
\[
\alpha^{\prime}\left(  f\right)  =\sup_{h\in L^{\infty}\left(  \mathbb{T}%
\right)  ,\alpha \left(  h\right)  \leq1}\left \vert \int_{\mathbb{T}%
}fhdm\right \vert \leq \left \Vert T\right \Vert <\infty.
\]
On the other hand, $\left \Vert Tg\right \Vert _{1}=\left \Vert fg\right \Vert
_{1}\leq \alpha^{\prime}\left(  f\right)  \alpha \left(  g\right)  $ implies
$\left \Vert T\right \Vert \leq \alpha^{\prime}\left(  f\right)  .$
\end{proof}

\begin{corollary}
\label{multalfa21}Suppose $\alpha$ is a rotationally symmetric norm and
$f:\mathbb{T}\rightarrow \mathbb{C}$ is measurable. Then
\[
f\cdot L^{\alpha}\left(  \mathbb{T}\right)  \subset L^{1}\left(
\mathbb{T}\right)  \Longleftrightarrow \text{ }f\in \mathcal{L}^{\alpha^{\prime
}}\left(  \mathbb{T}\right)  .
\]

\end{corollary}

The preceding theorem yields a Banach space characterization of strongly
continuous norms. A Banach space $X$ is \emph{weakly sequentially complete} if
and only if every weakly Cauchy sequence is weakly convergent.

\bigskip

\begin{theorem}
\label{wc}Suppose $\alpha \in \mathcal{R}_{c}$ has dual norm $\alpha^{\prime}$.
The following are equivalent:
\end{theorem}

\begin{enumerate}
\item $L^{\alpha}\left(  \mathbb{T}\right)  =\mathcal{L}^{\alpha}\left(
\mathbb{T}\right)  $ ($\alpha$ is strongly continuous);

\item $L^{\alpha}\left(  \mathbb{T}\right)  $ is weakly sequentially complete.
\end{enumerate}

\begin{proof}
$\left(  1\right)  \Rightarrow \left(  2\right)  $ Suppose (1) is true, and
suppose $\left \{  f_{n}\right \}  $ is a weakly Cauchy sequence in $L^{\alpha
}\left(  \mathbb{T}\right)  $. Then, by the uniform boundedness theorem,
$s=\sup_{k\geq1}\alpha \left(  f_{k}\right)  <\infty$. Also, for every
$h\in \mathcal{L}^{\alpha^{\prime}}\left(  \mathbb{T}\right)  =L^{\alpha
}\left(  \mathbb{T}\right)  ^{\#}$ and every $u\in L^{\infty}\left(
\mathbb{T}\right)  $ we have $\left \{  \int_{\mathbb{T}}f_{n}hudm\right \}  $
is Cauchy, which means%
\[
\lim_{n\rightarrow \infty}\int_{\mathbb{T}}f_{n}hudm\text{ exists.}%
\]
However, $\left \{  f_{n}h\right \}  $ is a sequence in $L^{1}\left(
\mathbb{T}\right)  $ and $L^{1}\left(  \mathbb{T}\right)  ^{\#}=L^{\infty
}\left(  \mathbb{T}\right)  $, so it follows that $\left \{  f_{n}h\right \}  $
is weakly Cauchy in $L^{1}\left(  \mathbb{T}\right)  $. However, $L^{1}\left(
\mathbb{T}\right)  $ is weakly sequentially complete \cite{Shaefer}, so there
is a $T\left(  h\right)  \in L^{1}\left(  \mathbb{T}\right)  $, such that, for
every $u\in L^{\infty}\left(  \mathbb{T}\right)  ,$ we have%
\[
\lim_{n\rightarrow \infty}\int_{\mathbb{T}}f_{n}hudm=\int_{\mathbb{T}}T\left(
h\right)  udm.
\]
The map $T:\mathcal{L}^{\alpha^{\prime}}\left(  \mathbb{T}\right)  \rightarrow
L^{1}\left(  \mathbb{T}\right)  $ is clearly linear. Moreover,%
\[
\left \Vert T\left(  h\right)  \right \Vert _{1}=\sup_{u\in L^{\infty}\left(
\mathbb{T}\right)  ,\left \Vert u\right \Vert _{\infty}\leq1}\left \vert
\int_{\mathbb{T}}T\left(  h\right)  udm\right \vert =\lim_{n\rightarrow \infty
}\left \vert \int_{\mathbb{T}}f_{n}hudm\right \vert \leq s\alpha^{\prime}\left(
h\right)  ,
\]
since $\left \vert \int_{\mathbb{T}}f_{n}hudm\right \vert \leq \alpha \left(
f_{n}\right)  \alpha^{\prime}\left(  hu\right)  \leq s\alpha^{\prime}\left(
h\right)  \left \Vert u\right \Vert _{\infty}$. For every $u,w\in L^{\infty
}\left(  \mathbb{T}\right)  $ and $h\in \mathcal{L}^{\alpha^{\prime}}\left(
\mathbb{T}\right)  $, we have
\[
\int_{\mathbb{T}}T\left(  hw\right)  udm=\lim_{n\rightarrow \infty}%
\int_{\mathbb{T}}f_{n}\left(  hw\right)  udm=\int_{\mathbb{T}}T\left(
h\right)  wudm,
\]
and we conclude that $T\left(  hw\right)  =T\left(  h\right)  w$. It follows
from Theorem \ref{multalpha1} that there is an $f\in \mathcal{L}^{\alpha
^{\prime \prime}}\left(  \mathbb{T}\right)  =\mathcal{L}^{\alpha}\left(
\mathbb{T}\right)  =L^{\alpha}\left(  \mathbb{T}\right)  $ such that $T\left(
h\right)  =fh$ for every $h\in \mathcal{L}^{\alpha^{\prime}}\left(
\mathbb{T}\right)  $. From the definition of $T$ we see that $f_{n}\rightarrow
f$ weakly.

$\left(  2\right)  \Rightarrow \left(  1\right)  $ Suppose $L^{\alpha}\left(
\mathbb{T}\right)  $ is weakly sequentially complete, and suppose
$f\in \mathcal{L}^{\alpha}\left(  \mathbb{T}\right)  $. Then we can choose a
sequence $\left \{  s_{n}\right \}  $ of simple functions such that $0\leq
s_{1}\leq s_{2}\leq \cdots$ and $s_{n}\left(  z\right)  \rightarrow \left \vert
f\left(  z\right)  \right \vert $ for every $z\in \mathbb{T}$. If $h\in
\mathcal{L}^{\alpha^{\prime}}\left(  \mathbb{T}\right)  $ and $h\geq0$, then,
by the monotone convergence theorem,%
\[
\lim_{n\rightarrow \infty}\int_{\mathbb{T}}s_{n}hdm=\int \left \vert f\right \vert
hdm.
\]
Hence, since $\mathcal{L}^{\alpha^{\prime}}\left(  \mathbb{T}\right)  $ is the
linear span of its nonnegative elements, the above limit holds for every
$h\in \mathcal{L}^{\alpha \prime}\left(  \mathbb{T}\right)  .$ Hence $\left \{
s_{n}\right \}  $ is weakly Cauchy in $L^{\alpha}\left(  \mathbb{T}\right)  $,
so there is a $w\in L^{\alpha}\left(  \mathbb{T}\right)  $ such that, for
every $h\in \mathcal{L}^{\alpha^{\prime}}\left(  \mathbb{T}\right)  ,$%
\[
\lim_{n\rightarrow \infty}\int_{\mathbb{T}}s_{n}hdm=\int whdm.
\]
Clearly $\left \vert f\right \vert =w\in L^{\alpha}\left(  \mathbb{T}\right)  ,$
which means $f\in L^{\alpha}\left(  \mathbb{T}\right)  $. Thus $\mathcal{L}%
^{\alpha}\left(  \mathbb{T}\right)  =L^{\alpha}\left(  \mathbb{T}\right)  $.
\end{proof}

\section{Dominated Convergence Theorem on $L^{\alpha}(\mathbb{T})$}

In this section we prove a dominated convergence theorem that generalizes the
classical dominated convergence theorem when $\alpha=\left \Vert \cdot
\right \Vert _{1}$. We first prove an extension of the notion of continuity for
a norm in $\mathcal{R}$.

\begin{theorem}
\label{GCT}(General Continuity Theorem) Suppose $\alpha$ is a continuous gauge
seminorm and $g\in L^{\alpha}({\mathbb{T}}).$ Then $\lim_{m(E)\rightarrow
0^{+}}\alpha(g\chi_{E})=0.$
\end{theorem}

\begin{proof}
Suppose $g\in L^{\alpha}({\mathbb{T}})$ and $\epsilon>0.$ Then there is an
$f\in L^{\infty}({\mathbb{T}})$ such that $\alpha(g-f)<\frac{\epsilon}{2}.$
Since $\alpha$ is continuous, there is $\delta>0$ such that when
$E\subset \mathbb{T}$ and $0<m(E)<\delta,$ we obtain $\alpha(\chi_{E}%
)<\frac{\epsilon}{2\Vert f\Vert_{\infty}}.$ Thus $E\subset \mathbb{T}$ and
$0<m(E)<\delta$ implies that
\begin{align*}
\alpha(g\chi_{E})  &  =\alpha((g-f)\chi_{E}+f\chi_{E})\\
&  \leq \alpha((g-f)\chi_{E})+\alpha(f\chi_{E})\\
&  \leq \alpha(g-f)\Vert \chi_{E}\Vert_{\infty}+\Vert f\Vert_{\infty}\alpha
(\chi_{E})\\
&  <\frac{\epsilon}{2}+\Vert f\Vert_{\infty}\frac{\epsilon}{2\Vert
f\Vert_{\infty}}=\epsilon,
\end{align*}
which gives the result.
\end{proof}

\begin{remark}
\label{examples}In this remark we identify $\mathbb{T}$ with $(0,1].$ Suppose
$0<t\leq1$. If we choose a subset $A\subset \mathbb{T}$ with $m\left(
A\right)  =t$, and we let $h=\frac{1}{t}\chi_{A},$ then the norm $\beta_{h}$
defined in part (5) of Lemma \ref{2} is called the \emph{Ky Fan norm} and is
denoted by $\left \Vert \cdot \right \Vert _{t}$ $\in \mathcal{S}_{c}$. Thus, for
any measurable $f:(0,1]\rightarrow \mathbb{C}$, we have%
\[
\left \Vert f\right \Vert _{t}=\frac{1}{t}\int_{0}^{t}\left \vert f\left(
r\right)  \right \vert ^{\bigstar}dr,
\]
the average over $(0,t)$ of the nonincreasing rearrangement of $\left \vert
f\right \vert $. This fact allows us to focus on nonincreasing nonnegative
functions $f$.

Suppose $u:(0,1]\rightarrow(0,1]$ is any function (maybe not even measurable)
with
\[
\sup \left \{  u\left(  t\right)  :t\in \left[  0,1\right]  \right \}  =1.
\]
Then we can define a norm $\beta^{u}$ by%
\[
\beta^{u}\left(  f\right)  =\sup \left \{  u\left(  t\right)  \left \Vert
f\right \Vert _{t}:t\in \left[  0,1\right]  \right \}  =\sup \left \{  u\left(
t\right)  \left \Vert f\right \Vert _{t}:t\in(0,1]\right \}  \text{.}%
\]
It follows from part (4) of Lemma \ref{2} that $\beta^{u}\in \mathcal{S}$. It
is not hard to show that if $u=\chi_{(0,1]}$, then $\beta^{u}=\left \Vert
\cdot \right \Vert _{\infty}$ and if $u=\chi_{\left \{  1\right \}  }$ or
$u\left(  t\right)  =t,$ then $\beta^{u}=\left \Vert \cdot \right \Vert _{1}$.
Thus we might have $\beta^{u}\notin \mathcal{S}_{c}$.

We know that $\beta^{u}\in \mathcal{S}_{c}$ if and only if
\[
\lim_{s\rightarrow0^{+}}\beta^{u}\left(  \chi_{\lbrack0,s)}\right)  =0.
\]
However,
\[
u\left(  t\right)  \left \Vert \chi_{\lbrack0,s)}\right \Vert _{t}%
=\frac{u\left(  t\right)  \min \left(  s,t\right)  }{t}.
\]
Hence $\beta^{u}\in \mathcal{S}_{c}$ if and only if%
\[
\lim_{s\rightarrow0^{+}}\sup_{0<t\leq1}\frac{u\left(  t\right)  \min \left(
s,t\right)  }{t}=0.
\]

Since $\sup_{0<t\leq1}\frac{u\left(  t\right)  \min \left(  s,t\right)  }%
{t}\geq \sup_{0\leq t\leq s}u\left(  t\right)  $, we conclude that
\[
\beta^{u}\in \mathcal{S}_{c}\Rightarrow \lim_{s\rightarrow0^{+}}u\left(
s\right)  =0.
\]

If $u\left(  t\right)  /t$ is decreasing, then
\[
\sup_{s\leq t\leq1}\frac{u\left(  t\right)  \min \left(  s,t\right)  }%
{t}=u\left(  s\right)  ,
\]
which means that
\[
\beta^{u}\left(  \chi_{\lbrack0,s)}\right)  =\max \left(  u\left(  s\right)
,\sup_{0\leq t\leq s}u\left(  t\right)  \right)  .
\]
In this case we see that $\beta^{u}\in \mathcal{S}_{c}$ if and only if
$\lim_{t\rightarrow0^{+}}u\left(  t\right)  =0$.

It is clear that if $f:(0,1]\rightarrow \lbrack0,\infty)$ is decreasing, then
$\left \Vert f\right \Vert _{t}=\frac{1}{t}\int_{0}^{t}f\left(  x\right)  dx$.
So $f\in \mathcal{L}^{\beta^{u}}\left(  \mathbb{T}\right)  $ if and only if%
\[
\sup_{0<t\leq1}\frac{u\left(  t\right)  }{t}\int_{0}^{t}f\left(  x\right)
dx<\infty.
\]
It follows from the continuity theorem that, for decreasing $f$, that $f\in
L^{\beta^{u}}\left(  \mathbb{T}\right)  $ if and only if%
\[
\lim_{s\rightarrow0}\sup_{0<t\leq s}\frac{u\left(  t\right)  }{t}\int_{0}%
^{t}f\left(  x\right)  dx=0.
\]
So, for example, when $u\left(  t\right)  =\sqrt{t},$ we see that $f\left(
t\right)  =\frac{1}{2\sqrt{t}}\in \mathcal{L}^{\beta^{u}}\left(  \mathbb{T}%
\right)  $ but not in $L^{\beta^{u}}\left(  \mathbb{T}\right)  $.

If $\varphi \left(  t\right)  =t/u\left(  t\right)  $ with $\varphi \left(
0\right)  =0$ is concave and increasing, $\varphi \left(  1\right)  =1$ and
$\lim_{t\rightarrow0^{+}}u\left(  t\right)  =\lim_{t\rightarrow0^{+}}\frac
{t}{\varphi \left(  t\right)  }=0,$ then $\beta^{u}$ is called the
\emph{Marcinkiewicz norm} on $L^{\infty}\left(  \mathbb{T}\right)  $
corresponding to $\varphi$, and these norms are continuous but not strongly
continuous, i.e., $L^{\beta^{u}}(\mathbb{T})\neq \mathcal{L}^{\beta^{u}%
}(\mathbb{T})$. For example, if $u\left(  t\right)  =\sqrt{t}$, $\beta^{u}$ is
such a norm.
\end{remark}

We now prove our generalized dominated convergence theorem. The generalization
is in two senses. The first is by extending from the $L^{p}$-norms to
continuous rotationally symmetric norms and the second is from greatly
extending the notion of dominance. Note that if $f,g:\mathbb{T}\rightarrow
\mathbb{C}$ are measurable, then $\left \vert f\right \vert \leq \left \vert
g\right \vert $ if and only if there is a function $u\in L^{\infty}\left(
\mathbb{T}\right)  $ with $\left \Vert u\right \Vert _{\infty}\leq1$ such that
$f=ug$.

\begin{theorem}
\label{DCT}(Dominated Convergence Theorem) Suppose $\alpha$ is a continuous
rotationally symmetric norm. Let
\[
G_{\alpha}=\left \{  \varphi \in MP\left(  \mathbb{T}\right)  :\forall h\in
L^{\infty}\left(  \mathbb{T}\right)  ,\text{ }\alpha \left(  h\circ
\varphi \right)  =\alpha \left(  h\right)  \right \}  .
\]
Suppose $g\in L^{\alpha}\left(  \mathbb{T}\right)  $, let
\[
K=co\left(  \left \{  \left \vert g\circ \varphi \right \vert u:\left \Vert
u\right \Vert _{\infty}\leq1,\text{ }\varphi \in G_{\alpha}\right \}  \right)  ,
\]
and let $\overline{K}^{m}$ denote the closure of $K$ in the topology of
convergence in measure.

Suppose $\left \{  f_{n}\right \}  $ is a sequence with $\left \vert
f_{n}\right \vert $ in $\overline{K}^{m}$ $\left(  n\in \mathbb{N}\right)  $ and
such that $f_{n}\rightarrow f$ in measure.

Then

\begin{enumerate}
\item $f\in L^{\alpha}\left(  \mathbb{T}\right)  ,$ and

\item $\alpha(f_{n}-f)\rightarrow0.$
\end{enumerate}
\end{theorem}

\begin{proof}
First note that $h\in K$ if and only if $\left \vert h\right \vert \in K$ since
$h=\left \vert h\right \vert e^{iArg\left(  h\right)  }$ and $\left \vert
h\right \vert =he^{-iArg\left(  h\right)  }$. Thus $\left \{  f_{n}\right \}  $
is in $\overline{K}^{m}$ and%
\[
K=co\left(  \left \{  \left(  g\circ \varphi \right)  u:\left \Vert u\right \Vert
_{\infty}\leq1,\text{ }\varphi \in G_{\alpha}\right \}  \right)  .
\]
Suppose $\varepsilon>0$. It follows from Theorem \ref{GCT} that there is a
$\delta>0$ such that, for every measurable $E\subseteq \mathbb{T}$, we have
$m\left(  E\right)  <\delta \Rightarrow \alpha \left(  \chi_{E}\right)
<\varepsilon/3$. Suppose $h\in K$. Then there are functions $u_{1}%
,\ldots,u_{s}\in$ \textrm{ball}$\left(  L^{\infty}\left(  \mathbb{T}\right)
\right)  $ and $\varphi_{1},\ldots,\varphi_{s}\in G_{\alpha}$ and $0\leq
t_{1},\ldots,t_{s}\leq1$ with $\sum_{1\leq k\leq s}t_{k}=1$, such that
\[
h=\sum_{k=1}^{s}t_{k}u_{k}\left(  g\circ \varphi_{k}\right)  .
\]
If $m\left(  E\right)  <\delta$, then $m\left(  \varphi_{k}\left(  E\right)
\right)  =m\left(  E\right)  <\delta$ for $1\leq k\leq s$. Hence, we have%
\begin{align*}
\alpha \left(  h\chi_{E}\right)   &  \leq \sum_{k=1}^{s}t_{k}\left \Vert
u_{k}\right \Vert _{\infty}\alpha \left(  \left(  g\circ \varphi_{k}\right)
\left(  \left[  \chi_{E}\circ \varphi_{k}^{-1}\right]  \circ \varphi_{k}\right)
\right) \\
&  \leq \sum_{k=1}^{s}t_{k}\alpha \left(  \left(  g\chi_{\varphi_{k}\left(
E\right)  }\right)  \circ \varphi_{k}\right)  =\sum_{k=1}^{s}t_{k}\alpha \left(
g\chi_{\varphi_{k}\left(  E\right)  }\right) \\
&  <\sum_{k=1}^{s}t_{k}\varepsilon/3=\varepsilon/3.
\end{align*}

\textbf{Case 1.} For each $n\in \mathbb{N}$, $f_{n}\in K$. Since $f_{n}%
\rightarrow f$ in measure and $\delta>0,$ there is an $N\in \mathbb{N}$ such
that $n,k\geq N$ implies that if $E_{k,n}=\left \{  z\in \mathbb{T}:\left \vert
f_{n}\left(  z\right)  -f_{k}\left(  z\right)  \right \vert \geq \varepsilon
/37\right \}  ,$ which implies%
\begin{align*}
\alpha \left(  f_{k}-f_{n}\right)   &  \leq \alpha \left(  \left(  f_{k}%
-f_{n}\right)  \chi_{E_{k,n}}\right)  +\alpha \left(  \left(  f_{k}%
-f_{n}\right)  \chi_{\mathbb{T}\backslash E_{k,n}}\right) \\
&  \leq \alpha \left(  f_{k}\chi_{E_{k,n}}\right)  +\alpha \left(  f_{n}%
\chi_{E_{k,n}}\right)  +\alpha \left(  \left(  \varepsilon/3\right)
\chi_{\mathbb{T}\backslash E_{k,n}}\right) \\
&  <3\varepsilon/3=\varepsilon.
\end{align*}
It follows from the fact that $\varepsilon>0$ was arbitrary, then $\left \{
f_{n}\right \}  $ is $\alpha$-Cauchy, so there is an $F\in L^{\alpha}\left(
\mathbb{T}\right)  $ such that $\alpha \left(  f_{n}-F\right)  \rightarrow0$.
Since $\left \Vert f_{n}-F\right \Vert _{1}\leq \alpha \left(  f_{n}-F\right)  $,
we see that $f_{n}\rightarrow F$ in measure, which implies $F=f$ .

\textbf{Case 2.} The general case. Since each $f_{n}\in \overline{K}^{m}$,
$f_{n}$ is a limit in measure of a sequence in $K$, and it follows from Case 1
that $f_{n}$ is an $\alpha$-limit of a sequence in $K$. Hence, for each
$n\in \mathbb{N}$, there is an $h_{n}\in K$ such that $\alpha \left(
f_{n}-h_{n}\right)  <1/n$. Hence $f_{n}-h_{n}\rightarrow0$ in measure, which
implies $h_{n}=f_{n}-\left(  f_{n}-h_{n}\right)  \rightarrow f$ in measure,
and it follows from Case 1 that $f\in L^{\alpha}\left(  \mathbb{T}\right)  $
and $\alpha \left(  h_{n}-f\right)  \rightarrow0.$ But
\[
\alpha \left(  f_{n}-f\right)  \leq \alpha \left(  f_{n}-h_{n}\right)
+\alpha \left(  h_{n}-f\right)  \rightarrow0,
\]
and the proof is complete.
\end{proof}

\begin{remark}
We have a few remarks about our dominated convergence theorem.

\begin{enumerate}
\item The restriction $f\in \overline{K}^{m}$ is much more general than
$\left \vert f\right \vert \leq \left \vert g\right \vert $. For example, suppose
$\alpha=\left \Vert \cdot \right \Vert _{1}$, $E\subseteq \mathbb{T}$ is an arc
with $m\left(  E\right)  =1/n$ and $g=n\chi_{E}$. Then if $\varphi_{k}\left(
z\right)  =ze^{k2\pi i/n}$, then $\sum_{k=1}^{n}\frac{1}{n}g\circ \varphi
_{k}=1$. Hence $1\in K$, but any $f$ with $\left \vert f\right \vert
\leq \left \vert g\right \vert $ must be zero off $E$.

\item Theorem \ref{DCT} and Theorem \ref{GCT} remain true if we replace
$\left(  \mathbb{T},m\right)  $ in Theorem \ref{DCT} and Theorem \ref{GCT}
with any finite measure space $\left(  \Omega,\mu \right)  $ and $\alpha$ with
any norm on $L^{\infty}\left(  \mu \right)  $such that

\begin{enumerate}
\item $\alpha \left(  1\right)  =1$,

\item $\alpha \left(  f\right)  =\alpha \left(  \left \vert f\right \vert \right)
$ for every $f\in L^{\infty}\left(  \mu \right)  $, and

\item Whenever $\left \{  f_{n}\right \}  $ is a sequence in $L^{\infty}\left(
\mu \right)  $ and $\alpha \left(  f_{n}\right)  \rightarrow0$, we have
$f_{n}\rightarrow0$ in measure.

\item For every $\varepsilon>0$ there is a $\delta>0$ such that $\mu \left(
E\right)  <\delta \Rightarrow \alpha \left(  \chi_{E}\right)  <\varepsilon$.
\end{enumerate}
\end{enumerate}
\end{remark}

The next corollary follows immediately from the dominated convergence theorem
on $L^{\alpha}(\mathbb{T}).$

\begin{corollary}
(Monotone Convergence Theorem) Suppose $\alpha \in{\mathcal{R}}_{c}$ and $f\in
L^{\alpha}\left(  \mathbb{T}\right)  .$ If $0\leq f_{1}\leq f_{2}\leq \cdots$
and $f_{n}(x)\rightarrow f\left(  x\right)  $ a.e. $\left(  m\right)  ,$ then
$\alpha \left(  f-f_{n}\right)  \rightarrow0$.
\end{corollary}

It is clear that in the preceding corollary, we have $\alpha \left(
f_{n}\right)  \rightarrow \alpha \left(  f\right)  .$ That this holds for
arbitrary measurable functions was proved in part (6) in Proposition
\ref{prop1} .

\bigskip

\begin{corollary}
$\left(  \mathcal{R}_{c},\mathcal{T}\left(  \mathcal{R}_{c}\right)  \right)  $
is separable and metrizable.
\end{corollary}

\begin{proof}
Let $\mathcal{F}$ be the linear span of $\left \{  z^{n}:n\in \mathbb{Z}%
\right \}  $ over the field $\mathbb{Q}+i\mathbb{Q}$ of complex-rational
numbers. Then, as in the proof of part (2) of Lemma \ref{2}, we only need to
show that $\mathcal{T}\left(  \mathcal{R}_{c}\right)  $ is the weak topology
induced by the set $\left \{  \pi_{f}:f\in \mathcal{F}\right \}  $. Suppose
$\left \{  \alpha_{\lambda}\right \}  $ is a net in $\mathcal{S}_{c}$ and
$\alpha \in \mathcal{S}_{c}$ and $\alpha_{\lambda}\left(  f\right)
\rightarrow \alpha \left(  f\right)  $ for every $f\in \mathcal{F}$. Since
$\mathcal{S}$ is compact, there is a subnet $\left \{  \alpha_{\lambda_{k}%
}\right \}  $ converging pointwise to $\beta \in \mathcal{R}$. We only need to
show that $\beta=\alpha$. We know that $\beta \left(  f\right)  =\alpha \left(
f\right)  $ for every $f\in \mathcal{F}$, and since the uniform closure of
$\mathcal{F}$ is $C\left(  \mathbb{T}\right)  ,$ we conclude $\beta \left(
f\right)  =\alpha \left(  f\right)  $ for every $f\in C\left(  \mathbb{T}%
\right)  $. Suppose $K$ is a closed subset of $\mathbb{T}$. Then there is a
function $h_{K}:\mathbb{T}\rightarrow \left[  0,1\right]  $ such that
$K=h_{K}^{-1}\left(  \left \{  1\right \}  \right)  $. Then $h_{K}^{n}%
\downarrow \chi_{K}$ on $\mathbb{T}$. Suppose $\left \{  K_{1},\ldots
,K_{m}\right \}  $ is a disjoint family of closed subsets of $\mathbb{T}$ and
$0\leq a_{1},\ldots,a_{m}$ and suppose $s=\sum_{j=1}^{m}a_{j}\chi_{K_{j}}$.
Then $f_{n}=\sum_{j=1}^{n}a_{j}h_{K_{j}}^{n}\in C\left(  \mathbb{T}\right)  $
and $f_{n}\downarrow s$. It follows that
\[
\beta \left(  s\right)  \leq \limsup_{n\rightarrow \infty}\beta \left(
f_{n}\right)  =\lim_{n\rightarrow \infty}\alpha \left(  f_{n}\right)
=\alpha \left(  s\right)  ,
\]
with the last equality following from our dominated convergence theorem. If
$u=\sum_{j=1}^{m}a_{k}\chi_{E_{k}}$, it follows from the regularity of $m$
that we can choose a sequence $\left \{  s_{n}\right \}  $ of simple functions
of the form of $s$ so that
\[
0\leq s_{1}\leq s_{2}\leq \cdots
\]
and $s_{n}\left(  z\right)  \rightarrow u\left(  z\right)  $ a.e. $\left(
m\right)  $. It follows from part (6) of Lemma \ref{prop1} that
\[
\beta \left(  u\right)  =\lim_{n\rightarrow \infty}\beta \left(  s_{n}\right)
\leq \lim_{n\rightarrow \infty}\alpha \left(  s_{n}\right)  =\alpha \left(
s\right)  .
\]
Since the simple functions are $\left \Vert \cdot \right \Vert _{\infty}$-dense
in $L^{\infty}\left(  \mathbb{T}\right)  $, we conclude that $\beta \left(
f\right)  \leq \alpha \left(  f\right)  $ for all $f\in L^{\infty}\left(
\mathbb{T}\right)  $. It now follows that $\beta \in \mathcal{S}_{c}$, and
reversing the roles of $\alpha$ and $\beta$ in the above arguments, we get
$\alpha \leq \beta.$ Hence $\alpha=\beta$.
\end{proof}

\begin{proposition}
The following statements are equivalent for a continuous rotationally
symmetric norm $\alpha$:

\begin{enumerate}
\item The General Continuity Theorem is true in ${\mathcal{L}}^{\alpha
}(\mathbb{T});$

\item The Dominated Continuity Theorem is true in ${\mathcal{L}}^{\alpha
}(\mathbb{T});$

\item $\alpha$ is strongly continuous, i.e., $\mathcal{L}^{\alpha}\left(
\mathbb{T}\right)  =L^{\alpha}\left(  \mathbb{T}\right)  .$
\end{enumerate}
\end{proposition}

\begin{proof}
$(1)\Rightarrow(2)$ It is clear from Theorem \ref{DCT}.

$(2)\Rightarrow(3)$ Suppose $f\in{\mathcal{L}}^{\alpha}(\mathbb{T}).$ Then
$0\leq|f|\in{\mathcal{L}}^{\alpha}(\mathbb{T})$ and there is a sequence of
simple functions $0\leq s_{1}\leq s_{2}\leq \cdots$ such that $|s_{n}|\leq|f|$
and $s_{n}(w)\rightarrow|f|(w)$ for every $w\in \mathbb{T}.$ It follows from
Theorem \ref{DCT} that $\alpha(s_{n}-|f|)\rightarrow0,$ and thus $f\in
L^{\alpha}(\mathbb{T}).$ This implies ${\mathcal{L}}^{\alpha}(\mathbb{T}%
)=L^{\alpha}(\mathbb{T}).$

$(3)\Rightarrow(1)$ It is obvious from Theorem \ref{GCT}.
\end{proof}

\section{Convolution product on $L^{\alpha}(\mathbb{T},X)$}

Suppose $X$ is a separable Banach space and $\alpha$ is a continuous
rotationally symmetric norm on $\mathbb{T}$. Suppose $f:\mathbb{T}\rightarrow
X$ is a function. If $w\in \mathbb{T}$, we define, as in the scalar case,
$f_{w}:{\mathbb{T}}\rightarrow X$ by $f_{w}(z)=f(\overline{w}z).$ We also
define $\left \vert f\right \vert :\mathbb{T}\rightarrow \lbrack0,\infty)$ by
\[
\left \vert f\right \vert \left(  z\right)  =\left \Vert f\left(  z\right)
\right \Vert ,
\]
i.e., $\left \vert f\right \vert =\left \Vert \cdot \right \Vert \circ f$.

For any rotationally symmetric norm on $L^{\infty}\left(  \mathbb{T}\right)  $
we define
\[
\alpha \left(  f\right)  =\alpha \left(  \left \Vert \cdot \right \Vert \circ
f\right)  =\alpha \left(  \left \vert f\right \vert \right)  ,
\]
and we define%

\[
L^{\alpha}\left(  \mathbb{T},X\right)  =\left \{  f|f:\mathbb{T}\rightarrow
X\text{ is measurable and }\left \vert f\right \vert \in L^{\alpha}\left(
\mathbb{T}\right)  \} \right \}  .
\]
\bigskip It is easy to show that $L^{\alpha}\left(  \mathbb{T},X\right)  $ is
a Banach space with the norm $\alpha$.

We also define $C\left(  \mathbb{T},X\right)  $ to be the set of all
continuous functions from $\mathbb{T}$ to $X$.

\begin{lemma}
\label{cont lemma} If $\alpha \in \mathcal{R}_{c}$ and $X$ is a separable Banach
space, then

\begin{enumerate}
\item $L^{\alpha}\left(  \mathbb{T},X\right)  $ is the closed linear span of
elements of the form $h\left(  z\right)  =\chi_{E}\left(  z\right)  x_{0}$
with $E\subset \mathbb{T}$ and $x_{0}\in X;$

\item $L^{\alpha}\left(  \mathbb{T},X\right)  $ is the closed linear span of
elements of the form $h\left(  z\right)  =f\left(  z\right)  x_{0}$ with $f\in
C\left(  \mathbb{T}\right)  $ and $x_{0}\in X;$

\item $C\left(  \mathbb{T},X\right)  ^{-\alpha}=L^{\alpha}\left(
\mathbb{T},X\right)  ;$

\item For every $f\in L^{\alpha}\left(  \mathbb{T},X\right)  $, $\lim
_{m\left(  E\right)  \rightarrow0}\alpha \left(  \chi_{E}f\right)  =0;$

\item Theorem \ref{DCT} is true when $L^{\alpha}\left(  \mathbb{T}\right)  $
is replaced with $L^{\alpha}\left(  \mathbb{T},X\right)  $.
\end{enumerate}
\end{lemma}

\begin{proof}
$(1)$ Suppose $\varepsilon>0$ and $f\in L^{\alpha}\left(  \mathbb{T},X\right)
$. Suppose $n\in \mathbb{N}$. Since $X$ is separable, we can find a disjoint
collection $\left \{  E_{n1},E_{n2},\ldots \right \}  $ of nonempty Borel subsets
whose union is $X$ such that for every $k\geq1$ and every $x,y\in E_{nk}$,
$\left \Vert x-y\right \Vert \leq1/n$. Let $F_{nk}=f^{-1}\left(  E_{nk}\right)
$ and choose $x_{nk}\in E_{nk}$. Define $g_{n}=\sum_{k=1}^{\infty}x_{nk}%
\chi_{F_{nk}}:\mathbb{T}\rightarrow X$. Then $g_{n}$ is measurable and
$\left \Vert g_{n}\left(  z\right)  -f\left(  z\right)  \right \Vert \leq1/n$
for every $z\in \mathbb{T}$. Hence $g_{n}-f\in L^{\infty}\left(  \mathbb{T}%
,X\right)  \subset L^{\alpha}\left(  \mathbb{T},X\right)  $ and $g_{n}%
=f+\left(  g_{n}-f\right)  \in L^{\alpha}\left(  \mathbb{T},X\right)  .$ Since
$\alpha \leq \left \Vert \cdot \right \Vert _{\infty},$ it is clear that
$\alpha \left(  g_{n}-f\right)  \leq1/n\rightarrow0$. Thus the functions of the
form $g=\sum_{k=1}^{\infty}x_{k}\chi_{F_{k}}$ are dense in $L^{\alpha}\left(
\mathbb{T},X\right)  $. But $m\left(  \cup_{k=N+1}^{\infty}F_{k}\right)
\rightarrow0$ as $N\rightarrow \infty$ implies
\[
\alpha \left(  \left \Vert \cdot \right \Vert \circ \left[  g-\sum_{k=1}^{N}%
x_{k}\chi_{F_{k}}\right]  \right)  =\alpha \left(  \left \Vert \cdot \right \Vert
\circ \left[  g\chi_{\cup_{k=N+1}^{\infty}F_{k}}\right]  \right)
\rightarrow0,
\]

which implies $g$ is the limit in $L^{\alpha}\left(  \mathbb{T},X\right)  $ of
$\sum_{k=1}^{N}x_{k}\chi_{F_{k}}$, and these are in the linear span of
functions of the form $x_{0}\chi_{E}\left(  z\right)  $.

$(2)$ If $E\subset \mathbb{T}$ is a Borel set, then $\chi_{E}\in L^{\alpha
}\left(  \mathbb{T}\right)  $, and since $C\left(  \mathbb{T}\right)  $ is
dense in $L^{\alpha}\left(  \mathbb{T}\right)  ,$ there is a sequence
$\left \{  f_{n}\right \}  $ in $C\left(  \mathbb{T}\right)  $ such that
$\alpha \left(  f_{n}-\chi_{E}\right)  \rightarrow0$. If $h_{n}\left(
z\right)  =f_{n}\left(  z\right)  x_{0},$ then $h_{n}\in C\left(
\mathbb{T},X\right)  $ and
\[
\alpha \left(  h_{n}-x_{0}\chi_{E}\right)  =\alpha \left(  f_{n}-\chi
_{E}\right)  \left \Vert x_{0}\right \Vert \rightarrow0.
\]
Hence $C\left(  \mathbb{T},X\right)  ^{-\alpha}$ contains the closed linear
span of the functions of the form $x_{0}\chi_{E}$.

$(3)$ This easily follows from (1) and (2).

$\left(  4\right)  $ If $f\in L^{\alpha}\left(  \mathbb{T},X\right)  ,$ then
$\left \vert f\right \vert \in L^{\alpha}\left(  \mathbb{T}\right)  $, and
$\alpha \left(  \left \vert \chi_{E}f\right \vert \right)  =\alpha \left(
\chi_{E}\left \vert f\right \vert \right)  ,$ the result easily follows from
Theorem \ref{GCT}.

$\left(  5\right)  $ We first note that $0\leq \left \vert \left \vert
f_{n}\right \vert -\left \vert f\right \vert \right \vert \leq \left \vert
f_{n}-f\right \vert \rightarrow0$ in measure implies $\left \vert f_{n}%
\right \vert \rightarrow \left \vert f\right \vert $ in measure. Using Theorem
\ref{DCT} we get $\left \vert f\right \vert \in L^{\alpha}\left(  \mathbb{T}%
\right)  ,$ which implies $f\in L^{\alpha}\left(  \mathbb{T},X\right)  $. We
then replace $g$ with $\left \vert g\right \vert +\left \vert f\right \vert $ and
use the fact that $\left \vert f_{n}-f\right \vert \rightarrow0$ in measure to
apply Theorem \ref{DCT} to get $\alpha \left(  f_{n}-f\right)  =\alpha \left(
\left \vert f_{n}-f\right \vert \right)  \rightarrow0.$
\end{proof}

\bigskip

For $f\in C\left(  \mathbb{T},X\right)  $ we define%

\[
\int_{\mathbb{T}}fdm=\lim_{n\rightarrow \infty}\frac{1}{2^{n}}\sum_{k=0}%
^{2^{n}-1}f\left(  e^{2\pi ik/2^{n}}\right)  \in X.
\]
The uniform continuity easily implies that the limit converges and it easily
follows that%

\[
\left \Vert \int_{\mathbb{T}}fdm\right \Vert \leq \int_{\mathbb{T}}\left \Vert
f\left(  z\right)  \right \Vert dm\left(  z\right)  =\left \Vert f\right \Vert
_{1}\leq \left \Vert f\right \Vert _{\infty}.
\]
Hence $\mathcal{I}:C\left(  \mathbb{T},X\right)  \rightarrow X$ defined by
$\mathcal{I}\left(  f\right)  =\int_{\mathbb{T}}fdm$ is a linear mapping and
is continuous when we give $C\left(  \mathbb{T},X\right)  $ either $\left \Vert
\cdot \right \Vert _{1}$ or $\left \Vert \cdot \right \Vert _{\infty}$. Since
$\mathbb{C}\left(  \mathbb{T},X\right)  $ is dense in $L^{1}\left(
\mathbb{T},X\right)  ,$ we see that $\mathcal{I}$ has a unique continuous
norm-one linear extension
\[
\mathcal{I}^{\prime}:L^{1}\left(  \mathbb{T},X\right)  \rightarrow X,
\]
and we will use $\int_{\mathbb{T}}fdm$ to denote $\mathcal{I}^{\prime}\left(
f\right)  $.

\begin{lemma}
\label{lemma4.2} Suppose $\alpha$ is a continuous rotationally symmetric norm
on $\mathbb{T}$ and $f\in L^{\alpha}\left(  \mathbb{T},X\right)  $. Then the
mapping $G:\mathbb{T}\rightarrow L^{\alpha}\left(  \mathbb{T},X\right)  $
defined by $G\left(  w\right)  =f_{w}$ is $\alpha$-continuous.
\end{lemma}

\begin{proof}
Clearly the set of $f\in L^{\alpha}\left(  \mathbb{T},X\right)  $ for which
the lemma is true is a closed linear subspace of $L^{\alpha}\left(
\mathbb{T},X\right)  $. Since $X$ is separable, the functions $f\in L^{\alpha
}\left(  \mathbb{T},X\right)  $ with countable range are dense in $L^{\alpha
}\left(  \mathbb{T},X\right)  $. Since $\alpha$ is continuous, the set of
simple functions is dense in $L^{\alpha}\left(  \mathbb{T},X\right)  $.
Suppose $E\subseteq \mathbb{T}$ is measurable. Then there is an open subset $U$
of $\mathbb{T}$ such that $E\subseteq U$ and $\alpha \left(  \chi_{E}-\chi
_{U}\right)  $ is arbitrarily small. However, $U$ is a countable disjoint
union of open arcs, so there is a finite disjoint union $V$ of open arcs such
that $\alpha \left(  \chi_{U}-\chi_{V}\right)  $ is arbitrarily small.
Moreover, every open arc is the disjoint union (a.e.) of at most $8$ arcs with
length at most $\pi/4.$ It follows that $L^{\alpha}\left(  \mathbb{T}%
,X\right)  $ is the closed linear span of functions $f=x\chi_{I}$ with $x\in
X$ and $I$ an arc with length at most $\pi/4$. It is easy to see from the
continuity of $\alpha$ that the lemma is true for such functions $f$.
\end{proof}

If $f\in L^{\alpha}({\mathbb{T}},X)$ and $g\in L^{1}(\mathbb{T}),$ we define
the convolution product $f\ast g:\mathbb{T}\rightarrow X$ for almost every
$z\in \mathbb{T}$ by%
\[
\left(  f\ast g\right)  (z)=\int_{\mathbb{T}}f(\bar{w}z)g(w)dm(w)=\int
_{\mathbb{T}}f_{w}(z)g(w)dm(w).
\]
Note that the changes of variable $w\mapsto wz$ or $w\mapsto \bar{w}$ do not
change the integral, so%
\[
\left(  f\ast g\right)  (z)=\int_{\mathbb{T}}f(w)g(\bar{w}z)dm(w).
\]
It is well-known that, if $f\in L^{p}({\mathbb{T}})$ for $1\leq p<\infty$ and
$g\in L^{1}({\mathbb{T}}),$ then
\[
f\ast g\in L^{p}({\mathbb{T}})\  \text{and }\Vert f\ast g\Vert_{p}\leq \Vert
f\Vert_{p}\Vert g\Vert_{1}.
\]
Our object is to prove the following extension, which is a more general result
in $L^{\alpha}({\mathbb{T}},X).$

\begin{theorem}
\label{mainthm} Suppose $\alpha$ is a continuous rotationally symmetric norm
and $X$ is a separable Banach space. If $f\in L^{\alpha}({\mathbb{T}},X)$ and
$g\in L^{1}({\mathbb{T}}),$ then
\[
f\ast g\in L^{\alpha}({\mathbb{T}},X)\text{ and }\alpha(f\ast g)\leq
\alpha(f)\Vert g\Vert_{1}.
\]

\end{theorem}

\begin{proof}
Suppose $f\in L^{\alpha}({\mathbb{T}},X)$ and $g\in L^{1}(\mathbb{T})$ with
$g\geq0.$ Then for arbitrary $z\in \mathbb{T},$
\begin{align*}
\left(  f\ast g\right)  (z)  &  =\int_{\mathbb{T}}f(\bar{w}z)g(w)dm(w)\\
&  =\int_{\mathbb{T}}f_{w}(z)\frac{|g(w)|}{\Vert g\Vert_{1}}dm(w)\Vert
g\Vert_{1}\\
&  =\int_{\mathbb{T}}f_{w}(z)d\mu(w)\Vert g\Vert_{1},
\end{align*}
where $\mu=\frac{|g(w)|}{\Vert g\Vert_{1}}m$ is a probability measure. The
convolution can be expressed as $f\ast g=\int_{\mathbb{T}}f_{w}d\mu(w)\Vert
g\Vert_{1}.$ It follows from Lemma \ref{lemma4.2} that
\[
f\ast g=\int_{\mathbb{T}}f_{w}d\mu(w)\Vert g\Vert_{1}\in L^{\alpha
}({\mathbb{T}},X),
\]
and
\begin{align*}
\alpha(f\ast g)  &  =\alpha(\int_{\mathbb{T}}f_{w}d\mu(w)\Vert g\Vert_{1})\\
&  \leq \int_{\mathbb{T}}\alpha(f_{w})d\mu(w)\Vert g\Vert_{1}\\
&  =\int_{\mathbb{T}}\alpha(f)d\mu(w)\Vert g\Vert_{1}=\alpha(f)\Vert
g\Vert_{1}.
\end{align*}
Next suppose $g\in L^{1}({\mathbb{T}}).$ Then there are $g_{1},g_{2}%
,g_{3},g_{4}\geq0$ in $L^{1}({\mathbb{T}})$ such that $g$ can be written as%
\[
g=\operatorname{Re}g^{+}-\operatorname{Re}g^{-}+i(\operatorname{Im}%
g^{+}-\operatorname{Im}g^{-})=g_{1}-g_{2}+ig_{3}-ig_{4},
\]
and hence%
\begin{align*}
f\ast g  &  =\int_{\mathbb{T}}f_{w}d\mu(w)\Vert g_{1}\Vert_{1}-\int
_{\mathbb{T}}f_{w}d\mu(w)\Vert g_{2}\Vert_{1}\\
&  +i\int_{\mathbb{T}}f_{w}d\mu(w)\Vert g_{3}\Vert_{1}-i\int_{\mathbb{T}}%
f_{w}d\mu(w)\Vert g_{4}\Vert_{1}\\
&  \in L^{\alpha}({\mathbb{T}},X),
\end{align*}
Since the definition of convolution product implies $|f\ast g|\leq|f|\ast|g|,$
it follows that
\[
\alpha(f\ast g)=\alpha(|f\ast g|)\leq \alpha(|f|\ast|g|)\leq \alpha(|f|)\Vert
g\Vert_{1}=\alpha(f)\Vert g\Vert_{1}.
\]

\end{proof}

\begin{definition}
\label{appx id def} An {approximate identity} in $L^{1}(\mathbb{T})$ is a net
$\{ \phi_{\lambda}\}$ in $L^{1}(\mathbb{T})$ with the properties:

\begin{enumerate}
\item $\phi_{\lambda}\geq0$ for all $\lambda;$

\item $\int_{\mathbb{T}}\phi_{\lambda}dm=1$ for all $\lambda;$

\item For every subset $E$ of $\mathbb{T}$ that is the complement of an open
neighborhood of $1,$ the net $(\phi_{\lambda})$ converges uniformly to $0$ on
$E.$
\end{enumerate}
\end{definition}

\begin{example}
\label{example}

\begin{enumerate}
\item The \emph{Poisson kernel} is defined by%

\[
P_{r}(e^{it})=\sum_{n=-\infty}^{\infty}r^{|n|}e^{int}%
\]
for $r\in \lbrack0,1).$ Given a function $F$ defined on $\mathbb{T}$, we call
the function $f$ defined on $\mathbb{D}$ by $f\left(  re^{it}\right)  =\left(
F\ast P_{r}\right)  \left(  e^{it}\right)  $ the \emph{Poisson integral} of
$F.$

\item The \emph{Fejer kernel} is defined by
\[
K_{n}=\frac{D_{0}+D_{1}+\cdots+D_{n}}{n+1},
\]
where $D_{n}(z)=\sum_{k=-n}^{n}z^{k}$ for all $n\geq0$. Note that, for any $F$
defined on $\mathbb{T}$ and any $k\in \mathbb{Z}$, we have
\[
\left(  F\ast z^{k}\right)  \left(  z\right)  =\left(  z^{k}\ast F\right)
\left(  z\right)  =\int_{\mathbb{T}}F\left(  w\right)  \left(  {\bar{w}%
}z\right)  ^{k}dm\left(  w\right)  =\left[  \int_{\mathbb{T}}F\left(
w\right)  \bar{w}^{k}dm\left(  w\right)  \right]  z^{k}.
\]
Hence $\left(  F\ast K_{n}\right)  \left(  z\right)  $ has the form
$\sum_{k=-n}^{n}c_{n,k}z^{k}.$ If we define the $n^{th}$ Fourier coefficient
\[
\hat{F}\left(  n\right)  =\int_{\mathbb{T}}F(w)\bar{w}^{n}dm(w),
\]
then
\[
c_{n,k}=\frac{\left(  n+1-\left \vert k\right \vert \right)  }{n+1}\hat
{F}(k)\text{ for }\left \vert k\right \vert \leq n,
\]
and
\[
\lim_{n\rightarrow \infty}c_{n,k}=\hat{F}(k)
\]
for $k\in \mathbb{Z}$.
\end{enumerate}
\end{example}

\begin{theorem}
\label{appx cont thm} Suppose $\alpha \in{\mathcal{R}}_{c}$ and $X$ is a
separable Banach space. If $f\in C({\mathbb{T}},X)$ and $\{ \phi_{\lambda}\}$
is an approximate identity, then $\{f\ast \phi_{\lambda}\}$ is a net of
continuous functions that converges uniformly to $f.$
\end{theorem}

\begin{proof}
Let $f_{w}(z)=f(\overline{w}z).$ Since $f$ is continuous and $\mathbb{T}$ is
compact, $f$ is uniformly continuous on $\mathbb{T},$ and therefore $\Vert
f-f_{w}\Vert_{\infty}\rightarrow0$ as $w\rightarrow1^{-}.$ Hence for
$\epsilon>0$ there exists a neighborhood $\mathbb{U}$ of $1$ in $\mathbb{T}$
such that $\Vert f-f_{w}\Vert \infty<\frac{\epsilon}{2}$ whenever
$w\in \mathbb{U}.$ Let $E={\mathbb{T}}\backslash{\mathbb{U}},$ and use
properties (1) and (3) in Definition \ref{appx id def} to choose $\lambda_{0}$
such that $\lambda \succcurlyeq \lambda_{0}$ implies $0\leq \phi_{\lambda}%
<\frac{\epsilon}{4\Vert f\Vert_{\infty}}$ on $E.$ Then for arbitrary
$z\in \mathbb{T}$ and $\lambda \succcurlyeq \lambda_{0}$ we have%
\begin{align*}
\Vert f(z)-f\ast \phi_{\lambda}(z)\Vert &  \leq \int_{\mathbb{T}}\Vert
f(z)-f(\overline{w}z)\Vert \phi_{\lambda}(w)dm(w)\\
&  \leq \int_{\mathbb{U}}\Vert f-f_{w}\Vert_{\infty}\phi_{\lambda}%
(w)dm(w)+\int_{\mathbb{E}}2\Vert f\Vert_{\infty}\phi_{\lambda}(w)dm(w)\\
&  <\frac{\epsilon}{2}+2\Vert f\Vert_{\infty}\frac{\epsilon}{4\Vert
f\Vert_{\infty}}=\epsilon.
\end{align*}
Properties $(2)$ was used in the last inequality. Therefore $\{f\ast
\phi_{\lambda}\}$ converges uniformly to $f.$

Similarly, if $z,z_{0}\in \mathbb{T},$ then
\[
\Vert f\ast \phi_{\lambda}(z)-f\ast \phi_{\lambda}(z_{0})\Vert \leq
\int_{\mathbb{T}}\Vert f(\overline{w}z)-f(\overline{w}z_{0})\Vert \phi
_{\lambda}(w)dm(w)\leq \Vert f_{z}-f_{z_{0}}\Vert_{\infty}.
\]
As above, uniform continuity implies $\Vert f_{z}-f_{z_{0}}\Vert_{\infty
}\rightarrow0$ as $z\rightarrow z_{0},$ and this implies the continuity of
$f\ast \phi_{\lambda}.$
\end{proof}

\begin{theorem}
\label{appx id thm} Suppose $\alpha \in{\mathcal{R}}_{c}$ and $X$ is a
separable Banach space. If $f\in L^{\alpha}({\mathbb{T}},X)$ and $\{
\phi_{\lambda}\}$ is an approximate identity, then

\begin{enumerate}
\item $f\ast \phi_{\lambda}\in L^{\alpha}({\mathbb{T}},X),$ and $\alpha
(f\ast \phi_{\lambda})\leq \alpha(f);$

\item $\lim_{\lambda}\alpha(f-f\ast \phi_{\lambda})=0.$
\end{enumerate}
\end{theorem}

\begin{proof}
$(1)$ It is clear that $\Vert \phi_{\lambda}\Vert_{1}=1.$ If $f\in L^{\alpha
}({\mathbb{T}},X),$ then by Theorem \ref{mainthm}, we obtain
\[
f\ast \phi_{\lambda}\in L^{\alpha}(\mathbb{T})\  \  \  \mbox{ and}\  \  \alpha
(f\ast \phi_{\lambda})\leq \alpha(f)\Vert \phi_{\lambda}\Vert_{1}=\alpha(f).
\]

$(2)$ Suppose $f\in L^{\alpha}({\mathbb{T}},X).$ Since $C({\mathbb{T}},X)$ is
dense in $L^{\alpha}({\mathbb{T}},X),$ there is a sequence $\{f_{n}\} \subset
C({\mathbb{T}},X)$ such that $\alpha(f_{n}-f)\rightarrow0.$ It follows from
Theorem \ref{appx cont thm} that $f_{n}\ast \phi_{\lambda}$ converges uniformly
to $f_{n},$ which, together with Theorem \ref{mainthm} and $\alpha
(f_{n}-f)\rightarrow0$ implies%
\begin{align*}
\alpha(f\ast \phi_{\lambda}-f)  &  =\alpha(f\ast \phi_{\lambda}-f_{n}\ast
\phi_{\lambda}+f_{n}\ast \phi_{\lambda}-f_{n}+f_{n}-f)\\
&  \leq \alpha((f-f_{n})\ast \phi_{\lambda})+\alpha(f_{n}\ast \phi_{\lambda
}-f_{n})+\alpha(f_{n}-f)\\
&  \leq \alpha(f-f_{n})\Vert \phi_{\lambda}\Vert_{1}+\alpha(f_{n}\ast
\phi_{\lambda}-f_{n})+\alpha(f_{n}-f)\\
&  \leq \alpha(f-f_{n})+\Vert f_{n}\ast \phi_{\lambda}-f_{n}\Vert_{\infty
}+\alpha(f_{n}-f)\\
&  \rightarrow0.
\end{align*}

\end{proof}

As with Poisson kernel we have the following corollary, which is mostly a
special case of Theorem \ref{appx id thm}.

\begin{corollary}
\label{cor after mainthm} Suppose $\alpha \in{\mathcal{R}}_{c}$ and $X$ is a
separable Banach space. If $f\in L^{\alpha}({\mathbb{T}},X),$ then

\begin{enumerate}
\item $f\ast P_{r}\in L^{\alpha}({\mathbb{T}},X)$ and $\alpha(f\ast P_{r}%
)\leq \alpha(f)$ for $0<r<1;$

\item $\alpha(f\ast P_{r})$ is an increasing function of $r$ on $(0,1);$

\item $\lim_{r\rightarrow1^{-}}\alpha(f\ast P_{r}-f)\rightarrow0;$

\item $\lim_{r\rightarrow1^{-}}\alpha(f\ast P_{r})=\alpha(f).$
\end{enumerate}
\end{corollary}

\begin{proof}
(1) and (3) follows immediately from Theorem \ref{appx id thm}.

(2) If $0\leq r<s<1,$ then choose $q\in(0,1)$ such that $r=sq,$ thus $f\ast
P_{r}=f\ast P_{sq}=f\ast(P_{s}\ast P_{s})=(f\ast P_{s})\ast P_{q}.$ By Theorem
\ref{mainthm}, $\alpha(f\ast P_{r})=\alpha((f\ast P_{s})\ast P_{q})\leq
\alpha(f\ast P_{s})\Vert P_{q}\Vert_{1}=\alpha(f\ast P_{s}).$ This implies
$\alpha(f\ast P_{r})$ is an increasing function of $r$ on $[0,1).$

(4) This follows immediately from (3).
\end{proof}

\section{Hardy classes on the circle and disk}

We will maintain the distinction throughout this section that $F$ and $f$ are
functions defined on $\mathbb{T}$ and $\mathbb{D}$ respectively that are
related by $f$ being the Poisson integral of $F.$ The Hardy spaces
$H^{p}(\mathbb{T})$ were defined as closed subspaces of $L^{p}(\mathbb{T})$
spanned by the set ${\mathcal{P}}_{+}=$\textrm{ span}$\{e_{n}:n\in
\mathbb{N}\}.$ Closure is with respect to the norm topology of $L^{p}%
(\mathbb{T})$ for finite $p$ and the weak* topology for $p=\infty.$ Functions
$F$ in $H^{p}(\mathbb{T})$ can also be characterized by the properties of
belonging to $L^{p}(\mathbb{T})$ and having no nonzero Fourier coefficients of
negative index.

Suppose $\alpha$ is a continuous rotationally symmetric norm. We define%
\[
H^{\alpha}({\mathbb{T}})=\left(  \text{\textrm{span}}\{e_{n}:n\in \mathbb{N}\}
\right)  ^{-\alpha}=\left(  {\mathcal{P}}_{+}\right)  ^{-\alpha},
\]
i.e., $H^{\alpha}({\mathbb{T}})$ is the closure in the $\alpha$-norm of the
set of polynomials in $z.$ Based on the convolution theorem on $L^{\alpha
}(\mathbb{T}),$ we have obtained the corresponding characterization of
$H^{\alpha}(\mathbb{T}).$

\begin{theorem}
\label{thm6.1} Suppose $\alpha$ is a continuous rotationally symmetric norm.
Then
\[
H^{\alpha}({\mathbb{T}})=\{F\in L^{\alpha}({\mathbb{T}}):\hat{F}%
(n)=\int_{\mathbb{T}}F(z)z^{-n}dm(z)=0,\text{\ for \ all }n<0\},
\]
i.e., the functions in $L^{\alpha}(\mathbb{T})$ whose negative Fourier
coefficients vanish.
\end{theorem}

\begin{proof}
Let $M=\{F\in L^{\alpha}({\mathbb{T}}):\hat{F}(n)=0,$ for all$\ n<0\}.$ It is
clear that ${\mathcal{P}_{+}}\subset M$ and $M$ is norm closed, then
$\overline{{\mathcal{P}}_{+}}^{\alpha}=H^{\alpha}({\mathbb{T}})\subset M.$

Conversely, assume $F\in M.$ Then $F\in L^{\alpha}({\mathbb{T}})\subset
L^{1}({\mathbb{T}})$ with $\hat{F}(n)=0$ for all $n<0.$ Since the partial sums
$S_{n}(F)=\sum_{k=-n}^{n}\hat{F}(n)e_{n}=\sum_{k=0}^{n}\hat{F}(n)e_{n}%
\in{\mathcal{P}}_{+}$ for all $n\geq0,$ it follows that the Cesaro means%
\[
\sigma_{n}(F)=\frac{S_{0}(F)+S_{1}(F)+\ldots+S_{n}(F)}{n+1}\in \mathcal{P}%
_{+}.
\]
The definition of the Cesaro means and Corollary \ref{cor after mainthm}
ensure that $\sigma_{n}(F)=F\ast K_{n}\rightarrow F$ in $L^{\alpha
}({\mathbb{T}}),$ thus $F\in \overline{{\mathcal{P}}_{+}}^{\alpha}=H^{\alpha
}({\mathbb{T}}),$ which means $M\subset H^{\alpha}({\mathbb{T}}),$ and
therefore
\[
H^{\alpha}({\mathbb{T}})=\{F\in L^{\alpha}({\mathbb{T}}):\hat{F}%
(n)=\int_{\mathbb{T}}F(z)z^{-n}dm(z)=0,\mbox{\ for\  all}\ n<0\}.
\]

\end{proof}

Recall that
\[
H^{1}({\mathbb{T}})=\{F\in L^{1}({\mathbb{T}}):\hat{F}(n)=\int_{\mathbb{T}%
}F(z)z^{-n}dm(z)=0,\mbox{\ for\  all}\ n<0\},
\]
the following corollary is an immediate consequence.

\begin{corollary}
\label{intersection cor} Suppose $\alpha$ is a continuous rotationally
symmetric norm. Then
\[
H^{\alpha}(\mathbb{T})=L^{\alpha}(\mathbb{T})\cap H^{1}(\mathbb{T}).
\]

\end{corollary}

Based on Corollary \ref{intersection cor}, for each continuous rotationally
symmetric norm $\alpha,$ we define%
\[
\mathcal{H}^{\alpha}\left(  \mathbb{T}\right)  =\mathcal{L}^{\alpha}\left(
\mathbb{T}\right)  \cap H^{1}\left(  \mathbb{T}\right)  ,
\]
or equivalently,
\[
\mathcal{H}^{\alpha}({\mathbb{T}})=\{F\in \mathcal{L}^{\alpha}({\mathbb{T}%
}):\hat{F}(n)=\int_{\mathbb{T}}F(z)z^{-n}dm(z)=0,\text{\ for \ all }n<0\}.
\]
\bigskip

If $\alpha$ is a continuous rotationally symmetric norm on $L^{\infty
}(\mathbb{T})$, then it follows from part (10) in Proposition \ref{prop1}
that
\[
H^{\infty}({\mathbb{T}})\subset H^{\alpha}({\mathbb{T}})\subset H^{1}%
({\mathbb{T}}).
\]
We can view $H^{1}({\mathbb{T}})=H^{1}({\mathbb{D}}),$ a space of analytic
functions on the open unit disk $\mathbb{D}.$ Since $H^{\alpha}({\mathbb{T}%
})\subset H^{1}({\mathbb{T}}),$ we can view $H^{\alpha}({\mathbb{T}})\subset
H^{1}({\mathbb{D}})$ using the Poisson kernel. By Corollary
\ref{cor after mainthm}, it is easy to see that the following result holds.

\begin{theorem}
\label{thm 6.4} Suppose $\alpha$ is a continuous rotationally symmetric norm,
$F\in L^{\alpha}({\mathbb{T}}),$ and let $f_{r}(e^{it})=f(re^{it})=(F\ast
P_{r})(e^{it}).$ Then

\begin{enumerate}
\item $f_{r}\in L^{\alpha}({\mathbb{T}});$

\item $\alpha(f_{r})$ is increasing in $r;$

\item $\alpha(F-f_{r})\rightarrow0$ as $r\rightarrow1^{-};$

\item $\lim_{r\rightarrow1^{-}}\alpha(f_{r})=\alpha(F).$
\end{enumerate}
\end{theorem}

We define%

\[
H^{\alpha}({\mathbb{D}})=\{f\in H^{1}({\mathbb{D}}):F\in H^{\alpha}%
(\mathbb{T})\}.
\]
\newline Then we can view $H^{\alpha}({\mathbb{D}})=H^{\alpha}(\mathbb{T}).$
Recall that for all $1\leq p<\infty,$
\[
H^{p}({\mathbb{D}})=\{f\in H({\mathbb{D}}):\sup_{0<r<1}\Vert f_{r}\Vert
_{p}<\infty \},
\]
but we can not define $H^{\alpha}(\mathbb{D})$ in this way. Actually, if
$g:{\mathbb{D}}\rightarrow{\mathbb{C}}$ is analytic and
\[
\sup_{0<r<1}\alpha(g_{r})<\infty,
\]
then, since $\Vert \Vert_{1}\leq \alpha,$ the radial limit function $G$ is in
$H^{1}({\mathbb{T}})$ and when we apply the Poisson kernel to $G$ we get $g.$
However, we do not know if $G\in H^{\alpha}({\mathbb{T}}),$ because maybe
$G\in{\mathcal{L}}^{\alpha}({\mathbb{T}})$ and not in $L^{\alpha}%
(\mathbb{T}).$ However, when $\alpha$ is strongly continuous, there is no problem.

\begin{proposition}
Suppose $\alpha$ is a strongly continuous rotationally symmetric norm on
$L^{\infty}\left(  \mathbb{T}\right)  $ and $f:\mathbb{D}\rightarrow
\mathbb{C}$. The following are equivalent:

\begin{enumerate}
\item $f\in H^{\alpha}\left(  \mathbb{D}\right)  $;

\item $f\in H\left(  \mathbb{D}\right)  $ and $\sup_{0<r<1}\alpha \left(
f_{r}\right)  <\infty$.
\end{enumerate}
\end{proposition}

\begin{proof}
(1)$\Rightarrow$(2) is clear.

(2)$\Rightarrow$(1) Since $f\in H\left(  \mathbb{D}\right)  $, each
$f_{r}\left(  e^{it}\right)  =f\left(  re^{it}\right)  $ is continuous and if
$0<r<s$, we have%
\[
f_{r}=f_{s}\ast P_{r/s},
\]
which implies $\alpha \left(  f_{r}\right)  $ is monotone in $r.$ Since
$\left \Vert \cdot \right \Vert _{1}\leq \alpha$, the supremum condition implies
that $f\in H^{1}\left(  \mathbb{D}\right)  ,$ which implies%
\[
F\left(  e^{it}\right)  =\lim_{r\rightarrow1^{-}}f_{r}\left(  e^{it}\right)
\]
exists a.e. $\left(  m\right)  $, and $F\in H^{1}\left(  \mathbb{T}\right)  $,
and $f$ is the Poisson integral of $F.$ Suppose $\left \{  r_{n}\right \}  $ is
a sequence in $\left(  0,1\right)  $ with $r_{n}\rightarrow1^{-}$. Then
\[
F\left(  e^{it}\right)  =\liminf_{n\rightarrow \infty}\left \vert f_{r_{n}%
}\left(  e^{it}\right)  \right \vert =\lim_{n\rightarrow \infty}\inf_{k\geq
n}\left \vert f_{r_{k}}\left(  e^{it}\right)  \right \vert .
\]
Since the sequence $\left \{  \inf_{k\geq n}\left \vert f_{r_{k}}\left(
e^{it}\right)  \right \vert \right \}  $ is increasing, it follows from part
(\ref{part1}) of Proposition \ref{prop1} that%
\[
\alpha \left(  F\right)  =\lim_{n\rightarrow \infty}\alpha \left(  \inf_{k\geq
n}\left \vert f_{r_{k}}\right \vert \right)  \leq \lim_{n\rightarrow \infty}%
\alpha \left(  f_{r_{n}}\right)  \leq \sup_{0<r<1}\alpha \left(  f_{r}\right)
<\infty.
\]
Since $\alpha$ is strongly continuous we conclude $F\in L^{\alpha}\left(
\mathbb{T}\right)  \cap H^{1}\left(  \mathbb{T}\right)  =H^{\alpha}\left(
\mathbb{T}\right)  $.
\end{proof}

When $\alpha$ is not strongly continuous but continuous, we need a stricter
assumption to get $f\in H^{\alpha}\left(  \mathbb{D}\right)  $.

\begin{theorem}
Suppose $\alpha$ is a continuous rotationally symmetric norm on $L^{\alpha
}(\mathbb{T}).$ If $f\in H^{1}({\mathbb{D}})$ and $F(e^{it})=\lim
_{r\rightarrow1^{-}}f_{r}(e^{it}),$ then the following are equivalent:

\begin{enumerate}
\item $F\in L^{\alpha}(\mathbb{T});$

\item $f\in H^{\alpha}(\mathbb{D});$

\item There is a sequence $r_{n}\rightarrow1^{-}$ such that
\[
\lim_{m,n\rightarrow \infty}\alpha(f_{r_{n}}-f_{r_{m}})=0.
\]

\end{enumerate}
\end{theorem}

\begin{proof}
$(1)\Rightarrow(2)$ Suppose $f\in H^{1}(\mathbb{D}).$ If $F\in L^{\alpha
}(\mathbb{T}),$ then it follows from Theorem \ref{thm 6.4} that $f_{r}=F\ast
P_{r}\in L^{\alpha}(\mathbb{T}),$ which implies $f_{r}\in L^{\alpha
}(\mathbb{T})\cap H^{1}(\mathbb{T})=H^{\alpha}(\mathbb{T}),$ and hence $f\in
H^{\alpha}(\mathbb{D}).$

$(2)\Rightarrow(3)$ Suppose $f\in H^{\alpha}(\mathbb{D}).$ The definition of
$H^{\alpha}(\mathbb{D})$ implies that $f$ is the Poisson integral of $F\in
H^{\alpha}(\mathbb{T}).$ By part (3) of Theorem \ref{thm 6.4}, $\alpha
(f_{r}-F)\rightarrow0,$ which means $\{f_{r}\}$ is $\alpha$-Cauchy with
respect to $r,$ i.e., if $\{ r_{n}\}$ is a sequence in $(0,1)$ with
$r_{n}\rightarrow1^{-},$ then $\lim_{m,n\rightarrow \infty}\alpha(f_{r_{n}%
}-f_{r_{m}})=0.$

$(3)\Rightarrow(1)$ Suppose the hypothesis of (3) holds. Then there is a
$F_{0}\in L^{\alpha}(\mathbb{T})$ such that $\alpha(f_{r}-F_{0})\rightarrow0.$
Since $\Vert \cdot \Vert_{1}\leq \alpha,$ we conclude $\Vert f_{r}-F_{0}\Vert
_{1}\rightarrow0,$ and therefore there is a subsequence $\{f_{r_{k}}\}$ such
that
\[
F_{0}(e^{it})=\lim_{k\rightarrow \infty}f_{r_{k}}(e^{it})=\lim_{{r_{k}%
}\rightarrow1^{-}}f_{r_{k}}(e^{it}).
\]
By the uniqueness of the limit, it follows that $F=F_{0}\in L^{\alpha
}(\mathbb{T}).$
\end{proof}

If $\alpha$ is a continuous rotationally symmetric norm and $X$ is a separable
Banach space, we define%
\[
H^{\alpha}({\mathbb{T}},X)=\{F\in L^{\alpha}({\mathbb{T}},X):\hat{F}%
(n)=\int_{\mathbb{T}}F(z)z^{-n}dm(z)=0,\mbox{\ for\  all}\ n<0\}.
\]

We identify $F\in L^{1}\left(  \mathbb{T},X\right)  $ with a formal Laurent
series
\[
F\sim \sum_{n=-\infty}^{\infty}\hat{F}(n)z^{n},
\]
keeping in mind that each $\hat{F}(n)$ is in $X$. We let $S_{n}\left(
F\right)  \left(  z\right)  =\sum_{k=-n}^{n}\hat{F}(k)z^{k}$ for $n\geq0$ and,
for each $n\geq1$ we define%
\[
\sigma_{n}\left(  F\right)  \left(  z\right)  =\frac{1}{n+1}\sum_{k=0}%
^{n}S_{k}\left(  F\right)  \left(  z\right)  =\left(  F\ast K_{n}\right)
\left(  z\right)  ,
\]
where $K_{n}$ is the Fejer kernel from Example \ref{example}. The following
follows from Theorem \ref{appx id thm}.

\begin{proposition}
\label{dense function}If $\alpha$ is a continuous rotationally symmetric norm
and $X$ is a separable Banach space, then

\begin{enumerate}
\item If $F\in L^{\alpha}({\mathbb{T}},X),$ then $\alpha \left(  F-\sigma
_{n}\left(  F\right)  \right)  \rightarrow0,$ so $L^{\alpha}\left(
\mathbb{T},X\right)  $ is the $\alpha$-closed linear span of $\left \{  x\cdot
z^{n}:x\in X,n\in \mathbb{Z}\right \}  $;

\item If $F\in H^{\alpha}\left(  \mathbb{T},X\right)  $, then $S_{n}\left(
F\right)  =\sum_{k=0}^{n}\hat{F}(k)z^{k},$ so $H^{\alpha}\left(
\mathbb{T},X\right)  $ is the $\alpha$-closed linear span of $\left \{  x\cdot
z^{n}:x\in X,n\geq0\right \}  .$
\end{enumerate}
\end{proposition}

\section{Beurling's Invariant subspace theorem}

It is well-known that all of the invariant subspaces for the unilateral shift
operator, i.e., $M_{z}$ on $H^{2}({\mathbb{T}}),$ have the form $\phi
H^{2}({\mathbb{T}})$ for some inner function $\phi,$ where an
$inner\ function$ is defined to be a member of $H^{\infty}\left(
\mathbb{T}\right)  $ that is unimodular on $\mathbb{T}.$ The original
statement concerning the space $H^{2}\left(  \mathbb{T}\right)  $ of functions
on the unit disk $\mathbb{D}$ was proved in 1949 by A. Beurling \cite{B},
\cite{Helson}, and was later extended to $H^{p}\left(  \mathbb{T}\right)  $
classes by T. P. Srinivasan \cite{Sr}. The standard proof for $H^{p}%
({\mathbb{T}})$ uses the $H^{2}$-result and considers the only two possible
cases $H^{p}({\mathbb{T}})\subset H^{2}({\mathbb{T}})$ or $H^{2}({\mathbb{T}%
})\subset H^{p}({\mathbb{T}}).$ Neither of these relations hold when
$H^{p}({\mathbb{T}})$ is replaced with $H^{\alpha}({\mathbb{T}}),$ when
$\alpha$ is a rotationally symmetric norm on $L^{\alpha}(\mathbb{T}).$ In this
section, we aim at a similar result for $H^{\alpha}(\mathbb{T})$ case by using
different techniques.

First, let us review two important topologies. Suppose $\mathcal{H}$ is a
Hilbert space.

The $weak\ operator\ topology$ (WOT) on $\mathcal{B}(\mathcal{H})$ is defined
as the weakest topology such that the sets
\[
{\mathcal{W}}(T,x,y):=\{A\in{\mathcal{B}(\mathcal{H}):|((T-A)x,y)|<1}\}
\]
are open. The sets
\[
{\mathcal{W}}(T_{i},x_{i},y_{i};1\leq i\leq n):=\bigcap_{i=1}^{n}{\mathcal{W}%
}(T_{i},x_{i},y_{i})
\]
form a base for the WOT topology. A net $\{T_{\lambda}\}$ converges WOT to an
operator $T$ if and only if
\[
\lim_{\lambda}(T_{\lambda}x,y)=(Tx,y)\  \  \  \mbox{for\ all\ }x,y\in
\mathcal{H}.
\]

Analogously, the $strong\ operator\ topology$ (SOT) is defined by the open
sets
\[
\mathcal{S}(T,x):=\{A\in{\mathcal{B}(\mathcal{H})}:\Vert(T-A)x\Vert<1\}.
\]
A net $\{T_{\lambda}\}$ converges SOT to $T$ if and only if
\[
\lim_{\lambda}T_{\lambda}x=Tx\  \  \  \mbox{for \ all
}x\in \mathcal{H}.
\]

\begin{lemma}
\label{weakly closed} Suppose $X$ is a Banach space and $M$ is a closed linear
subspace of $X.$ Then $M$ is weakly closed.
\end{lemma}

\begin{proof}
It is clear that $M\subset \overline{M}^{w}.$ Suppose there is $x_{0}%
\in \overline{M}^{w}$ with $x_{0}\notin M.$ Then by the Hahn-Banach theorem
there is a linear functional $\phi \in X^{\sharp}$ such that $\phi|_{M}=0$ and
$\phi(x_{0})\neq0.$ Since $x_{0}\in \overline{M}^{w},$ there is a net
$\{x_{\lambda}\}$ in $M$ such that $x_{\lambda}\rightarrow x_{0}$ weakly,
which implies that $\phi(x_{\lambda})\rightarrow \phi(x_{0})\neq0.$ But
$\phi(x_{\lambda})=0$ for all $\lambda,$ which is a contradiction, and
therefore $M=\overline{M}^{w}.$ This completes the proof.
\end{proof}

\begin{lemma}
Suppose $X$ is a Banach space and $M$ is a closed linear subspace of $X.$ Let
${\mathcal{A}}=\{T\in B(X):TM\subset M\}.$ Then $\mathcal{A}$ is closed in the
weak operator topology.
\end{lemma}

\begin{proof}
Suppose $T\in B(X)$ and $\{T_{\lambda}\}$ is a net in $\mathcal{A}$ with
$T_{\lambda}\rightarrow T$ in the weak operator topology. If $x\in M,$ then
$T_{\lambda}x\rightarrow Tx$ weakly and $T_{\lambda}x\in M$ for all $\lambda.$
It follows from Lemma \ref{weakly closed} that $M$ is a weakly closed subspace
of $X,$ which implies $Tx\in M$ for all $x\in M,$ and hence $T\in \mathcal{A},$
which implies that $\mathcal{A}$ is closed in the weak operator topology.
\end{proof}

The following lemma is well-known \cite{Hof}.

\begin{lemma}
\label{fejer lemma}Suppose $K_{n}$ is the Fejer's kernel in Example
\ref{example}. If $f\in L^{p}(\mathbb{T})$ for $1\leq p\leq \infty,$ then
$f\ast K_{n} $ converges in norm to $f$ as $r\rightarrow1^{-}$ when $p$ is
finite and in the weak* topology when $p=\infty.$
\end{lemma}

\begin{lemma}
\label{lemma6.2} Suppose $\alpha$ is a continuous rotationally symmetric norm
on $L^{\alpha}(\mathbb{T}).$ If $M$ is a closed subspace of $H^{\alpha}\left(
\mathbb{T}\right)  $ invariant under $M_{z},$ which means $zM\subset M,$ then
$H^{\infty}({\mathbb{T}})\cdot M\subset M.$
\end{lemma}

\begin{proof}
It follows from $zM\subset M$ that for any polynomial $P\in{\mathcal{P}}_{+},$
$P(z)M\subset M.$ Suppose $h\in M$ and $\phi \in(H^{\alpha}({\mathbb{T}%
}))^{\sharp}.$ Then by the Hahn-Banach theorem, there is a linear functional
$\psi \in(L^{\alpha}({\mathbb{T}}))^{\sharp}$ such that $\psi|_{H^{\alpha
}({\mathbb{T}})}=\phi$ and $\Vert \psi \Vert=\Vert \phi \Vert.$ Since $\alpha$ is
continuous, by Proposition \ref{dualspace} $(L^{\alpha}({\mathbb{T}}%
))^{\sharp}={\mathcal{L}}^{\alpha^{\prime}}({\mathbb{T}}),$ and therefore
there is an $u\in{\mathcal{L}}^{\alpha^{\prime}}({\mathbb{T}})$ such that
$\psi(h)=\int_{\mathbb{T}}hudm,$ and thus $\Vert hu\Vert_{1}\leq
\alpha(h)\alpha^{\prime}(u)<\infty,$ which implies $hu\in L^{1}({\mathbb{T}%
}).$

Next suppose $f\in H^{\infty}({\mathbb{T}}).$ It is clear that the Cesaro
means
\[
\sigma_{n}(f)=\frac{S_{0}(f)+S_{1}(f)+\ldots S_{n}(f)}{n+1}\in{\mathcal{P}%
_{+}}.
\]
Therefore by Lemma \ref{fejer lemma}, $\sigma_{n}(f)\rightarrow f$ in the
weak* topology. Since $hu\in L^{1}({\mathbb{T}}),$ it follows that
\[
\int_{\mathbb{T}}\sigma_{n}(f)hudm\rightarrow \int_{\mathbb{T}}fhudm.
\]
We also note that $\sigma_{n}(f)h\in{\mathcal{P}_{+}}M\subset M\subset
L^{\alpha}({\mathbb{T}})$ and $u\in{\mathcal{L}}^{\alpha^{\prime}}%
(\mathbb{T}),$ which implies $\sigma_{n}(f)h\rightarrow fh$ weakly, and since
$M$ is weakly closed, we see $fh\in M$ for all $f\in H^{\infty}({\mathbb{T}%
}),$ and thus $H^{\infty}({\mathbb{T}})h\subset M$ for all $h\in M.$ This
completes the proof.
\end{proof}

\begin{proposition}
\label{frac prop} Suppose $f\in H^{\alpha}({\mathbb{T}}).$ Then there are two
measurable functions $u,v\in H^{\infty}(\mathbb{T})$ such that $f=\frac{u}%
{v}.$
\end{proposition}

\begin{proof}
It is a familiar fact that every function $f$ in $H^{1}(\mathbb{T})$ can be
written as $f=\frac{u}{v},$ where $u,v\in H^{\infty}(\mathbb{T}).$ Since
$H^{\alpha}({\mathbb{T}})\subset H^{1}({\mathbb{T}}),$ the result follows.
\end{proof}

\begin{proposition}
\label{prop on ball} Suppose $\alpha$ is a continuous rotationally symmetric
norm. Then on the unit ball $\mathbb{B}(L^{\infty}({\mathbb{T}}))=\{f\in
L^{\infty}(\mathbb{T}):\Vert f\Vert_{\infty}\leq1\},$ the following statements hold:

\begin{enumerate}
\item The $\alpha$-topology coincides with the topology of convergence in measure;

\item $\mathbb{B}(L^{\infty}({\mathbb{T}}))=\{f\in L^{\infty}(\mathbb{T}%
):\Vert f\Vert_{\infty}\leq1\}$ is $\alpha$-closed;

\item $\mathbb{B}(H^{\infty}({\mathbb{T}}))=\{f\in H^{\infty}(\mathbb{T}%
):\Vert f\Vert_{\infty}\leq1\}$ is $\alpha$-closed.
\end{enumerate}
\end{proposition}

\begin{proof}
$(1)$ It was shown in \cite{HN2}.

$(2)$ Suppose $\{g_{n}\}$ is a sequence in $\mathbb{B}(L^{\infty}({\mathbb{T}%
}))$ with $\alpha(g_{n}-g)\rightarrow0.$ Then $g\in L^{\alpha}(\mathbb{T})$
and $\Vert g_{n}-g\Vert_{1}\leq \alpha(g_{n}-g)\rightarrow0,$ and thus there is
a subsequence $\{g_{n_{k}}\}$ with $g_{n_{k}}\rightarrow g$ a.e. $(m).$ It
follows from $|g_{n_{k}}|\leq1$ that $|g|\leq1,$ and hence $g\in
\mathbb{B}(L^{\infty}({\mathbb{T}})).$ This completes the proof.

$(3)$ Suppose $\{g_{n}\}$ is a sequence in $\mathbb{B}(H^{\infty}({\mathbb{T}%
}))$ with $\alpha(g_{n}-g)\rightarrow0.$ It is clear that $H^{\infty
}(\mathbb{T})\subset L^{\infty}(\mathbb{T}).$ It follows from part (2) above
that $g\in \mathbb{B}(L^{\infty}({\mathbb{T}})).$ Since $g_{n}\in
\mathbb{B}(H^{\infty}({\mathbb{T}}))\subset H^{\infty}(\mathbb{T})\subset
H^{\alpha}(\mathbb{T})$ and $\alpha(g_{n}-g)\rightarrow0,$ we conclude $g\in
H^{\alpha}(\mathbb{T}),$ and thus $g\in \mathbb{B}(H^{\infty}({\mathbb{T}})).$
\end{proof}

The following Lemma is the Krein-Smulian theorem.

\begin{lemma}
Let $X$ be a Banach space. A convex set in $X^{\sharp}$ is weak* closed if and
only if its intersection with $\{ \phi \in X^{\sharp}: \| \phi \| \leq1\}$ is
weak* closed.
\end{lemma}

The following theorem is the generalized version of the very important 1949
theorem of Beurling.

\begin{theorem}
\label{invariantthm} Suppose $\alpha$ is a continuous rotationally symmetric
norm and $M$ is a closed subspace of $H^{\alpha}\left(  \mathbb{T}\right)  .$
Then $zM\subseteq M$ if and only if $M=\varphi H^{\alpha}\left(
\mathbb{T}\right)  $ for some inner function $\varphi.$
\end{theorem}

\begin{proof}
The only if part is obvious. Suppose $f\in M\subset H^{\alpha}({\mathbb{T}})$
with $f\neq0.$ It follows from Proposition \ref{frac prop} that there are two
measurable functions $u,v\in H^{\infty}(\mathbb{T})$ such that $f=\frac{u}%
{v},$ where $u\neq0.$ Then by Lemma \ref{lemma6.2}, $u=f\cdot v\in M\cdot
H^{\infty}\subset M,$ which means $0\neq u\in M\cap H^{\infty}(\mathbb{T}).$
Let
\[
{\mathcal{A}}=\{u\in H^{\infty}({\mathbb{T}}):\exists \ v\in H^{\infty
}({\mathbb{T}}),\  \frac{u}{v}\in M\}.
\]
If $u\in \mathcal{A},$ then there is a $v\in H^{\infty}$ such that $u=\frac
{u}{v}\cdot v\in M\cdot H^{\infty}({\mathbb{T}})\subset M,$ and thus
$\mathcal{A}$$\subset H^{\infty}({\mathbb{T}})\cap M.$ On the other hand,
suppose $u\in H^{\infty}({\mathbb{T}})\cap M.$ Since $1\in H^{\infty
}({\mathbb{T}})$ and $u\in M,$ it follows that $u=\frac{u}{1}\in M,$ which
implies $u\in \mathcal{A}.$ Hence $\mathcal{A}=H^{\infty}({\mathbb{T}})\cap M.$

\textbf{Claim:} $\mathcal{A}$$=H^{\infty}({\mathbb{T}})\cap M$ is weak* closed.

In fact, it is clear that ${\mathcal{A}}\cap \mathbb{B}(L^{\infty}({\mathbb{T}%
}))=M\cap \mathbb{B}(L^{\infty}({\mathbb{T}})),$ by Corollary
\ref{prop on ball} we see ${\mathcal{A}}\cap \mathbb{B}(L^{\infty}({\mathbb{T}%
}))$ is $\alpha$ closed. Since $\alpha$ is continuous, it follows from
\cite[Theorem 4.1]{HN2} that ${\mathcal{A}}\cap \mathbb{B}(L^{\infty
}({\mathbb{T}}))$ is SOT closed. The fact that ${\mathcal{A}}\cap
\mathbb{B}(L^{\infty}({\mathbb{T}}))$ is convex implies that ${\mathcal{A}%
}\cap \mathbb{B}(L^{\infty}({\mathbb{T}}))$ is WOT closed, hence it is weak*
closed, since the WOT and the weak*-topology coincide on bounded sets.
Therefore it follows from the Krein-Smulian theorem that ${\mathcal{A}}$ is
weak* closed in $H^{\infty}({\mathbb{T}}).$

Furthermore, since $z\cdot M\subset M$ and $z\cdot H^{\infty}\subset
H^{\infty},$ we conclude $\mathcal{A}$ is an invariant subspace in $H^{\infty
}({\mathbb{T}})$ under the unilateral shift operator $M_{z}.$ Then, by
Srinivasan's theorem in \cite{Sr}, $M\cap H^{\infty}({\mathbb{T}}%
)=\mathcal{A}=\phi H^{\infty}({\mathbb{T}}),$ where $\phi$ is inner, so $\phi
H^{\infty}({\mathbb{T}})\subset M,$ which implies that
\[
\overline{\phi H^{\infty}({\mathbb{T}})}^{\alpha}=\phi \overline{H^{\infty
}({\mathbb{T}})}^{\alpha}=\phi H^{\alpha}({\mathbb{T}})\subset \overline
{M}^{\alpha}=M.
\]
Conversely, suppose $0\neq f\in M\subset H^{\alpha}(\mathbb{T}).$ Then there
are two measurable functions $u,v\in H^{\infty}({\mathbb{T}})$ such that
$f=\frac{u}{v},$ where $v$ is outer. The definition of $\mathcal{A}$ yields
$0\neq u\in{\mathcal{A}}=\phi H^{\infty}({\mathbb{T}}),$ which implies there
is a function $u_{1}$ with $0\neq u_{1}\in H^{\infty}({\mathbb{T}})$ such that
$u=\phi \cdot u_{1},$ thus $f=\phi \cdot \frac{u_{1}}{v},$ where $\phi$ is inner,
and so $\frac{u_{1}}{v}\in L^{\alpha}(\mathbb{T})\subset L^{1}(\mathbb{T}).$
Since $v$ is outer, it follows that $\frac{u_{1}}{v}\in H^{1}({\mathbb{T}}).$
Then by Corollary \ref{intersection cor}, $\frac{u_{1}}{v}\in H^{1}%
(\mathbb{T})\cap L^{\alpha}(\mathbb{T})=H^{\alpha}({\mathbb{T}}),$ and hence
$f\in \phi H^{\alpha}({\mathbb{T}}).$ This implies $M\subset \phi H^{\alpha
}({\mathbb{T}}).$
\end{proof}

\section{Outer functions in $H^{\alpha}(\mathbb{T})$}

Suppose $\alpha$ is a continuous rotationally symmetric norm and $f\in
H^{\alpha}\left(  \mathbb{T}\right)  $. We say that $f$ is $\alpha
$\emph{-outer} if $f$ is a cyclic vector for $H^{\infty}\left(  \mathbb{T}%
\right)  $ acting on $H^{\alpha}\left(  \mathbb{T}\right)  $, i.e., $\left(
H^{\infty}\left(  \mathbb{T}\right)  f\right)  ^{-\alpha}=H^{\alpha}\left(
\mathbb{T}\right)  .$ Originally the terms \emph{inner} and \emph{outer} were
defined for functions in $H^{p}\left(  \mathbb{T}\right)  $ by Beurling (see
\cite{B}) for $0<p<\infty$. It was shown that a function in $H^{p}\left(
\mathbb{T}\right)  $ is outer if and only if it is outer in $H^{1}\left(
\mathbb{T}\right)  $. Since $H^{p}\left(  \mathbb{T}\right)  \subset
H^{1}\left(  \mathbb{T}\right)  $, the term \emph{outer} was used without
reference to $p.$ We prove the same result for $H^{\alpha}\left(
\mathbb{T}\right)  $ when $\alpha$ is a continuous rotationally symmetric norm.

\begin{theorem}
\label{thm7.1} Suppose $\alpha$ is a continuous rotationally symmetric norm
and $f\in H^{\alpha}(\mathbb{T}).$ Then
\[
\left(  H^{\infty}({\mathbb{T}})\cdot f\right)  ^{-\  \alpha}=H^{\alpha
}({\mathbb{T}})\Longleftrightarrow f\text{ is outer in }H^{1}\left(
\mathbb{T}\right)  .
\]

\end{theorem}

\begin{proof}
Let $M=\left(  H^{\infty}({\mathbb{T}})\cdot f\right)  ^{-\  \alpha}$. It is
clear that $M_{z}\cdot H^{\infty}({\mathbb{T}})\cdot f\subset H^{\infty
}({\mathbb{T}})\cdot f,$ and so $M$ is a closed invariant subspace of
$H^{\alpha}({\mathbb{T}}).$ It follows from Theorem \ref{invariantthm} that
$M=\left(  H^{\infty}({\mathbb{T}})\cdot f\right)  ^{-\  \alpha}=\phi \cdot
H^{\alpha}({\mathbb{T}})$ for some inner function $\phi.$ Since $f=1\cdot
f\in \left(  H^{\infty}({\mathbb{T}})\cdot f\right)  ^{-\  \alpha}=M=\phi \cdot
H^{\alpha}({\mathbb{T}}),$ there is a $g\in H^{\alpha}({\mathbb{T}})$ such
that $f=\phi g.$ If $f$ is outer, then $\phi \equiv$ Constant, which implies
\[
M=\left(  H^{\infty}({\mathbb{T}})\cdot f\right)  ^{-\  \alpha}=\phi \cdot
H^{\alpha}({\mathbb{T}})=H^{\alpha}({\mathbb{T}}).
\]

Conversely, suppose $\left(  H^{\infty}({\mathbb{T}})\cdot f\right)
^{-\  \alpha}=H^{\alpha}({\mathbb{T}}).$ Then
\[
H^{\infty}({\mathbb{T}})\subset H^{\alpha}({\mathbb{T}})=\left(  H^{\infty
}({\mathbb{T}})\cdot f\right)  ^{-\  \alpha}\subset \left(  H^{\infty
}({\mathbb{T}})\cdot f\right)  ^{-\  \Vert \cdot \Vert_{1}},
\]
and thus
\[
H^{1}({\mathbb{T}})=H^{\infty}\left(  \mathbb{T}\right)  ^{-\left \Vert
\cdot \right \Vert _{1}}\subset \left(  H^{\infty}({\mathbb{T}})\cdot f\right)
^{-\  \Vert \cdot \Vert_{1}}\subset H^{1}({\mathbb{T}}),
\]
which means $H^{1}({\mathbb{T}})=\left(  H^{\infty}({\mathbb{T}})\cdot
f\right)  ^{-\  \Vert \cdot \Vert_{1}}.$ This implies $f$ is outer.
\end{proof}

The following corollary shows that, given $f\in H^{\alpha}(\mathbb{T})\subset
H^{1}\left(  \mathbb{T}\right)  ,$ the two factors in the inner-outer
factorization of $f$ are both in $H^{\alpha}(\mathbb{T}).$

\begin{corollary}
(Inner-outer factorization) Suppose $\alpha$ is a continuous rotationally
symmetric norm and $f\in H^{\alpha}(\mathbb{T})\subset H^{1}\left(
\mathbb{T}\right)  ,$ and suppose $\phi$ is an inner function and and $g\in
H^{1}\left(  \mathbb{T}\right)  $ is an outer function such that $f=\phi g$,
i.e., $f=\phi g$ is the Riesz-Smirnov inner-outer factorization of $f$ in
$H^{1}\left(  \mathbb{T}\right)  $. Then $g\in H^{\alpha}(\mathbb{T}).$
\end{corollary}

\begin{proof}
Since $H^{\alpha}(\mathbb{T})\subset H^{1}(\mathbb{T}),$ the inner-outer
factorization in $H^{1}(\mathbb{T})$ applies of course to functions in
$H^{\alpha}(\mathbb{T}),$ that is, there exists an inner function $\phi \in
H^{\infty}(\mathbb{T})$ and an outer function $g\in H^{1}(\mathbb{T})$
satisfying $g(\zeta)=|f(\zeta)|$ for almost every $\zeta \in{\mathbb{T}}$ such
that $f=\phi g.$ Then it follows from Theorem \ref{thm7.1} that $g$ is outer
in $H^{\alpha}(\mathbb{T}),$ and so we have the desired factorization.
\end{proof}

\begin{theorem}
\label{outer function thm} (Characterization of outer functions) Suppose
$\alpha$ is a continuous rotationally symmetric norm. A nonzero function $g\in
H^{\alpha}({\mathbb{T}})$ is outer if and only if it has the following property:

for every $f\in H^{\alpha}({\mathbb{T}}),$ if $\frac{f}{g}\in L^{\alpha
}(\mathbb{T}),$ then $\frac{f}{g}\in H^{\alpha}(\mathbb{T}%
).\  \  \  \  \  \  \  \  \  \  \  \  \  \  \ (\ast)$
\end{theorem}

\begin{proof}
Suppose $g$ is outer in $H^{\alpha}({\mathbb{T}}).$ Then it follows from
\ref{thm7.1} that $g$ is outer in $H^{1}(\mathbb{T}).$ Since $\frac{f}{g}\in
L^{\alpha}({\mathbb{T}})\subset L^{1}({\mathbb{T}})$ and $f\in H^{\alpha
}({\mathbb{T}})\subset H^{1}({\mathbb{T}}),$ we conclude $\frac{f}{g}\in
H^{1}({\mathbb{T}}),$ and hence $\frac{f}{g}\in L^{\alpha}({\mathbb{T}})\cap
H^{1}({\mathbb{T}})=H^{\alpha}({\mathbb{T}}).$

Conversely, suppose $g\in H^{\alpha}(\mathbb{T})$ satisfy the property
$(\ast).$ By the inner-outer factorization, we can write $g=\phi G,$ where
$\phi$ is inner and $G\in H^{\alpha}(\mathbb{T})$ is outer, then $\phi
=\frac{g}{G}$ and ${\bar{\phi}}=\frac{G}{g},$ being unimodular, is in
$L^{\alpha}(\mathbb{T}).$ It follows from the property $(\ast)$ that
${\bar{\phi}}=\frac{G}{g}\in H^{\alpha}(\mathbb{T}),$ thus $\phi,\ {\bar{\phi
}}\in H^{\alpha}(\mathbb{T})$ with $|\phi|=1,$ which implies $\phi \equiv$
Constant, and therefore $g=\phi G$ is outer.
\end{proof}

\section{Multipliers of $H^{\alpha}(\mathbb{T})$}

In \cite{HN1}, D. Hadwin and E. Nordgren proved that, if $\alpha$ is a
continuous symmetric gauge norm and if $Y$ is the set of all measurable
complex functions on $\mathbb{T}$, then $\left(  L^{\alpha}\left(
\mathbb{T}\right)  ,Y\right)  $ is a multiplier pair, and if $f\in L^{\alpha
}\left(  \mathbb{T}\right)  $, then
\[
f\cdot L^{\alpha}({\mathbb{T}})\subset L^{\alpha}({\mathbb{T}}%
)\Longleftrightarrow f\in L^{\infty}({\mathbb{T}}),
\]
i.e., the \emph{multipliers} of $L^{\alpha}({\mathbb{T}})$ are the functions
in $L^{\infty}(\mathbb{T}).$ In this section we will extend these results to
the case in which $\alpha$ is a continuous rotationally symmetric norm.
Another multiplier-type results are Theorem \ref{multalpha1} and Corollary
\ref{multalfa21}. We will also prove similar results for $H^{\alpha}\left(
\mathbb{T}\right)  $.

\begin{definition}
$\left(  W,\left \Vert \cdot \right \Vert \right)  $ is a functional Banach space
on a nonempty set $X$ if and only if

\begin{enumerate}
\item $W$ is a vector space of functions from $X$ to $\mathbb{C};$

\item For all $x\in W,$ there is a $f\in W$ such that $f(x)\neq0;$

\item there is a norm $\left \Vert \cdot \right \Vert $ such that $(W,\left \Vert
\cdot \right \Vert )$ is a Banach space;

\item For all $x\in X,$ there is a $r_{x}$ such that $\ |f(x)|\leq r_{x}\Vert
f\Vert$ for all $f\in W,$ i.e., for all $x\in X,$ the map $E_{x}%
:W\rightarrow \mathbb{C}$ defined by $E_{x}(f)=f(x)$ is a linear bounded functional.
\end{enumerate}
\end{definition}

\begin{proposition}
\label{corollary8.2} Suppose $\alpha$ is a continuous rotationally symmetric
norm. Then $H^{\alpha}({\mathbb{T}})=H^{\alpha}\left(  \mathbb{D}\right)  $ is
a functional Banach space.
\end{proposition}

\begin{proof}
It is clear that $1\in H^{\infty}({\mathbb{T}})\subset H^{\alpha}({\mathbb{T}%
})$ and $(H^{\alpha}({\mathbb{T}}),\alpha)$ is a Banach space.

Furthermore, suppose $F_{n},F\in H^{\alpha}({\mathbb{T}})$ with $\alpha
(F_{n}-F)\rightarrow0.$ Since $\alpha$ is continuous, we see $\Vert \cdot
\Vert_{1}\leq \alpha,$ and then $\Vert F_{n}-F\Vert_{1}\leq \alpha
(F_{n}-F)\rightarrow0.$ The fact that $(H^{1}({\mathbb{T}}),\Vert \cdot
\Vert_{1})$ is a functional Banach space implies $F_{n}(z)\rightarrow F(z)$
for all $z\in \mathbb{T},$ and hence $(H^{\alpha}({\mathbb{T}}),\alpha)$ is a
functional Banach space.
\end{proof}

\begin{lemma}
\label{functional}Suppose $W$ is a functional Banach space on $X,$ and
$\phi:X\rightarrow \mathbb{C}$ and $\phi W\subset W.$ Define $A:W\rightarrow W$
by $Af=\phi f.$ Then $A$ is bounded.
\end{lemma}

\begin{proof}
Suppose $f_{n}\rightarrow f$ and $Af_{n}=\phi f_{n}\rightarrow g.$ Since $W$
is a functional Banach space, it follows that $f_{n}(x)\rightarrow f(x)$ for
all $x\in X.$ Therefore $(Af_{n})(x)=\phi(x)f_{n}(x)\rightarrow g(x).$ Because
$\phi(x)\in \mathbb{C}$, $\phi(x)f_{n}(x)\rightarrow \phi(x)f(x),$ and thus
$g(x)=\phi(x)f(x)$ for all $x\in X.$ Therefore $g=\phi f,$ and the closed
graph theorem implies that $A$ is bounded.
\end{proof}

If $W$ is a space of (equivalence classes) of functions and $\psi$ is a
function such that $\psi W\subset W$, we say that $\psi$ is a \emph{multiplier
of }$W$ and we define the \emph{multiplication operator }$M_{\psi
}:W\rightarrow W$ by%
\[
M_{\psi}f=\psi f.
\]
We now compute the multipliers of $L^{\alpha}\left(  \mathbb{T}\right)  $ and
$H^{\alpha}\left(  \mathbb{T}\right)  =H^{\alpha}\left(  \mathbb{D}\right)  $.

\begin{theorem}
(Multipliers on $L^{\alpha}\left(  \mathbb{T}\right)  $ and $H^{\alpha
}({\mathbb{T}})$) Suppose $\alpha$ is a continuous rotationally symmetric
norm, $\phi:\mathbb{D}\rightarrow \mathbb{C}$ is analytic and $\psi
:\mathbb{T}\rightarrow \mathbb{C}$ is measurable. Then

\begin{enumerate}
\item $\psi L^{\alpha}\left(  \mathbb{T}\right)  \subset L^{\alpha}\left(
\mathbb{T}\right)  $ if and only if $\psi \in L^{\infty}\left(  \mathbb{T}%
\right)  $. Moreover, $\left \Vert \psi \right \Vert _{\infty}=\left \Vert
M_{\psi}\right \Vert ;$

\item $\phi H^{\alpha}\left(  \mathbb{D}\right)  \subset H^{\alpha}\left(
\mathbb{D}\right)  $ if and only if $\phi \in H^{\infty}\left(  \mathbb{D}%
\right)  $. Moreover, $\left \Vert \phi \right \Vert _{\infty}=\left \Vert
M_{\phi}\right \Vert $.
\end{enumerate}
\end{theorem}

\begin{proof}
(1) Suppose $\alpha \left(  f_{n}-f\right)  \rightarrow0$ and $\alpha \left(
M_{\psi}f_{n}-g\right)  \rightarrow0$. Then $f_{n}\rightarrow f$ in measure
and $\psi f_{n}\rightarrow g$ in measure, so we see that $g=\psi f.$ Hence, by
the closed graph theorem, $M_{\psi}$ is bounded. Suppose $\varepsilon>0$ and
let $E=\left \{  z\in \mathbb{T}:\left \vert \psi \left(  z\right)  \right \vert
\geq \left \Vert M_{\psi}\right \Vert +\varepsilon \right \}  $. Then $\left \vert
\psi \chi_{E}\right \vert \geq \left(  \left \Vert M_{\psi}\right \Vert
+\varepsilon \right)  ,$ so
\[
\left \Vert M_{\psi}\right \Vert \alpha \left(  \chi_{E}\right)  \geq
\alpha \left(  M_{\psi}\chi_{E}\right)  =\alpha \left(  \psi \chi_{E}\right)
\geq \left(  \left \Vert M_{\psi}\right \Vert +\varepsilon \right)  \alpha \left(
\chi_{E}\right)  ,
\]
which implies $\chi_{E}=0,$ or $m\left(  E\right)  =0.$ Since $\varepsilon>0$
was arbitrary, we see that $\left \vert \psi \left(  z\right)  \right \vert
\leq \left \Vert M_{\psi}\right \Vert $ a.e. $\left(  m\right)  $, so $\psi \in
L^{\infty}\left(  \mathbb{T}\right)  $ and $\left \Vert \psi \right \Vert
_{\infty}\leq \left \Vert M_{\psi}\right \Vert $. On the other hand,
$\alpha \left(  M_{\psi}f\right)  =\alpha \left(  \psi f\right)  \leq \left \Vert
\psi \right \Vert _{\infty}\alpha \left(  f\right)  $ implies $\left \Vert
M_{\psi}\right \Vert \leq \left \Vert \psi \right \Vert _{\infty}$.

(2) It follows from Lemma \ref{functional} that $M_{\phi}$ is bounded. Suppose
$f\in H^{\alpha}({\mathbb{T}}).$ Then $\alpha(\phi^{n}f)=\alpha(\left(
M_{\phi}\right)  ^{n}f)\leq \Vert M_{\phi}\Vert^{n}\alpha(f).$ Therefore, if
$M_{\phi}=0,$ then $\phi=0\in H^{\infty}({\mathbb{T}}).$ Otherwise, if
$M_{\phi}\neq0,$ let $\psi=\frac{\phi}{\Vert M_{\phi}\Vert}.$ Then%
\[
\alpha(\psi^{n}f)=\alpha(\frac{\phi^{n}}{\Vert M_{\phi}\Vert^{n}}f)\leq
\frac{\Vert M_{\phi}\Vert^{n}}{\Vert M_{\phi}\Vert^{n}}\alpha(f)=\alpha(f).
\]
It follows from $\phi \in H^{\alpha}({\mathbb{T}})$ that $\phi^{n}f\in
H^{\alpha}({\mathbb{T}}),$ thus $\psi^{n}f\in H^{\alpha}({\mathbb{T}}).$ By
Corollary \ref{corollary8.2}, $H^{\alpha}(\mathbb{T})$ is a functional Banach
space. Therefore for all $x\in \mathbb{T},$ there is an $f\in H^{\alpha
}({\mathbb{T}})$ such that $f(x)\neq0,$ and there is a $r_{x}>0$ such that
\[
|\psi^{n}(x)f(x)|\leq r_{x}\alpha(\psi^{n}f)\leq r_{x}\alpha(f)<\infty,\text{
for all }n\geq1.
\]
Hence $|\psi(x)|\leq1,$ which means for all $x\in \mathbb{T},$ $|\phi
(x)|\leq \Vert M_{\phi}\Vert<\infty,$ and therefore $\phi \in H^{\infty
}({\mathbb{T}})$ with $\Vert \phi \Vert_{\infty}\leq \Vert M_{\phi}\Vert.$
Furthermore, since $\alpha(M_{\phi}f)=\alpha(\phi f)\leq \Vert \phi \Vert
_{\infty}\alpha(f),$ it follows that $\Vert M_{\phi}\Vert \leq \Vert \phi
\Vert_{\infty}.$ This implies $\Vert M_{\phi}\Vert=\Vert \phi \Vert_{\infty}.$
\end{proof}

Multiplier pairs were created and studied in \cite{HN1}, \cite{HN2} and
\cite{HNL} . The following result is an easy consequence of our results and
results in \cite{HN1}.

\begin{corollary}
If $\alpha \in \mathcal{R}_{c}$, $Y_{1}$ is the set of all measurable functions
topologized by convergence in measure, and $Y_{2}$ is the set of all analytic
functions on $\mathbb{D}$ topologized by uniform convergence on compact
subsets, then

\begin{enumerate}
\item $\left(  L^{\alpha}\left(  \mathbb{T}\right)  ,Y_{1}\right)  $ is a
multiplier pair and $\left \{  M_{\psi}:\psi \in L^{\infty}\left(
\mathbb{T}\right)  \right \}  $ is a maximal abelian algebra of operators on
$L^{\alpha}\left(  \mathbb{T}\right)  ;$

\item $\left(  H^{\alpha}\left(  \mathbb{D}\right)  ,Y_{2}\right)  $ is a
multiplier pair and $\left \{  M_{\phi}:\phi \in L^{\infty}\left(
\mathbb{T}\right)  \right \}  $ is a maximal abelian algebra of operators on
$H^{\alpha}\left(  \mathbb{D}\right)  .$
\end{enumerate}
\end{corollary}

We now give a Banach space characterization of the condition $H^{\alpha
}\left(  \mathbb{T}\right)  =\mathcal{H}^{\alpha}\left(  \mathbb{T}\right)  $.
We need a characterization of $\alpha \left(  h\right)  $ when $h\in
H^{1}\left(  \mathbb{T}\right)  $. The key ingredient is based on the
following result that uses the Herglotz kernel \cite{Duren}.

\begin{lemma}
\label{herglotz}$\left \{  \left \vert h\right \vert :h\in H^{1}\left(
\mathbb{T}\right)  \right \}  =\left \{  \varphi \in L^{1}\left(  \mathbb{T}%
\right)  :\varphi \geq0\text{ and }\log \varphi \in L^{1}\left(  \mathbb{T}%
\right)  \right \}  $. In fact, if $\varphi \geq0$ and $\varphi,\log \varphi \in
L^{1}\left(  \mathbb{T}\right)  $, then
\[
h\left(  z\right)  =\exp \int_{\mathbb{T}}\frac{w+z}{w-z}\log \varphi \left(
w\right)  dm\left(  w\right)
\]
defines an outer function $h$ on $\mathbb{D}$ and $\left \vert h\right \vert
=\varphi$ on $\mathbb{T}$.
\end{lemma}

\bigskip

\begin{lemma}
\label{Hnorm}Suppose $f\in H^{1}\left(  \mathbb{T}\right)  $ and $\alpha
\in \mathcal{R}$. Then%
\[
\alpha \left(  f\right)  =\sup \left \{  \left \Vert fh\right \Vert _{1}:h\in
H^{\infty}\left(  \mathbb{T}\right)  ,\alpha^{\prime}\left(  h\right)
\leq1\right \}  .
\]

\end{lemma}

\begin{proof}
Let $S=\sup \left \{  \left \Vert fh\right \Vert _{1}:h\in H^{\infty}\left(
\mathbb{T}\right)  ,\alpha^{\prime}\left(  h\right)  \leq1\right \}  $. Suppose
$\varphi \geq0$ is a simple function and $\alpha^{\prime}\left(  \varphi
\right)  \leq1$. For each $\varepsilon>0$, $\varphi_{\varepsilon}%
=\frac{\varphi+\varepsilon}{1+\varepsilon}\geq0$ and $\varphi_{\varepsilon}\in
L^{1}\left(  \mathbb{T}\right)  $ and since $\varphi \in L^{1}\left(
\mathbb{T}\right)  $ and $\log \left(  \frac{\varepsilon}{1+\varepsilon
}\right)  \leq \log \left(  \varphi_{\varepsilon}\right)  \leq \varphi
+\varepsilon$, we see that there is an $h\in H^{1}\left(  \mathbb{T}\right)  $
such that $\left \vert h\right \vert =\varphi_{\varepsilon}$. Hence $h\in
H^{\infty}\left(  \mathbb{T}\right)  $ and $\alpha^{\prime}\left(  h\right)
=\alpha^{\prime}\left(  \varphi_{\varepsilon}\right)  \leq1$. Hence%
\[
S\geq \left \Vert f\varphi_{\varepsilon}\right \Vert _{1}%
\]
for every $\varepsilon>0.$ Letting $\varepsilon \rightarrow0^{+},$ we have
$S\geq \left \Vert f\varphi \right \Vert _{1}$. It follows that $S\geq
\alpha \left(  f\right)  .$ It is clear that $S\leq \alpha \left(  f\right)  $.
\end{proof}

\begin{theorem}
\label{multHalpha1}Suppose $\alpha$ is a rotationally symmetric norm and
$T:\mathcal{H}^{\alpha}\left(  \mathbb{T}\right)  \rightarrow H^{1}\left(
\mathbb{T}\right)  $ is a bounded linear operator such that, for every $h\in
H^{\infty}\left(  \mathbb{T}\right)  $ and every $g\in \mathcal{H}^{\alpha
}\left(  \mathbb{T}\right)  ,$%
\[
T\left(  hg\right)  =hT\left(  g\right)  .
\]
Then there is an $f\in \mathcal{H}^{\alpha^{\prime}}\left(  \mathbb{T}\right)
$ such that, for every $g\in \mathcal{H}^{\alpha}\left(  \mathbb{T}\right)  ,$%
\[
Tg=fg.
\]
Moreover, $\left \Vert T\right \Vert =\alpha^{\prime}\left(  f\right)  $. The
same conclusion holds when $\mathcal{H}^{\alpha}\left(  \mathbb{T}\right)  $
is replaced with $H^{\alpha}\left(  \mathbb{T}\right)  $.
\end{theorem}

\begin{proof}
Let $f=T\left(  1\right)  $. Suppose $h\in \mathcal{H}^{\alpha}\left(
\mathbb{T}\right)  $. Then $h\in H^{1}\left(  \mathbb{T}\right)  ,$ so there
are functions $u,v\in H^{\infty}\left(  \mathbb{T}\right)  ,$ with $v$ outer,
such that $h=u/v$. It follows that
\[
vT\left(  h\right)  =T\left(  u\right)  =uT\left(  1\right)  =uf,
\]
which implies $T\left(  h\right)  =fh$. The equality $\left \Vert T\right \Vert
=\alpha^{\prime}\left(  f\right)  $ follows from Lemma \ref{Hnorm}.
\end{proof}

\begin{corollary}
Suppose $\alpha$ is a rotationally symmetric norm with dual norm
$\alpha^{\prime}$, and suppose $f\in H^{1}\left(  \mathbb{T}\right)  $. Then
\[
f\cdot H^{\alpha}\left(  \mathbb{T}\right)  \subset H^{1}\left(
\mathbb{T}\right)  \Longleftrightarrow \text{ }f\in \mathcal{H}^{\alpha^{\prime
}}\left(  \mathbb{T}\right)  .
\]

\end{corollary}

\begin{proof}
It follows from the closed graph theorem that the map $T:H^{\alpha}\left(
\mathbb{T}\right)  \rightarrow H^{1}\left(  \mathbb{T}\right)  $ defined by
$T\left(  g\right)  =fg$ is bounded, and it follows from Theorem
\ref{multHalpha1} that $f\in \mathcal{H}^{\alpha^{\prime}}\left(
\mathbb{T}\right)  $.
\end{proof}

\bigskip

We now relate the strong continuity of $\alpha \in \mathcal{R}_{c}$ to the
condition $H^{\alpha}\left(  \mathbb{T}\right)  =\mathcal{H}^{\alpha}\left(
\mathbb{T}\right)  $. The proof of (3)$\Rightarrow$(2) below is an adaptation
of an argument shown to us by Eric Nordgren.

\begin{theorem}
Suppose $\alpha \in \mathcal{R}_{c}$. The following are equivalent:

\begin{enumerate}
\item $L^{\alpha}\left(  \mathbb{T}\right)  =\mathcal{L}^{\alpha}\left(
\mathbb{T}\right)  ;$

\item $H^{\alpha}\left(  \mathbb{T}\right)  =\mathcal{H}^{\alpha}\left(
\mathbb{T}\right)  ;$

\item $H^{\alpha}\left(  \mathbb{T}\right)  $ is weakly sequentially complete;

\item $L^{\alpha}\left(  \mathbb{T}\right)  $ is weakly sequentially complete.
\end{enumerate}
\end{theorem}

\begin{proof}
The statement (1)$\Leftrightarrow$(4) was proved in Theorem \ref{wc}.

To show (2)$\Leftrightarrow$(1), suppose $\varphi \in \mathcal{L}^{\alpha
}\left(  \mathbb{T}\right)  $ and $\varphi \geq0.$ Then, by Lemma
\ref{herglotz}, there is an $h\in H^{1}\left(  \mathbb{T}\right)  $ such that
$\left \vert h\right \vert =\varphi+1$. Hence $h\in \mathcal{H}^{\alpha}\left(
\mathbb{T}\right)  $ and $h\in H^{\alpha}\left(  \mathbb{T}\right)  $ if and
only if $\varphi=\left(  \varphi+1\right)  -1\in L^{\alpha}\left(
\mathbb{T}\right)  .$

The implication (4)$\Rightarrow$(3) follows from the fact that $H^{\alpha
}\left(  \mathbb{T}\right)  $ is a closed subspace of $L^{\alpha}\left(
\mathbb{T}\right)  $.

We now show (2)$\Rightarrow$(3). Suppose (2) holds and suppose $\left \{
f_{n}\right \}  $ is a weakly Cauchy sequence in $H^{\alpha}\left(
\mathbb{T}\right)  $. Then $\left \{  f_{n}\right \}  $ is a weakly Cauchy
sequence in $L^{\alpha}\left(  \mathbb{T}\right)  $. Following the proof of
(1)$\Rightarrow$(2) in Theorem \ref{wc}, there is an $f\in \mathcal{L}^{\alpha
}\left(  \mathbb{T}\right)  $ such that%
\[
\lim_{n\rightarrow \infty}\int_{\mathbb{T}}f_{n}hdm=\int_{\mathbb{T}}fhdm
\]
for every $h\in \mathcal{L}^{\alpha^{\prime}}\left(  \mathbb{T}\right)  $.
Thus, for every $k\geq0,$ we have%
\[
\int_{\mathbb{T}}fz^{k}dm=0.
\]
Hence $f\in H^{1}\left(  \mathbb{T}\right)  \cap \mathcal{L}^{\alpha}\left(
\mathbb{T}\right)  =\mathcal{H}^{\alpha}\left(  \mathbb{T}\right)  =H^{\alpha
}\left(  \mathbb{T}\right)  $. Since $L^{\alpha}\left(  \mathbb{T}\right)
^{\#}=\mathcal{L}^{\alpha^{\prime}}\left(  \mathbb{T}\right)  ,$ it follows
from the Hahn-Banach extension theorem that $f_{n}\rightarrow f$ weakly in
$H^{\alpha}\left(  \mathbb{T}\right)  $.

We finally show (3)$\Rightarrow$(2). Suppose $H^{\alpha}\left(  \mathbb{T}%
\right)  $ is weakly sequentially complete. Assume, via contradiction, that
there is an $f\in \mathcal{H}^{\alpha}\left(  \mathbb{T}\right)  $ such that
$f\notin H^{\alpha}\left(  \mathbb{T}\right)  $. Thus $\left \vert f\right \vert
\in \mathcal{L}^{\alpha}\left(  \mathbb{T}\right)  $ and $\left \vert
f\right \vert \notin L^{\alpha}\left(  \mathbb{T}\right)  $. Since the outer
part of $\varphi$ has the same modulus as $\varphi$ on $\mathbb{T}$, we can
assume that $\varphi$ is outer. For each positive integer $n$, let
$E_{n}=\left \{  z\in \mathbb{T}:\left \vert f\left(  z\right)  \right \vert
>n\right \}  $ and define
\[
\rho_{n}\left(  z\right)  =\left \{
\begin{array}
[c]{cc}%
1 & \text{if }z\notin E_{n}\\
1/\left \vert f\left(  z\right)  \right \vert  & \text{if }z\in E_{n}%
\end{array}
\right.  .
\]
Since $f$ is outer and $f\in H^{1}\left(  \mathbb{T}\right)  $, $\left \vert
f\right \vert $ and $\log \left \vert f\right \vert $ are in $L^{1}\left(
\mathbb{T}\right)  $. Since $\rho_{n}\leq1$ and $\left \vert \log \rho
_{n}\right \vert =\chi_{E_{n}}\log \left \vert \varphi \right \vert ,$ we see that
$\rho_{n}$ and $\log \rho_{n}$ are in $L^{1}\left(  \mathbb{T}\right)  $. Hence
there is an outer function $\varphi_{n}$ such that $\left \vert \varphi
_{n}\right \vert =\rho_{n}$. Since $\left \vert \varphi_{n}f\right \vert
=\left \vert \rho_{n}f\right \vert \leq n$, we see that $\varphi_{n}f\in
H^{\infty}\left(  \mathbb{T}\right)  $ for each $n\in \mathbb{N}$. Also%
\[
\left \Vert 1-\varphi_{n}\right \Vert _{2}^{2}=1+\left \Vert \varphi
_{n}\right \Vert _{2}^{2}-2\operatorname{Re}\varphi_{n}\left(  0\right)
\leq2\left(  1-\varphi_{n}\left(  0\right)  \right)  .
\]
However, by Lemma \ref{herglotz},
\[
\varphi_{n}\left(  0\right)  =\exp \int_{\mathbb{T}}\log \rho_{n}dm=\exp
\int_{\mathbb{T}}\chi_{E_{n}}\log \left \vert f\right \vert dm\rightarrow1,
\]
we see that $\left \Vert 1-\varphi_{n}\right \Vert _{2}\rightarrow1,$ and by
replacing $\varphi_{n}$ with a subsequence, if necessary, we can assume that
$\varphi_{n}\left(  z\right)  \rightarrow1$ a.e. $\left(  m\right)  $. Now
suppose $h\in L^{\alpha^{\prime}}\left(  \mathbb{T}\right)  .$ Then,
$\left \vert f\right \vert \left \vert h\right \vert \in L^{1}\left(
\mathbb{T}\right)  $, and, for every $n\in \mathbb{N}$, we have $\left \vert
\varphi_{n}fh\right \vert \leq \left \vert f\right \vert \left \vert h\right \vert
$. It follows from the dominated convergence theorem that%
\[
\lim_{n\rightarrow \infty}\int_{\mathbb{T}}\left(  \varphi_{n}f\right)
hdm=\int fhdm.
\]
Since $H^{\alpha}\left(  \mathbb{T}\right)  \subseteq L^{\alpha}\left(
\mathbb{T}\right)  $ and $L^{\alpha}\left(  \mathbb{T}\right)  ^{\#}%
=\mathcal{L}^{\alpha^{\prime}}\left(  \mathbb{T}\right)  ,$ we see that
$\left \{  \varphi_{n}f\right \}  $ is a weakly Cauchy sequence in $H^{\alpha
}\left(  \mathbb{T}\right)  .$ Hence there is an $F\in H^{\alpha}\left(
\mathbb{T}\right)  $ such that%
\[
\lim_{n\rightarrow \infty}\int_{\mathbb{T}}\varphi_{n}hdm=\int_{\mathbb{T}%
}Fhdm
\]
for every $h\in L^{\alpha^{\prime}}\left(  \mathbb{T}\right)  $. Hence $f=F\in
H^{\alpha}\left(  \mathbb{T}\right)  $, a contradiction.
\end{proof}

\section{Closed densely defined operators in a Multiplier pair}

Suppose $X=H^{\alpha}$ (on the unit disk) and $Y=N$ is the set of meromorphic
functions in the Nevanlinna class, i.e., functions of the form $\frac{f}{g}$
with $f,g\in H^{\infty}$ and $g$ not identically 0. Then $(X,Y)$ is a special
multiplier pair. The Smirnov class $N^{+}$ consists of all members of $N$
having a denominator that is an outer function. D. Hadwin, E. Nordgren and Z.
Liu \cite{HNL} have observed that the closed densely defined operators that
commute with the unilateral shift on $H^{p}$ are multiplications induced by
members of the Smirnov class. In this section, we will give a more general
result in $H^{\alpha}.$

\begin{lemma}
\label{lemma9.1} If $\phi \in N$ and $\phi \neq0,$ then there exist relatively
prime inner functions $u$ and $v$ and outer functions $a$ and $b$ satisfying
$|a|+|b|=1$ a.e. on the unit circle such that
\[
\phi=\frac{vb}{ua}.
\]

\end{lemma}

\begin{proof}
Recall that an outer function is positive at zero and is uniquely determined
by its absolute boundary values, which are necessarily absolutely log
integrable. Suppose $\phi$ is a nonzero function in $N$ and the inner-outer
factorization is applied to each of the numerator and denominator of $\phi$,
so
\[
\phi=\frac{vf_{1}}{uf_{2}},
\]
where $u$ and $v$ are relatively prime inner functions and $f_{1}$ and $f_{2}$
are outer functions in $H^{\infty}.$

Observe that on the unit circle $\mathbb{T},$
\[
\max \{|f_{1}|,|f_{2}|\} \leq|f_{1}|+|f_{2}|,
\]
It follows from $f_{1},f_{2}\in H^{\infty}$ and $\log(|f_{1}|+|f_{2}%
|)\leq|f_{1}|+|f_{2}|$ that%
\[
-\infty<\int_{\mathbb{T}}\log|f_{1}|dm\leq \int_{\mathbb{T}}\log(|f_{1}%
|+|f_{2}|)dm\leq \int_{\mathbb{T}}(|f_{1}|+|f_{2}|)dm<\infty,
\]
and therefore $(|f_{1}|+|f_{2}|)$ is log integrable. Thus there exists an
outer function $\psi$ in $H^{\infty}$ such that $|\psi|=|f_{1}|+|f_{2}|$ a.e.
on $\mathbb{T}.$ Put $a=\frac{f_{2}}{\psi}$ and $b=\frac{f_{1}}{\psi}$ and
observe that the definition of $\psi$ implies that $|a|+|b|=1$ a.e. on
$\mathbb{T}$ and $\phi=\frac{vb\psi}{ua\psi}=\frac{vb}{ua}.$

Now we need to show that $a$ and $b$ are outer functions. Since $a=\frac
{f_{2}}{\psi}, \psi=|f_{1}|+|f_{2}|,$ we see $|a|=\frac{|f_{2}|}{|\psi|}%
\leq1.$ The fact that $\max \{|f_{1}|, |f_{2}|\} \leq \psi$ shows that
\begin{align*}
\mid \int_{\mathbb{T}}\log|a| dm\mid &  =\mid \int_{\mathbb{T}}(\log|f_{2}%
|-\log|\psi|) dm\mid \\
&  \leq \mid \int_{\mathbb{T}}\log|f_{2}| dm\mid+\mid \int_{\mathbb{T}}\log|\psi|
dm\mid \\
&  <\infty,
\end{align*}
which means $a=\frac{f_{2}}{\psi}$ is outer. Similarly, $b=\frac{f_{1}}{\psi}$
is outer.
\end{proof}

\begin{corollary}
If $\phi \in N,$ where $\phi=\frac{vb}{ua}$ as in Lemma above, then the graph
\textrm{Graph}$(M_{\phi})$ of $M_{\phi}$ is the closed subset $\{(uag)\oplus
(vbg):g\in H^{\alpha}\}$ of $H^{\alpha}\oplus H^{\alpha}.$
\end{corollary}

\begin{proof}
If $g\in H^{\alpha},$ then $uag\in H^{\alpha}$ and $M_{\phi}uag=vbg\in
H^{\alpha},$ thus%
\[
\{(uag)\oplus(vbg):g\in H^{\alpha}\} \subset \mathrm{Graph}(M_{\phi}).
\]
For the opposite inclusion suppose both $f$ and $\phi f$ belong to $H^{\alpha
}.$ Then
\[
\frac{|f|}{|a|}=\frac{|a|+|b|}{|a|}|f|=|f|+|\phi||f|
\]
on $\mathbb{T},$ hence $\frac{f}{a}\in L^{\alpha}.$ Since $a$ is outer, it
follows that $\frac{f}{a}\in H^{\alpha}.$ Let $g_{1}=\frac{f}{a}.$ Then
$f=ag_{1}$ and $u\phi f=u\phi ag_{1}=vbg_{1}.$ Since $u$ and $v$ are
relatively prime and $b$ is outer, the last equation shows that $u$ is a
factor of $g_{1},$ and thus $g_{1}=ug$ for some $g\in H^{\alpha}.$ We have
shown that $f=uag$ and $\phi f=vbg,$ hence the required inclusion is established.
\end{proof}

\begin{theorem}
Suppose $G\subset H^{\alpha}\oplus H^{\alpha}$ is a graph that is invariant
under $M_{z}\oplus M_{z}.$ Then there is a meromorphic $\phi \in N$ such that
$G\subset \mathrm{Graph}(M_{\phi}).$ If the domain of $G$ is dense in
$H^{\alpha},$ then $\phi$ is in the Smirnov class. If, in addition, $G$ is
closed, then $\phi$ is in the Smirnov class and $G=\mathrm{Graph}(M_{\phi}).$
\end{theorem}

\begin{proof}
The first assertion follows from Corollary 2 in \cite{HNL}. Next suppose the
domain ${\mathcal{D}}(G)$ is dense in $H^{\alpha},\phi=\frac{vb}{ua}$ as in
Lemma \ref{lemma9.1}. $G\subset \mathrm{Graph}(M_{\phi})$ implies that
${\mathcal{D}}(G)$ is contained in the domain of $M_{\phi},$ which is
${\mathcal{D}}(G)\subset uaH^{\alpha}\subset H^{\alpha}.$ Thus $uaH^{\alpha}$
is dense in $H^{\alpha},$ and so it follows from $a$ is outer that
\[
H^{\alpha}={\overline{uaH^{\alpha}}}^{\alpha}={\overline{L_{u}(aH^{\alpha})}%
}^{\alpha}=L_{u}{\overline{(aH^{\alpha})}}^{\alpha}=L_{u}(H^{\alpha
})=uH^{\alpha},
\]
hence $u$ is a constant. Therefore $\phi \in N^{+}.$

Suppose $G$ is closed. Assume $H^{\alpha}\oplus H^{\alpha}$ is given the norm
defined by $\Vert f\oplus g\Vert=\alpha(|f|+|g|).$ If we define $V:H^{\alpha
}\longrightarrow H^{\alpha}\oplus H^{\alpha}$ by
\[
V(g)=uag\oplus vbg,
\]
then
\[
\Vert V(g)\Vert=\Vert uag\oplus vbg\Vert=\alpha(|uag|+|vbg|)=\alpha
((|a|+|b|)g)=\alpha(g).
\]
Thus $V$ is an isometry from $H^{\alpha}$ onto \textrm{Graph}$(M_{\phi}).$ Let
$M$ be the inverse image of $G$ under $V.$ Then $M$ is a closed subspace of
$H^{\alpha}$ and for $g\in M,$ we have
\[
VM_{z}g=V(zg)=uazg\oplus vbzg=(M_{z}\oplus M_{z})Vg\in G,
\]
hence $M\subset H^{\alpha}$ is invariant under $M_{z},$ it follows from
Theorem \ref{invariantthm} that $M=wH^{\alpha}$ for some inner function $w,$
thus
\[
G=V(M)=\{uawg\oplus vbwg:g\in H^{\alpha}\}=(M_{w}\oplus M_{w})\mathrm{Graph}%
(M_{\phi}).
\]
It follows that if the domain of $G$ is dense in $H^{\alpha},$ then $w$ is a
constant, hence $G=\mathrm{Graph}(M_{\phi}).$
\end{proof}

As a corollary to the proof we have the following.

\begin{corollary}
If $G\subset H^{\alpha}\oplus H^{\alpha}$ is a closed graph that is invariant
under $M_{z}\oplus M_{z},$ then there is a meromorphic function $\phi$ in the
Nevanlinna class and an inner function $w$ such that
\[
G=(M_{z}\oplus M_{z})\mathrm{Graph}(M_{\phi}).
\]

\end{corollary}

\section{A Corrected Result on von Neumann algebras}

Suppose $\mathcal{M}$ is a diffuse type $II_{1}$ von Neumann algebra acting on
a separable Hilbert space. This means that there is a faithful normal tracial
state $\tau:\mathcal{M}\rightarrow \mathbb{C}$ (i.e., $\tau \left(  ab\right)
=\tau \left(  ba\right)  $, and $\tau \left(  a^{\ast}a\right)  =0\Rightarrow
a=0$). Suppose $\alpha$ is a symmetric gauge norm on $L^{\infty}\left(
\mathbb{T}\right)  $. In \cite{FHNS} J. Fang, D. Hadwin, E. Nordgren and J.
Shen defined the Banach space $L^{\alpha}\left(  \mathcal{M},\tau \right)  $
which is the completion of $\mathcal{M}$ with respect to a norm induced by
$\alpha$, which they still denote by $\alpha$. They stated a theorem that
$L^{\alpha}\left(  \mathcal{M},\tau \right)  $ is a reflexive Banach space if
and only if $\alpha$ and $\alpha^{\prime}$ are both continuous. However, as
pointed out by Fyodor A. Sukochev in Mathematical Reviews: MR2417813
(2010a:46151), this theorem is not correct. In this section we state and prove
the corrected version. We will freely use terminology and notation from
\cite{FHNS}.

We first describe how $\alpha$ is defined on $\mathcal{M}$. Suppose
$\mathcal{A}$ is a masa (i.e., a maximal abelian C*-subalgebra) in
$\mathcal{M}$. A theorem of von Neumann says that there is a selfadjoint
element $a=a^{\ast}\in \mathcal{A}$ such that $\mathcal{A}=W^{\ast}\left(
a\right)  $ (the von Neumann algebra generated by $a$). If $Q_{s}%
=\chi_{\lbrack0,s)}\left(  a\right)  $ denotes the spectral projection of $a$
with respect to the set $[0,s)$, then $\mathcal{A}=W^{\ast}\left(  a\right)  $
is generated by the chain $\left \{  Q_{s}:s\in \lbrack0,\infty)\right \}  $ of
projections. This chain is contained in a maximal chain $\mathcal{C}$ of
projections in $\mathcal{M}$. Since $\tau$ is faithful, $\tau:\mathcal{C}%
\rightarrow \left[  0,1\right]  $ is an injective order-preserving map. Since
$\mathcal{M}$ has no minimal projections, $\tau \left(  \mathcal{C}\right)  $
must be $\left[  0,1\right]  $. Hence we can write $\mathcal{C}=\left \{
P_{t}:t\in \left[  0,1\right]  \right \}  $ where $\tau \left(  P_{t}\right)  =t$
for every $t\in \left[  0,1\right]  $. The map $P_{t}\mapsto \chi_{\left \{
e^{2\pi is}:s\in \lbrack0,t)\right \}  }$ extends to an isomorphism from
$\mathcal{A}=W^{\ast}\left(  \mathcal{C}\right)  $ onto $L^{\infty}\left(
\mathbb{T}\right)  $, such that if $b\in A$ is associated to the function
$f\in L^{\infty}\left(  \mathbb{T}\right)  $, then $\tau \left(  b\right)
=\int_{\mathbb{T}}fdm$. But $L^{\infty}\left(  \mathbb{T}\right)  =W^{\ast
}\left(  z\right)  $ where $z\left(  \lambda \right)  =\lambda$. If we let
$U\in \mathcal{A}$ be the element associated with $z$, we have that $U$ is a
unitary, $\mathcal{A}=W^{\ast}\left(  U\right)  $, and such that, for every
$h\in L^{\infty}\left(  \mathbb{T}\right)  ,$%
\[
\tau \left(  h\left(  U\right)  \right)  =\int_{\mathbb{T}}h\left(  z\right)
dm\left(  z\right)  .
\]
Such a unitary element $U$ in $\mathcal{M}$ is called a \emph{Haar unitary}
and is completely characterized by
\[
\tau \left(  U^{n}\right)  =0\text{ for }n\geq1.
\]
Hence, for every selfadjoint element $A\in \mathcal{M}$, $A$ is contained in a
masa in $\mathcal{M}$, so there is a Haar unitary $U\in \mathcal{M}$ and a
$\varphi \in L^{\infty}\left(  \mathbb{T}\right)  $ such that $A=\varphi \left(
U\right)  $. We define $\alpha \left(  A\right)  =\alpha \left(  \varphi \right)
$. More generally we define $\alpha \left(  T\right)  =\alpha \left(  \left \vert
T\right \vert \right)  $, where $\left \vert T\right \vert =\left(  T^{\ast
}T\right)  ^{1/2}$. The difficulty is showing that $\alpha$ is well-defined
(i.e., independent of $U$ and $\varphi$) and that $\alpha$ is a norm on
$\mathcal{M}$ (see \cite{FHNS}).

\begin{theorem}
Suppose $\alpha \in \mathcal{S}$ and $\mathcal{M}$ is a diffuse type $II_{1}$
von Neumann algebra with a faithful tracial state $\tau$ acting on a separable
Hilbert space. Then $L^{\alpha}\left(  \mathcal{M},\tau \right)  $ is a
reflexive Banach space if and only if $\alpha$ and $\alpha^{\prime}$ are both
strongly continuous.
\end{theorem}

\begin{proof}
We know that $\mathcal{M}$ contains a Haar unitary $U$ and that $W^{\ast
}\left(  U\right)  $ is a copy of $L^{\infty}\left(  \mathbb{T}\right)  $ so
that $\tau \left(  f\left(  U\right)  \right)  =\int_{\mathbb{T}}fdm$ and
$\alpha \left(  f\left(  U\right)  \right)  =\alpha \left(  f\right)  $ for
every $f\in L^{\infty}\left(  \mathbb{T}\right)  .$ Hence $L^{\alpha}\left(
\mathbb{T}\right)  $ is isometrically isomorphic to a closed subspace of
$L^{\alpha}\left(  \mathcal{M},\tau \right)  $. Hence if $L^{\alpha}\left(
\mathcal{M},\tau \right)  $ is reflexive, then so is $L^{\alpha}\left(
\mathbb{T}\right)  $, which by part $(3)$ of Theorem \ref{reflexive space},
implies $\alpha$ and $\alpha^{\prime}$ are both strongly continuous.

Now assume $\alpha$ and $\alpha^{\prime}$ are both strongly continuous and
suppose $\varphi:L^{\alpha}\left(  \mathcal{M},\tau \right)  \rightarrow
\mathbb{C}$ is a continuous linear functional. It was shown in \cite{HN2} that
on the closed unit ball of $\mathcal{M}$, the $\alpha$-topology and the strong
operator topology coincide. It follows that if $\left \{  P_{n}\right \}  $ is
an orthogonal sequence of projections in $\mathcal{M}$, then
\[
\varphi \left(  \sum_{n=1}^{\infty}P_{n}\right)  =\sum_{n=1}^{\infty}%
\varphi \left(  P_{n}\right)  ,
\]
\newline which implies $\varphi:\mathcal{M}\rightarrow \mathbb{C}$ is
weak*-continuous. Hence there is an $A\in L^{1}\left(  \mathcal{M}%
,\tau \right)  $ such that, for every $T\in \mathcal{M}$%
\[
\varphi \left(  T\right)  =\tau \left(  AT\right)  .
\]
This, and the fact that $\mathcal{M}$ is dense in $L^{\alpha}\left(
\mathcal{M},\tau \right)  $, implies
\[
\alpha^{\prime}\left(  A\right)  =\sup \left \{  \tau \left(  AT\right)
:T\in \mathcal{M},\text{ }\alpha \left(  T\right)  \leq m\right \}  =\left \Vert
\varphi \right \Vert .
\]
This is the hard part of the proof that $L^{\alpha}\left(  \mathcal{M}%
,\tau \right)  ^{\#}=L^{\alpha^{\prime}}\left(  \mathcal{M},\tau \right)  $.
Since $\alpha^{\prime \prime}=\alpha$ and $\alpha^{\prime}$ is strongly
continuous, we see that%

\[
L^{\alpha}\left(  \mathcal{M},\tau \right)  ^{\# \#}=L^{\alpha^{\prime}}\left(
\mathcal{M},\tau \right)  ^{\#}=L^{\alpha}\left(  \mathcal{M},\tau \right)  .
\]

\end{proof}

\begin{remark}
The proof of the preceding theorem shows that if $\alpha$ is continuous and
$\alpha^{\prime}$ is strongly continuous, then $L^{\alpha}\left(
\mathcal{M},\tau \right)  ^{\#}=L^{\alpha^{\prime}}\left(  \mathcal{M}%
,\tau \right)  $.
\end{remark}

We conclude this section by noting that the analogues of Theorems
\ref{multalpha1} and \ref{wc} hold in the von Neumann algebra case for
symmetric gauge norms. The proofs are easy adaptations and we omit them here.
We do need, however, the fact that $L^{1}\left(  \mathcal{M},\tau \right)
^{\#}=\mathcal{M}$, which, by a theorem of C. Akemann \cite{Ak}, implies that
$L^{1}\left(  \mathcal{M},\tau \right)  $ is weakly sequentially complete. A
key ingredient is that there is a completion $\mathcal{Y}$ of $\mathcal{M}$
\emph{in measure} (see \cite{Nelson}), which is an algebra containing
$L^{1}\left(  \mathcal{M},\tau \right)  $ such that, for every $h\in
\mathcal{Y}$ there exist $u_{1},v_{1},u_{2},v_{2}\in \mathcal{M}$ with
$v_{1},v_{2}$ invertible in $\mathcal{Y}$ (but maybe not in $\mathcal{M}$)
such that $h=v_{1}^{-1}u_{1}=u_{2}v_{2}^{-1}$. Note that since $\mathcal{M}$
may not be commutative, there is a difference between left $\mathcal{M}%
$-module homomorphisms and right $\mathcal{M}$-module homomorphisms, which is
reflected in parts (1) and (2) below.

\begin{theorem}
Suppose $\mathcal{M}$ is a $II_{1}$ von Neumann algebra with a faithful normal
tracial state $\tau$ and suppose $\alpha \in \mathcal{S}_{c}$. Then

\begin{enumerate}
\item If $T:\mathcal{L}^{\alpha^{\prime}}\left(  \mathcal{M},\tau \right)
\rightarrow L^{1}\left(  \mathcal{M},\tau \right)  $ is a bounded linear map
such that $T\left(  hg\right)  =hT\left(  g\right)  $ whenever $h\in
\mathcal{M}$ and $g\in \mathcal{L}^{\alpha^{\prime}}\left(  \mathcal{M}%
,\tau \right)  $, then there is an $f\in \mathcal{L}^{\alpha}\left(
\mathcal{M},\tau \right)  $ such that $\alpha \left(  f\right)  =\left \Vert
T\right \Vert $ and $T\left(  h\right)  =hf$ for every $h\in \mathcal{L}%
^{\alpha^{\prime}}\left(  \mathcal{M},\tau \right)  $;

\item If $T:\mathcal{L}^{\alpha^{\prime}}\left(  \mathcal{M},\tau \right)
\rightarrow L^{1}\left(  \mathcal{M},\tau \right)  $ is a bounded linear map
such that $T\left(  gh\right)  =T\left(  g\right)  h$ whenever $h\in
\mathcal{M}$ and $g\in \mathcal{L}^{\alpha^{\prime}}\left(  \mathcal{M}%
,\tau \right)  $, then there is an $f\in \mathcal{L}^{\alpha}\left(
\mathcal{M},\tau \right)  $ such that $\alpha \left(  f\right)  =\left \Vert
T\right \Vert $ and $T\left(  h\right)  =fh$ for every $h\in \mathcal{L}%
^{\alpha^{\prime}}\left(  \mathcal{M},\tau \right)  $;

\item The following are equivalent:

\begin{enumerate}
\item $\mathcal{L}^{\alpha}\left(  \mathcal{M},\tau \right)  =L^{\alpha}\left(
\mathcal{M},\tau \right)  ;$

\item $\alpha$ is strongly continuous;

\item $L^{\alpha}\left(  \mathcal{M},\tau \right)  $ is weakly sequentially complete.
\end{enumerate}
\end{enumerate}
\end{theorem}

\end{document}